\documentclass[a4paper]{amsart}

\usepackage[margin=0.85in]{geometry}
\usepackage{amssymb,amsmath}
\usepackage{tikz-cd}
\usepackage{tikz}

\usepackage{hyperref}
\usepackage[all]{xy}
\usepackage{enumitem}
\usepackage{mathtools}
\usepackage{multirow}
\usepackage{floatrow}
\usepackage{graphicx}
\usepackage{color}
\usepackage{pdflscape}
\usepackage{xfrac}
\usepackage{array}
\usepackage{centernot}
\usepackage{mathdots}

\newdir{ >}{{}*!/-6pt/@{>}}

\usetikzlibrary{patterns}

\renewcommand{\arraystretch}{1.3}
\theoremstyle{plain}
\newtheorem{theorem}{Theorem}[section]
\newtheorem{lemma}[theorem]{Lemma}
\newtheorem{sublemma}[theorem]{Sublemma}
\newtheorem{proposition}[theorem]{Proposition}
\newtheorem{corollary}[theorem]{Corollary}
\theoremstyle{definition}
\newtheorem{definition}[theorem]{Definition}
\newtheorem{example}[theorem]{Example}
\newtheorem{remark}[theorem]{Remark}

\numberwithin{equation}{section}

\newcommand{\Ch}{\textnormal{Ch}}
\newcommand{\M}{\mathbf{M}}
\newcommand{\V}{\mathbf{V}}
\newcommand{\can}{\textnormal{can}}
\newcommand{\Pic}{\mathrm{Pic}}
\newcommand{\GW}{\mathrm{GW}}
\newcommand{\rank}{\mathrm{rank}\,\relax}
\newcommand{\Gr}{\mathrm{Gr}}
\newcommand{\LG}{\mathrm{LG}}
\newcommand{\Bl}{\mathrm{B}}
\newcommand{\SO}{\mathcal{O}}
\newcommand{\nhd}{\centernot{\subset}}

\newcommand{\LF}{\mathrm{LF}}
\newcommand{\LFo}{\mathrm{L}\mathring{\mathrm{F}}}

\newcommand{\LFb}{\mathbb{LF}}

\newcommand{\drawgray}[3]{%
	\draw[fill, opacity=0.2] (#1,#2) rectangle ++(1,-#3);
}

\newcommand{\drawalmostyoung}[4]{%
	\foreach \llambda [count=\i] in {#4}{
		\drawgray{#1 * #3 + \i-1}{#2}{\llambda}}
}

\newcounter{mydiag}
\newcommand{\diagram}{\refstepcounter{mydiag} {\rm \fbox{$\scriptstyle \themydiag$}}}
\newcommand{\diag}[1]{{\rm \fbox{$\scriptstyle #1$}}} 

\newcounter{myeq}

\DeclareFontFamily{U}{mathx}{}
\DeclareFontShape{U}{mathx}{m}{n}{<-> mathx10}{}
\DeclareSymbolFont{mathx}{U}{mathx}{m}{n}
\DeclareMathAccent{\widehat}{0}{mathx}{"70}
\DeclareMathAccent{\widecheck}{0}{mathx}{"71}

\makeatletter
\newcommand{\btg}{\mathpalette\btg@\relax}
\newcommand{\btg@}[2]{%
	\begingroup
	\sbox\z@{$\m@th#1\triangle$}%
	\makebox[\wd\z@]{%
		\raisebox{0.04\height}{%
			\resizebox{1.1\wd\z@}{0.96\ht\z@}{%
				$\m@th#1\blacktriangle$%
			}%
		}%
	}%
	\endgroup
}
\makeatother

\author{Tao Huang}
\address{Tao Huang, School of Mathematics, Sun Yat-sen University, No. 135 Xingang Xi Road, Guangzhou, 510275, China}
\email{huangt233@mail2.sysu.edu.cn}
\author{Heng Xie}
\address{Heng Xie, School of Mathematics, Sun Yat-sen University, No. 135 Xingang Xi Road, Guangzhou, 510275, China}
\email{xieh59@mail.sysu.edu.cn}
\title[Hermitian $K$-theory of Lagrangian Grassmannians]{Hermitian K-theory of Lagrangian Grassmannians via  reducible Gorenstein models}

\begin{document}
%%%%%%%%%%%%%%%%%% Abstract %%%%%%%%%%%%%%%%%%
\begin{abstract}
    We construct a family of moduli spaces, called generalized Lagrangian flag schemes, that are reducible Gorenstein (hence singular), and that admit well-behaved pushforward and pullback operations in Hermitian $K$-theory. These schemes arise naturally in our computations. Using them, we prove that the Hermitian $K$-theory of a Lagrangian Grassmannian over a regular base splits as a direct sum of copies of the base's (Hermitian) $K$-theory, indexed by certain shifted Young diagrams. The isomorphism is realized via pullback to each generalized Lagrangian flag scheme followed by pushforward to the Lagrangian Grassmannian. This yields an unusual example in which both the base and the target are regular schemes,  while the intermediate reducible Gorenstein models remain sufficient to allow explicit computations in Hermitian K-theory of regular schemes.
\end{abstract}
\maketitle
%%%%%%%%%%%%%%%%%% Introduction %%%%%%%%%%%%%%%%%%
\section{Introduction}
In Quillen's seminal work \cite{quillen1973higher}, the computation of the algebraic $K$-theory of projective bundles illustrates how geometric properties of projective bundles can be applied to compute abstract cohomology. This basic example already demonstrates that computing the algebraic $K$-theory of schemes drives the exploration of useful properties of geometric objects. Hermitian $K$-theory, introduced around the same time as Quillen's algebraic $K$-theory by Bass and Karoubi \cite{karoubi1974periodicite}, serves as a quadratic refinement of algebraic $K$-theory. In practice, Hermitian $K$-theory often yields finer invariants than algebraic $K$-theory, indicating that Hermitian $K$-theory is generally more difficult to compute. In our work \cite{huang2023the}, we constructed pushforward and pullback maps in Hermitian $K$-theory, and established several tools, including base-change, excess intersection, and projection formulas, based on the framework of Schlichting \cite{schlichting2010mayer} and \cite{schlichting2017hermitian}. Using these tools, we showed that pushforward and pullback through generalized flag varieties induce an isomorphism between the Hermitian $K$-theory of Grassmannians and a direct sum of copies of the $K$-theory and Hermitian $K$-theory of the base scheme (cf. \cite{huang2023the}). The generalized flag varieties involved are certain Schubert varieties in flag varieties, whose useful properties have been studied in detail in \cite{balmer2012witt} and further analyzed in \cite{huang2023the}.

In this paper, we explain how pushforward and pullback in \cite{huang2023the} can help to compute Hermitian $K$-theory of Lagrangian Grassmannians, and to expose so-called generalized Lagrangian flag schemes. These schemes arise naturally from our induction theorem (cf. Theorem \ref{thm:lg_induc}). Let us introduce them in the case where $S = \mathrm{Spec}(k)$, with $k$ a field of characteristic different from two. Assume that $(V,\beta)$ is a non-degenerate symplectic form with a complete flag
$
    0=V_{0} \subset V_{1} \subset V_{2} \subset \ldots \subset V_{n-1} \subset V_{n} = V_{n}^{\bot} \subset V_{n-1}^{\bot} \subset \ldots \subset V_{1}^{\bot} \subset V_{0}^{\bot} = V ,
$
such that $\rank(V_j) = j$. Let $\mathbf{d} = (d_0,d_1,d_2,\ldots, d_k)$, $\mathbf{e} = (e_0,e_1,\ldots,e_{k-1})$ and $\mathbf{t} =(t_0,t_1,\ldots,t_{k-1}) $ be fixed sequences of non-negative non-decreasing integers. The \textit{generalized Lagrangian flag scheme} $\LF_{\mathbf{d}}(\mathbf{e})_{\mathbf{t}}$ is the scheme parameterizing the following strata:
\[
    \setlength\arraycolsep{1pt} \renewcommand{\arraystretch}{1}  \left\{  \begin{matrix}
                    &         &                       &         & V_{e_{0}}       & \subset & P_{n-t_0}^{(0)}    & \subset & L_{n}^{(0)}      & \subset & V_{d_0}^{\bot} \\
                    &         &                       &         &                 &         & \cap                                                                       \\
                    &         & V_{e_1}               & \subset & P_{n-t_1}^{(1)} & \subset & L_{n}^{(1)}        & \subset & V_{d_{1}}^{\bot}                            \\
                    &         &                       &         & \vdots                                                                                                 \\
        V_{e_{k-1}} & \subset & P_{n-t_{k-1}}^{(k-1)} & \subset & L_{n}^{(k-1)}   & \subset & V_{d_{k-1}}^{\bot}                                                         \\
                    &         & \cap                                                                                                                                     \\
                    &         & L_{n}^{(k)}           & \subset & V_{d_k}^{\bot}
    \end{matrix}\right\}
\]
where $e_i \leq d_i \leq n, e_i \leq d_{i+1}$ and $e_i \leq n-t_i$ for all $i$, and where $L_n^{(j)}$ are Lagrangian subspaces of rank $n$ inside $V$ and $P_{n-t_i}^{(i)}$ are subspaces of rank $n-t_i$ inside both $L_n^{(i)}$ and $L_n^{(i+1)}$. If the length $k=0$ and $d_0=0$, then $\LF_{\mathbf{d}}(\mathbf{e})_{\mathbf{t}}$ is the usual Lagrangian Grassmannian $\LG(n,V)$.

Computing Hermitian $K$-theory leads us to show that \textit{the scheme $\LF_{\mathbf{d}}(\mathbf{e})_{\mathbf{t}}$ is Gorenstein if $\mathbf{d}^{< k} - \mathbf{e} \leq 1$} for any $\mathbf{t}$, where $\mathbf{d}^{< k}$ is the $k$-tuple obtained by removing $d_k$ from $\mathbf{d}$. It is not hard to see that $\LF_{\mathbf{d}}(\mathbf{e})_{\mathbf{t}}$ is regular if $\mathbf{d}^{< k} = \mathbf{e}$, which should be well-known to experts. To our best knowledge, the case $ \mathbf{e} \neq \mathbf{d}^{< k}$ has not been well-studied in the existing literature. To make sure that $\LF_{\mathbf{d}}(\mathbf{e})_{\mathbf{t}}$ can produce meaningful pushforward on Hermitian $K$-theory, we compute its canonical line bundle and dimension. All these properties are explored in Section \ref{sec:LagrangianGrass}.

\begin{theorem}[Theorem \ref{thm:main-theo}]\label{thm:main-theo-intro}
    Let $S$ be a regular scheme with $\frac{1}{2} \in \SO_S$. Suppose that $V$ is a trivial vector bundle over $S$ of rank $2n$ with a hyperbolic symplectic form $\beta$ (in fact, we deal more generally with a symplectic bundle that admits a complete flag).\	Let $\Delta_n$ be the determinant line bundle of the tautological bundle on the Lagrangian Grassmannian $\LG(n)$ associated to $V$.
    \begin{enumerate}[leftmargin=20pt,label={\rm (\alph*)},ref=(\alph*)]
        \item Suppose that $n$ is even, then the maps
              \[   (\Theta^l, \Omega^l): \bigoplus_{\Lambda \in \mathfrak{E}_n \backslash \mathfrak{A}_n^r } K(S) \oplus \bigoplus_{\Lambda \in \mathfrak{A}_{n}^r} \GW^{[m-|\Lambda|]}(S) \xrightarrow{ (\sum \mu_\Lambda^1, \sum\xi^1_\Lambda)} \GW^{[m]}(\LG(n), \Delta_n)
              \]
              \[
                  (\Theta^0, \Omega^0): \bigoplus_{\Lambda \in \mathfrak{E}_n \backslash \mathfrak{A}_n^c } K(S) \oplus \bigoplus_{\Lambda \in \mathfrak{A}_{n}^c} \GW^{[m-|\Lambda|]}(S) \xrightarrow{ (\sum \mu_\Lambda^0, \sum \xi_\Lambda^0)} \GW^{[m]}(\LG(n))
              \]
              are stable equivalences in the stable homotopy category of spectra.
        \item Suppose that $n$ is odd, then the maps
              \[
                  (\Theta^0, \Omega^0): \bigoplus_{\Lambda \in \mathfrak{E}_n \backslash \mathfrak{A}_n} K(S) \oplus \bigoplus_{\Lambda \in \mathfrak{A}_n} \GW^{[m-|\Lambda|]}(S) \xrightarrow{ (\sum \mu_\Lambda^0, \sum \xi_\Lambda^0)} \GW^{[m]}(\LG(n))
              \]
              \[
                  \Theta^1: \bigoplus_{\Lambda \in \mathfrak{E}_n } K(S) \xrightarrow{ \sum \mu_\Lambda^1} \GW^{[m]}(\LG(n), \Delta_n)
              \]
              are stable equivalences in the stable homotopy category of spectra.
    \end{enumerate}
    Here, $\mathfrak{A}_n$ (resp. $\mathfrak{E}_n$) is the set of almost even (resp. $K$-even) shifted Young diagrams in the shifted $n$-frame. Moreover, $\mathfrak{A}_n^r$ (resp.  $\mathfrak{A}_n^c$) is the subset of $\mathfrak{A}_n$ consisting of those almost even shifted Young diagrams such that the topmost row is full (resp. the rightmost column is empty). The map $\xi^i_\Lambda$ is defined by first pulling back from Hermitian $K$-theory of the base scheme $S$ to that of generalized Lagrangian flag scheme associated to $\Lambda$ and then pushing forward to that of $\LG(n,V)$. The map $\mu^i_\Lambda$ is defined analogously to $\xi^i_\Lambda$, but using $K$-theory and additionally composed with the hyperbolic functor.
\end{theorem}

% It is not harmful to understand this result when $S= \mathrm{Spec}(k)$ or even $\mathrm{Spec}(\mathbb{C})$.
To understand the additive basis in Theorem \ref{thm:main-theo-intro}, we recall the definition of almost even shifted Young diagrams (cf. \cite{martirosian2021witt}) and introduce $K$-even Young diagrams. A \textit{strict partition} is a sequence of positive integers $\Theta:=(\theta_i, \ldots ,\theta_2, \theta_1)$ such that $\theta_{i} > \theta_{i-1} > \ldots > \theta_2 > \theta_1 > 0$.\ Let $n$ be a positive integer. A \textit{shifted $n$-frame} is the strict partition
$ \digamma_n  :=  (n,n-1,n-2,\ldots,2,1).$
A \textit{shifted Young diagram} $\Lambda$ in $\digamma_{n}$ is an $n$-tuple of non-negative integers
$\Lambda:=(\Lambda_n, \Lambda_{n-1}, \ldots, \Lambda_2, \Lambda_1)$
such that
$n \geq \Lambda_{n} \geq \Lambda_{n-1} \geq \ldots \geq \Lambda_2 \geq \Lambda_1 \geq 0$
with $\Lambda_i \leq i$, and $\Lambda_{i+1} > \Lambda_{i}$ whenever $\Lambda_{i} \neq 0$.
The \textit{boundary} $b_{\Lambda}$ of a shifted Young diagram $\Lambda$ in $\digamma_n$ is a lattice path that goes from the top-right to the bottom-left corner of $\digamma_n$ so that the Young diagram $\Lambda$ lies precisely in the upper-left region determined by the path. The boundary $b_{\Lambda}$ consists of the \textit{segments} $s_1, s_2, \ldots, s_{l_\Lambda}$, ordered from the top-right to the bottom-left. Write $|s_i|$ for the lattice length of the segment $s_i$. We always assume that $s_1$ is vertical, even if its length $|s_1|$ is zero, and we assume that $1 \leq |s_i| $ for $ 2\leq i \leq l_\Lambda$. Denote the terminal point of the segment $s_{t}$ by $a_{t+1}$ ($1 \leq t \leq l_{\Lambda}$) and let the initial point of the segment $s_1$ be the \textit{origin} $O$ of $\Lambda$. 
See Figure \ref{fig:shifted-young-boundary-vertex} for an explanation. Note that the length of the boundary ($\sum_{i} |s_i|$) is always equal to $n$. Note also that $s_i$ is always vertical (resp. horizontal) if $i$ is odd (resp. even).

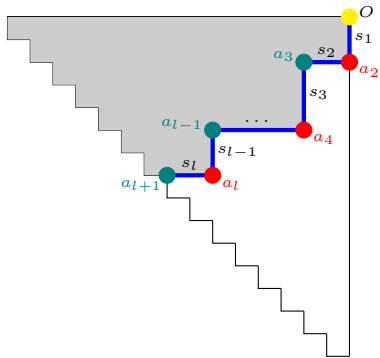
\begin{figure}[!ht]

    \begin{center}
        \begin{tikzpicture}[scale=0.3]
            \def\n{15}
            \foreach \i in {0,...,14} {
                    \draw  (\i+1,\n-1-\i) -- (\i,\n-1-\i) -- (\i,\n-\i) ;
                }
            \foreach \i in {0,...,14} {
                    \draw (\n,\n) -- (0,\n);
                }
            \draw (\n,0) -- (\n,\n);

            \path[fill=black!20] (15,15) -- (15, 13) -- (13,13) -- (13,10) -- (9,10) -- (9,8) -- (6,8)  -- (6,9) -- (5,9) -- (5,10) -- (4,10) -- (4,11) -- (3,11) -- (3,12) -- (2,12) -- (2,13) -- (1,13) -- (1,14) -- (0,14) -- (0,15) -- (15,15);

            \draw[blue, ultra thick] (15,15) -- (15,13);
            \draw[blue, ultra thick] (15,13) -- (13,13) -- (13,10) -- (9,10) -- (9,8) -- (7,8);

            \filldraw [yellow] (15,15) circle (10pt);
            \filldraw [red] (15,13) circle (10pt);
            \filldraw [teal] (13,13) circle (10pt);
            \filldraw [red] (13,10) circle (10pt);
            \filldraw [teal] (9,10) circle (10pt);
            \filldraw [red] (9,8) circle (10pt);
            \filldraw [teal] (7,8) circle (10pt);

            \path (15,15)
            node[text=black,anchor=base west] {\tiny $O$};
            \path (15-.2,14)
            node[text=black,anchor=base west] {\tiny $s_1$};
            \path (15,13-.5)
            node[text=red,anchor=base west] {\tiny $a_2$};
            \path (14,13-.2)
            node[text=black,anchor=south] {\tiny $s_2$};
            \path (13,13+.2)
            node[text=teal,anchor=base east] {\tiny $a_3$};
            \path (13-.2,11.5)
            node[text=black,anchor=base west] {\tiny $s_3$};
            \path (13,10-.5)
            node[text=red,anchor=base west] {\tiny $a_4$};
            \path (11,10-.2)
            node[text=black,anchor=south] {\tiny $\cdots$};
            \path (9,10+.2)
            node[text=teal,anchor=base east] {\tiny $a_{l-1}$};
            \path (9-.2,9)
            node[text=black,anchor=base west] {\tiny $s_{l-1}$};
            \path (9,8-.5)
            node[text=red,anchor=base west] {\tiny $a_{l}$};
            \path (8,8-.2)
            node[text=black,anchor=south] {\tiny $s_l$};
            \path (7+.2,8-.5)
            node[text=teal,anchor=base east] {\tiny $a_{l+1}$};

        \end{tikzpicture}
    \end{center}
    \caption{A shifted Young diagram. Its boundary, colored in blue, runs from the top-right to the bottom-left corner. The red (resp. teal) bullets are convex (resp. concave) corners. }
    \label{fig:shifted-young-boundary-vertex}
\end{figure}

\begin{definition}\label{def:shifted_young_index}
    The \textit{index} $w_{\Lambda}$ of a shifted Young diagram $\Lambda$ in $\digamma_n$ is the smallest integer $t$ with
    \[
        1\leq t \leq l_\Lambda, \quad |Oa_{t+1}|\neq 0, \quad \text{and} \quad |Oa_{t+1}| \equiv n \pmod 2,
    \]
    where $|Oa_t|$ is the lattice length along $b_\Lambda$ from $O$ to $a_t$ (i.e. $|Oa_t|=\sum_{i=1}^{t-1}|s_i|$).
    The index is well-defined, since $|Oa_{l_\Lambda+1}| = n$.
\end{definition}
\begin{definition}
    A shifted Young diagram $\Lambda$ in $\digamma_{n}$ is called \textit{almost even} if $w_{\Lambda}=l_\Lambda$. A shifted Young diagram $\Lambda$ in $\digamma_{n}$ is called \textit{$K$-even} if $w_{\Lambda}$ is even.
\end{definition}
In view of Figure \ref{fig:shifted-young-boundary-vertex}, a shifted Young diagram is almost even if all \textit{inner} segments have even length except the last inner segment (if it exists), which has odd length. If $n$ is odd, a shifted Young diagram $\Lambda$ is $K$-even if $|s_1|, \ldots, |s_{w_\Lambda -1}|$ are even and $|s_{w_\Lambda}|$ is odd and horizontal. There is a natural way to associate either an almost even or a $K$-even shifted Young diagram to a generalized Lagrangian scheme (cf. Section \ref{sec:K_theory_of_lg} and \ref{sec:proof_of_the_main_theorem}), and all the relevant generalized Lagrangian schemes satisfy the property $\mathbf{d}^{< k} - \mathbf{e} \leq 1$, and hence are Gorenstein.

% The underlying maps $\xi^i$ are defined by first pulling back from Hermitian $K$-theory of the base to that of generalized Lagrangian flag scheme, and then pushforward to that of the Lagrangian Grassmannian. The underlying maps $\mu^i$ are defined by pulling back from $K$-theory of the base to that of generalized Lagrangian flag scheme and then pushforward to $K$-theory of the Lagrangian Grassmannian, and finally composing with the hyperbolic map. 

If $\mathbf{d}^{< k} = \mathbf{e}$, the scheme $\LF_{\mathbf{d}}(\mathbf{e})_{\mathbf{t}}$ serves as a model for resolution of singularities of the Schubert varieties in the Lagrangian Grassmannians. One may wonder why we do not use this much simpler variety to define the additive basis in Hermitian $K$-theory. The condition $\mathbf{d}^{< k} = \mathbf{e}$ is indeed sufficient for computing $K$-theory of Lagrangian Grassmannian, cf. Section 3. However, when $\mathbf{d}^{< k} = \mathbf{e}$ the scheme $\LF_{\mathbf{d}}(\mathbf{e})_{\mathbf{t}}$ does not have the correct twist to pushforward in Hermitian $K$-theory, as the canonical line bundle is generally not trivial in the Picard group modulo two. Therefore, the additive basis in Hermitian $K$-theory becomes meaningless if we let $\mathbf{d}^{< k} = \mathbf{e}$. To remedy this problem, we introduce a \textit{padding} technique and show that the twist can be aligned if we make an appropriate choice of $\mathbf{d},\mathbf{t},\mathbf{e}$ by loosening the condition $ d_i -e_i = 0$ to $ d_i -e_i \leq 1$.

Since our maps factor through schemes which are not regular (or even not irreducible), we have to check carefully that our tools in \cite{huang2023the} are still valid, and the \textit{alignment} (cf. \cite{balmer2012bases}) do not cause any trouble, which turned out to be technical and have been written down in Appendix \ref{sec:base_change_gorenstein} and \ref{sec:pushforward}.

Our work is partly motivated by the two-step blow-up strategy introduced by Martirosian in his thesis, where he computed the Witt groups of Lagrangian Grassmannians \cite{martirosian2021witt}.  Martirosian also conjectured a geometric model for a possible additive basis of these Witt groups, obtained via an alternative construction inside the Bott–Samelson resolution. However, as noted in \cite[Remark 3.4.6]{martirosian2021witt}, a full proof was not supplied and no published work has carried out a complete realization of this idea. In the present paper, we introduce a different model, namely the generalized Lagrangian flag scheme, and prove that it induces an additive basis for the Hermitian $K$-theory of Lagrangian Grassmannians, which encompasses the result that Martirosian originally sought to establish. The details of the two-step blow-up construction, in the context of Hermitian $K$-theory, are provided in Appendix \ref{sec:two_step_blow_up}.

In fact, such generalized Lagrangian flag schemes emerged while we were computing the Hermitian $K$-theory of $\LG(2)$ (cf. Theorem \ref{thm:lg_induc}(b)), where an explicit splitting of the localization sequence must be required to establish the underlying isomorphism. This situation has no analogue in Witt-group calculations, because one term in the localization sequence vanishes: $W^*(\LG(1), \Delta) =W^*(\mathbb{P}^1, \mathcal{O}(1)) = 0$ (cf. \S\ref{sec:Witt-Q3}). Therefore, a purely Witt-theoretic approach would be unlikely to expose these flag schemes unless there are other intriguing reasons. This example also shows that Hermitian $K$-theory reveals structures invisible in Witt groups and can lead to new insights.

\subsection{Some consequences}
We now record some immediate consequences from our work.

\subsubsection{Compare to $KO$-theory}
$KO$-theory of complex Lagrangian Grassmannians in topology has been computed by Kono-Hara \cite{kono1992ko} for the case of trivial line bundles and \cite{zibrowius2011witt} for the case of twisted line bundles by means of Atiyah-Hirzebruch spectral sequence. From the work of Zibrowius  \cite{zibrowius2011witt}, we know that the comparison map
$$gw^m: \GW^{[m]}(\LG(n),L ) \to \mathrm{KO}^{2m}(\LG(n),L )$$
induces an isomorphism. Our proof under the special case $S = \mathrm{Spec}(\mathbb{C})$ recovers the computation of \cite{kono1992ko}  and \cite{zibrowius2011witt} provides an alternative way for the computation. It seems to us that it is difficult to write down geometrically the underlying map of the isomorphism in $KO$-theory, which is now available from the perspective of Hermitian $K$-theory by Theorem \ref{thm:main-theo-intro}.

\subsubsection{The Witt groups of $Q_3$}\label{sec:Witt-Q3} The Witt groups of $Q_3$  been computed by Nenashev \cite{nenashev2009on} (see also \cite{xie2019witt} for more general quadrics). Note that $Q_3 \cong \LG(2)$. Nenashev considered the localization sequence
$$0 \to W^{i-2}(\mathbb{P}^1, \Delta) \xrightarrow{\iota_{x*}} W^i(Q_3) \xrightarrow{\iota_y^*} W^i(\mathbb{P}^1) \to 0,  $$
and showed that $\iota^*_y$ is an isomorphism since $W^{i-2}(\mathbb{P}^1, \Delta) = 0$. However, it is more delicate to geometrically define an inverse of $\iota_y^*$. In Theorem \ref{thm:lg_induc} (b), we construct an inverse map of $\iota_y^*$ by the composition $$W^i(\mathbb{P}^1) \xrightarrow{\alpha^*} W^i(\LF_{1,2}(0;\mathring{V})_1) \xrightarrow{\pi_*} W^i(Q_3)$$
where $\alpha$  and $ \pi$ are both projections, and $\mathring{V} := V\oplus \SO_S^{\oplus 2}$ is the so-called padding. The interesting point is that the intermediate scheme $\LF_{1,2}(0;\mathring{V})_1$ is not regular, but rather reduced Gorenstein and consisting of two irreducible components of relative dimension zero over $Q_3$. We verify that the canonical line bundle of $\LF_{1,2}(0;\mathring{V})_1$ becomes trivial in the Picard group modulo two, guaranteeing that the pushforward map $\pi_*$ is well-defined.

\subsubsection{Semi-orthogonal decomposition} It is well-known that semi-orthogonal decompositions of derived categories can help to compute Hermitian $K$-theory by the work of Walter \cite{walter2003triangulated}. The point is that one needs to understand the symmetry of semi-orthogonal decompositions of derived categories under the Hom functor with values in a line bundle. A semi-orthogonal decomposition of derived category of coherent sheaves of Lagrangian Grassmannian $D^b(\LG(n))$ has been obtained by Fonarev \cite{fonarev2022full}, but with the assumption that $S=\mathrm{Spec}(k)$ for $k$ to be an algebraically closed field of characteristic zero. Since our work allows the base to be a field of positive characteristic, it is reasonable to conjecture that the semi-orthogonal decomposition of $D^b(\LG(n))$ is independent of the base.

\subsection*{Acknowledgement}
We thank Marcus Zibrowius for useful comments. The second author would like to thank Nicolas Perrin for a discussion on the two-step blow-up of Lagrangian Grassmannian some years ago. This work is partially supported by the National Key Research and Development Program of China No. 2023YFA1009800, NSFC Grant 12271529, NSFC Grant 12271500, the Fundamental Research Funds from the Central Universities, Sun Yat-sen University 34000-31610294. HX would also  like to acknowledge the EPSRC Grant EP/M001113/1, DFG Priority Programme SPP 1786 and the DFG-funded research training group GRK 2240: Algebro-Geometric Methods in Algebra, Arithmetic and Topology.

\subsection*{Convention}
Throughout this paper, we assume that every scheme is separated Noetherian of finite Krull dimension over $\mathbb{Z}[{\frac{1}{2}}]$.
% \footnote{The two invertible condition can now possibly be removed by the work \cite{marlowe2024higher}.} 
For any map $f:Y \to X$, write $L_{Y/X}:=\omega_{f} \otimes f^{\ast} L$. $f^{\natural}L:=\omega_{f} \otimes f^{\ast} L$.

%%%%%%%%%%%%%%%%%%%% Geometry of Lagrangian Grassmannians %%%%%%%%%%%%%%%%%%
\section{Geometry of Lagrangian Grassmannians}\label{sec:LagrangianGrass}

Let $V$ be a vector bundle over a regular scheme $S$ of rank $2n$, and let $ \beta: V \to V^{\vee}$ be a non-degenerate skew-symmetric form, that is, $\beta$ is an isomorphism such that $-\beta^\vee\circ \eta = \beta $ where $\eta: V \to V^{\vee \vee}$ is the double dual identification.

\begin{definition}
    Define the Lagrangian Grassmannian $\LG(n,V)$ to be the subscheme of the Grassmannian $p:\Gr(n,V)\to S$ that parametrizes isotropic subbundles in the form $(V,\beta)$. More precisely, the subscheme $\LG(n,V)$ is defined as the locus where the composition
    \begin{equation}\label{eqn:lg_section_seq}
        \xymatrix{L_n \,\,\ar[r]^-{i} & p^*V \ar[r]^-{i^\vee \circ \beta} & L_n^\vee  }
    \end{equation}
    vanishes. Here $L_n$ is the universal rank $n$ bundle over $\Gr(n,V)$.
\end{definition}
Note that $\LG(n,V)$ is a closed subscheme of $\Gr(n,V)$.
Denote $q:\LG(n,V) \to S $ the canonical projection. It follows that, after restricting \eqref{eqn:lg_section_seq} to $\LG(n,V)$, the sequence
\begin{equation}\label{eqn:lg_tautological_bundle_exact}
    \xymatrix{0 \ar[r] & L_n \ar[r]^-{i} & q^*V \ar[r]^-{i^\vee \circ \beta} & L_n^\vee \ar[r] & 0  }
\end{equation}
becomes exact, where $L_n$ still stands for the universal rank $n$ bundle over $\LG(n,V)$, i.e. the pullback of the universal bundle of the Grassmannian $\Gr(n,V)$. If no confusion occurs, we will drop the mention of $q^*$ in the notation $q^*V $ and write $V$ for simplicity.
\begin{remark}
    The scheme $\LG(n,V)$ is smooth of relative dimension $\binom{n+1}{2}$ over $S$. To see this, the relative tangent sheaf $T_{\LG(n,V)/S}$ can be identified with the vector bundle $(S^2L_n)^\vee$, which is of rank $\binom{n+1}{2}$.
\end{remark}
\begin{remark}
    If the rank of $V$ is two, note that $\LG(1,V) \cong \mathbb{P}(V)$, since every line is automatically isotropic.
\end{remark}

\subsection{A closed embedding}
Let $W$ be an isotropic subbundle of $V$ with rank $r$. Consider the universal exact sequence \eqref{eqn:lg_tautological_bundle_exact}
which leads to the closed embedding
$$\iota_W: \LG^{W}(n,V) \to \LG(n,V)$$
defined as the zero locus of composition $W \to V \xrightarrow{i^\vee\beta} L_n^\vee$.\
That is equivalent to taking the zero locus of the composition $s: L_n \xrightarrow{i} V \to V /W^\perp$.
Thus, over $\LG^{W}(n,V)$, we have the canonical subbundle relations $W \subset L_n$ and $L_n \subset W^\perp$.

\begin{lemma}
    The quotient map
    \[
        \begin{aligned}
            \rho:  \LG^{W}(n,V) & \xrightarrow{\cong} \LG(n-r,W^\perp/W), \\
            L_n                 & \mapsto L_n/W
        \end{aligned}
    \]
    is an isomorphism.
\end{lemma}
\begin{proof}
    The proof is similar to the symmetric case \cite[Lemma 2.12]{hudson2022witt}, which is based on \cite[Proposition 6.5]{quebbemann1979quadratic}.
\end{proof}
\subsection{One step blow-up}\label{subsec:lg-geometric}
In this subsection, we study the geometric blow-up setup of the Lagrangian Grassmannians. Assume that the form $(V,\beta)$ admits a complete flag, i.e., a filtration
\[
    V_{\bullet} = \{0=V_{0} \subset V_{1} \subset V_{2} \subset \ldots \subset V_{n-1} \subset V_{n} = V_{n}^{\bot} \subset V_{n-1}^{\bot} \subset \ldots \subset V_{1}^{\bot} \subset V_{0}^{\bot} = V\},
\]
with $\rank(V_j) = j$, and all inclusions are admissible (i.e., their quotients are locally free). Set $V^{j}:=V_{j}^{\bot}/V_{j}$, which admits an induced complete flag
\[
    V^{j}_{\bullet} = \{0=V^{j}_{0} \subset V^{j}_{1} \subset V^{j}_{2} \subset \ldots \subset V^{j}_{n-j-1} \subset V^{j}_{n-j} = V_{n-j}^{j \bot} \subset V_{n-j-1}^{j \bot} \subset \ldots \subset V_{1}^{j \bot} \subset V_{0}^{j \bot} = V^{j}\},
\]
where $V^{j}_{i} = V_{i+j}/V_{j}$ for $0\leq i+j\leq n$.

\begin{remark}
    Note that $\mathrm{det} V = \det V_n \otimes \det V_n^\vee = \SO$.
\end{remark}
Consider the Grassmannian $\Gr(n,V_1^\perp)$ and the closed immersion $\iota:\Gr(n,V_1^\perp) \to \Gr(n,V) $. Using functors of points, we see that the inclusion $L_n \subset V_1^\perp$ over $\LG^{V_1}(n,V)$ induces a morphism $\tilde{g}: \LG^{V_1}(n,V) \to \Gr(n, V_1^\perp)$, which yields the following fiber square:
\begin{equation}
    \xymatrix{
    & \LG^{V_1}(n,V) \ar[r]^-{\iota_1} \ar[d]^-{\tilde g} & \LG(n,V) \ar[d]^-{g}
    \\
    & \Gr(n,V_1^\bot) \ar[r]^-{\iota} & \Gr(n,V)
    }
\end{equation}
where we have $\iota_1 = \iota_{V_1}$. We further form the following cube diagram:
\begin{equation}\label{eqn:cube-lg}
    \xymatrix@R=25pt@C=0pt{
    & \LG^{V_1}(n,V) \ar[rr]^-{\iota_1} \ar'[d][dd]^-{\tilde g}
    & & \LG(n,V) \ar[dd]^-{g}
    \\
    \mathrm{E}_{1}(n,V) \ar[ur]^-{\tilde{\pi}_1}\ar[rr]^(.65){\tilde\iota_1}\ar[dd]^-{\imath}
    & & \Bl_1(n,V) \ar[ur]_-{\pi_1}\ar[dd]^(.3){\jmath}
    \\
    & \Gr(n,V_1^\bot) \ar'[r]^-{\iota}[rr]
    & & \Gr(n,V)
    \\
    \mathrm{Fl}_{n-1,n}(V_1^\bot,V_1^\bot) \ar[rr]^-{\tilde\iota}\ar[ur]_-{\tilde{\pi}}
    & & \mathrm{Fl}_{n-1,n}(V_1^\bot,V) \ar[ur]_-{\pi}
    }
\end{equation}
where the scheme $\Bl_1(n,V)$ (resp. $\mathrm{E}_1(n,V)$) is defined as the pullback of $g$ and $\pi$ (resp. $\tilde{\pi}$ and $\tilde{g}$) and the map $\tilde{\iota}_1$ is defined by the universal property of pullback. Here, recall that $\mathrm{Fl}_{n-1,n}(V_1^\bot,V_1^\bot)$ and $\mathrm{Fl}_{n-1,n}(V_1^\bot,V)$ are the generalized flag varieties defined in \cite{balmer2012witt}.
Every square in the cube diagram is a fiber product. Let $U_1(n,V)$ be the open complement $\LG(n,V) - \LG^{V_1}(n,V)$. Consider the square
\begin{equation}\label{eq:universal-element-Fl}
    \xymatrix{P_{n-1} \ar[d]\ar[r]& L_n \ar[d] \\
        V_1^\perp \ar[r] & V }
\end{equation}
which is the universal element of $\mathrm{Fl}_{n-1,n}(V_1^\bot,V)$.

\begin{lemma}\label{lem:birational-1} The morphism $\pi_1: \Bl_1(n,V) \to  \LG(n,V) $ is birational.
\end{lemma}
\begin{proof}
    Diagram \eqref{eq:universal-element-Fl}  induces a commutative diagram between two short exact sequences of vector bundles
    \begin{equation}\label{eqn:lg_bl_e_section}
        \xymatrix{
        0 \ar[r] & P_{n-1} \ar[r] \ar[d] & L_n \ar[r] \ar[dr]^-{s} \ar[d] & L_n / P_{n-1} \ar[r] \ar@{ -->}[d]_-{\tilde{s}}& 0\\
        0 \ar[r] & V_1^\bot \ar[r] & V \ar[r] & V / V_1^\bot \ar[r] & 0
        }
    \end{equation}
    over $\Bl_1(n,V)$. Note that $E_1(n,V)$ is the locus where the section $\tilde{s}: L_n / P_{n-1} \to V / V_1^\bot$ vanishes. Now, the open complement $\Bl_1(n,V)-\mathrm{E}_1(n,V)$ can be described as the locus where the morphism $\tilde{s}$ is surjective, cf. \cite[Proposition 8.4(1)]{gortz2020algebraic}. For a similar reason, the open complement $U_1(n,V)$ is the locus where the morphism $s: L_n \to V/V_1^\perp$ is surjective. This description leads to an isomorphism $U_1(n,V) \to \Bl_1(n,V)-\mathrm{E}_1(n,V)$ by taking the kernel $\ker(s)$, which is a vector bundle of rank $n-1$.
\end{proof}

\begin{proposition}\label{prop:blow-up-1}
    The scheme $\Bl_1(n,V)$ is the blow-up of $\LG(n,V)$ along $\LG^{V_1}(n,V)$ with exceptional fiber $\mathrm{E}_{1}(n,V)$.
\end{proposition}
\begin{proof}
    We verify that $\Bl_1(n,V)$ satisfies the universal property of blow-up. We claim that $\mathrm{E}_1(n,V)$ is an effective Cartier divisor of $\Bl_1(n,V)$. Observe that $\mathrm{E}_1(n,V)$ is a projective bundle of dimension $n-1$ over $\LG^{V_1}(n,V)$ since $\tilde{\pi}$ is so in view of the left face of Diagram \eqref{eqn:cube-lg}. Moreover, the codimension of $\iota_1$ is $n$, and thus the codimension of $\tilde{\iota}_1$ is $1$ by Lemma \ref{lem:birational-1}. Since $\tilde{\iota}_1$ can be identified with the zero locus of $\tilde{s}: L_n/P_{n-1} \to V/V_1^\perp$, which is a morphism of line bundles (cf. Diagram \eqref{eqn:lg_bl_e_section}), we see that $\mathrm{E}_1(n,V)$ must be an effective Cartier divisor of $\Bl_1(n,V)$.

    We need to check for any morphism $f:W\to \LG(n,V)$ such that $f^{-1}(\LG^{V_1}(n,V))$ is an effective Cartier divisor of $W$, then there is a unique morphism $\tilde{f}:W\to \Bl_1(n,V)$ such that $f = \pi_1 \circ \tilde{f}$. This follows along the same line as in the proof of \cite[Proposition 5.4]{balmer2012witt}.
\end{proof}

\begin{remark}
    One may also adopt an alternative proof by showing that the back face of the cube diagram \eqref{eqn:cube-lg} is tor-independent, which follows from Lemma \ref{lma:meet-proper-tor-independent} and a dimension count. By Lemma \ref{lma:tor-independent-eclate}, we get an alternative proof of Proposition 3.6.
\end{remark}

\subsection{Two step blow-up}
Consider the morphism $P_{n-1}^\perp/L_n \to V/V_1^\bot $ over $E_1(n,V)$ defined as the morphism induced on the cokernels in the following diagram:
\begin{equation}
    \xymatrix{
        0 \ar[r] & L_n \ar[r] \ar[d] & P_{n-1}^\bot \ar[r] \ar[d] & P_{n-1}^\bot/L_n  \ar[r] \ar@{-->}[d] & 0\\
        0 \ar[r] & V_1^\bot \ar[r] & V \ar[r] & V / V_1^\bot \ar[r] & 0
    }
\end{equation}
Let $Z_2(n,V)$ be the locus where the morphism $P_{n-1}^{\bot}/L_{n} \to V / V_{1}^{\bot}$ of line bundles vanishes. The induced embedding $\kappa:Z_2(n,V) \to E_1(n,V)$ is a regular embedding of codimension one.\ It follows that the composition
\[\iota_2:= \tilde{\iota}_1 \kappa: Z_2(n,V) \to E_1(n,V) \to \Bl_1(n,V)\]
is a regular embedding of codimension $2$. It is given by the vanishing locus of the canonical morphism $P_{n-1}^{\bot}/P_{n-1} \to V / V_{1}^{\bot}$ on $\Bl_1(n,V)$.

Let $\Bl_2(n,V)$ be the closed subscheme of $\Bl_1(n,V) \times \LG^{V_1}(n,V)$ defined as the vanishing locus of the canonical morphism $P_{n-1}\boxtimes \SO \to  V \boxtimes \SO = \SO \boxtimes V  \to  \SO \boxtimes V/L_n$. One has two projections:
\[ \xymatrix{\LG^{V_1}(n,V) & \ar[l]_-{\alpha} \Bl_2(n,V) \ar[r]^-{\pi_2} & \Bl_1(n,V). }  \]
\begin{remark}
    If no confusion occurs, we may drop the mention of $n$ and $V$ and write $B_1:= B_1(n,V)$ for instance.
\end{remark}
The following diagram summarizes the construction.
\begin{equation}\label{eqn:two_step_blow_up_lg}
    \xymatrix{
    &\LG^{V_1}   \ar[r]^-{\iota_1} &\LG   & \ar[l]_-{v_1} U_{1}   \ar[ld]_-{\tilde{v}_1} \ar@{=}[r]& U_{1}   \ar[dd]^-{\vartheta} \ar[ld]_-{w_1} \\
    Z_{2} \ar[r]^-{\kappa}
    %			 \ar@/_1pc/[rr]_-{\iota_2} 
    & E_{1}   \ar[r]^-{\tilde{\iota}_1} \ar[u]^-{\tilde{\pi}_1} & \Bl_{1}   \ar[u]^-{\pi_1} & \ar[l]_-{v_2} U_{2}   \ar[ld]_-{\tilde{v}_2}\\
    E_{2}   \ar[r]^-{\check{\kappa}} \ar[u]^-{\tilde \pi_2}
    %			\ar@/_1pc/[rr]_-{\tilde{\iota}_2} 
    & F_2   \ar@{..>}[u]^-{\check{\pi}_2} \ar[r]^{\check{\iota}_2} &   \ar[u]^-{{\pi}_2}  \Bl_{2}   \ar@{-->}[rr]^-{\alpha}& &\LG^{V_1}
    }
\end{equation}
where $ \vartheta:=\alpha \circ \tilde{v}_2 \circ w_1$, $U_2  $ is the open complement of $\iota_2$, $F_2$ is the fiber product of $\pi_2$ and $\tilde{\iota}_1$ and $E_2  $ is the fiber product of $\pi_2$ and $\iota_2$. Every morphism in Diagram \ref{eqn:two_step_blow_up_lg} is canonically defined. The functors of points are drawn in Figure \ref{fig:lg_blow_up_functor_of_points}. Let
\[ \iota_2 := \tilde{\iota}_1 \kappa, \quad \tilde\iota_2 := \check{\iota}_2 \check{\kappa},  \quad \pi:= \pi_1 \pi_2, \quad \tilde{\pi}:=\tilde{\pi}_1 \kappa\tilde{\pi}_2, \quad \tilde{\alpha}:= \alpha \tilde{\iota}_2, \quad \check{\pi}:= \tilde{\pi}_1\check{\pi}_2, \quad \check{\alpha} := \alpha \check{\iota}_2   . \]

\begin{proposition}\label{prop:blow-up-two-step-setup}
    The scheme $\Bl_2 $ is the blow-up of $Z_2 $ inside $\Bl_1 $ with the exceptional fiber $E_2 $.  Moreover, the morphism $\vartheta$ is an affine bundle. In other words, Lagrangian Grassmannians fit Hypothesis \ref{sect:hypothesis}.
\end{proposition}
\begin{proof}
    Let $Y= \LG^{V_1} $ and $\widetilde{L}_n$ be the universal rank $n$ bundle over $Y$. We give another construction of $\Bl_2 $ through $Y$. Consider the Grassmann bundle $ p_1: G \to Y$, where $G:=\Gr_Y(n-1, \widetilde{L}_n)$. On $G$, we have the canonical admissible embeddings $P_{n-1} \to \widetilde{L}_n \to V$ starting with the universal bundle of rank $n-1$. It follows that the bundle $P_{n-1}^\perp/P_{n-1}$ is a rank two vector bundle over $G$. Moreover, the projective bundle $p_2 : \mathbb{P}_G(P_{n-1}^\perp/P_{n-1}) \to G$ can be identified with $\Bl_2 $, and $\alpha = p_1 \circ p_2$.

    Observe that $\tilde{\iota}_2: \mathrm{E}_2  \to \Bl_2 $ is an effective Cartier divisor.\ To see this, we form the following pullback square
    \[
        \xymatrix{
        \mathrm{E}_{2}  \ar[d]_-{}\ar[r]^-{\tilde\iota_2} & \Bl_2  \ar[d]_-{p_2}\\
        \Gr_{Y}^{V_1}(n-1, \widetilde{L}_n) \ar[r]^-{} & \Gr_{Y}(n-1,\widetilde{L}_n)
        }
    \]
    where $ \Gr^{V_1}_Y(n-1,\widetilde{L}_n)$ is the closed subscheme of $\Gr_Y(n-1,\widetilde{L}_n)$ given by the locus where the morphism $V_1 \to \widetilde{L}_n \to \widetilde{L}_n/P_{n-1}$ of line bundles vanishes. Thus, $ \Gr^{V_1}_Y(n-1,\widetilde{L}_n)$ is an effective Cartier divisor of $\Gr_Y(n-1,\widetilde{L}_n)$. The first claim follows from the fact that $p_2$ is flat. Following the argument as in the proof of \cite[Proposition 5.4]{balmer2012witt}, we see that the map $\pi_2:\Bl_2  \to \Bl_1 $ satisfies the universal property of blow-up.

    We show that $\vartheta$ is an affine bundle. Let $W = \Gr_Y(n-1,\widetilde{L}_n) - \Gr^{V_1}_Y(n-1,\widetilde{L}_n)$. On the scheme $W$, we see that the short exact sequence $0 \to P_{n-1} \to \widetilde{L}_n \to \widetilde{L}_n/ P_{n-1} \to 0 $ splits, where $\widetilde{L}_n/ P_{n-1} \cong V_1$. Thus, the canonical map $V_1 \oplus P_{n-1}  \to \widetilde{L}_n$ is an isomorphism on $W$. Moreover, consider the following diagram

    \[
        \xymatrix@C=10pt@R=35pt{
        \mathbb{P}_W(P_{n-1}^\perp/\widetilde{L}_n) \ar[r] \ar[dr]  & \mathbb{P}_W(P_{n-1}^\perp/P_{n-1}) \ar[r] \ar[d] & \mathbb{P}_G(P_{n-1}^\perp/P_{n-1}) \ar[d]_-{p_2} \ar@/^2pc/[dd]^-{\alpha}\\
        &W\ar[r] \ar[dr] & G \ar[d]_-{p_1}\\
        & & Y
        }
    \]
    where all the maps are canonical defined and the square is a fiber product. Observe that $U_1 $ can be identified with the affine bundle $$\mathbb{P}_W(P_{n-1}^{\perp}/P_{n-1}) - \mathbb{P}_W(P_{n-1}^{\perp}/\widetilde{L}_{n})$$ over $W$. The result follows since $W$ is an affine bundle over $Y$.
\end{proof}

\subsection{On the Picard groups} Let $\Delta_n$ denote the determinant of the universal bundle $\Lambda^nL_n$ on $\LG(n,V)$.
\begin{proposition} The map
    \begin{align*}
        \mathbb{Z} \oplus \Pic(S) & \to \Pic(\LG(n,V))                          \\
        (m, L)                    & \mapsto (\Delta_n)^{\otimes m} \otimes q^*L
    \end{align*}
    is an isomorphism of abelian groups.
\end{proposition}
\begin{proof}
    The Picard group $\Pic(\LG(n,V))$ can be identified with the Chow group $\mathrm{CH}^1(\LG(n,V))$ (cf. \cite[Remark 11.41, p.311]{gortz2020algebraic}). Repeatedly applying Proposition \ref{prop:oriented-cohomology}, we obtain the following isomorphisms by pullbacks
    \[\mathrm{CH}^1(\LG(n,V)) \cong \mathrm{CH}^1(\LG^{V_1}(n,V))  \cong  \ldots \cong \mathrm{CH}^1(\LG^{V_{n-1}}(n,V)) \cong \mathrm{CH}^1(\LG(1,V^{n-1}))  \]
    Recall that $\LG(1,V^{n-1}) \cong \mathbb{P}(V^{n-1})$, and $\mathrm{CH}^1(\mathbb{P}(V^{n-1})) \cong \mathbb{Z} \oplus \Pic(S)$.
\end{proof}
\begin{remark}\label{rmk:LG-omega}
    Note that $\omega_{\LG/S} \cong \det(S^2L_n) \cong  (\Delta_n)^{\otimes n+1}$.
\end{remark}
\begin{proposition}
    Over $\LG^{V_1}(n,V)  $, we have that $\omega_{\iota_1} \cong (V_{1}^\vee)^{\otimes n} \otimes (\Delta_{n}^\vee).$
\end{proposition}
\begin{proof}
    Recall that section $s: \SO_{\LG  } \to V_1^\vee \otimes  L_n^\vee$ induces a regular embedding $\iota_1: \LG^{V_1}   \to \LG  $ of codimension $n$, which is precisely the rank of the bundle $V_1^\vee \otimes  L_n^\vee$. It follows that the normal bundle $N_{\iota_1} $ is isomorphic to $ V_1^\vee \otimes  L_n^\vee$ by \cite[Section 7.2]{calmes2011push} (see also \cite[Remark 6.8]{huang2023the}). Therefore, we conclude $\omega_{\iota_1} = \Lambda^n N_{\iota_1} \cong \Lambda^n(V_1^\vee \otimes  L_n^\vee) \cong (V_1^\vee)^{\otimes n} \otimes \Lambda^nL_n^\vee $. The result follows. Alternatively, one may use Remark \ref{rmk:LG-omega} to conclude another proof.
\end{proof}

\begin{lemma}\label{lem:picard-bl} The identity holds
    \[
        \alpha^* (\vartheta^*)^{-1}  v_1^*(\Delta_{n}) = (\pi_1 \circ \pi_2)^*(\Delta_{n}) \otimes \pi_2^*\SO(E_1) \otimes \SO(E_2)
    \]
    in the Picard group $\Pic(\Bl_2(n,V))$.
\end{lemma}
\begin{proof}
    Note that    $\SO(E_1) = (L_n/P_{n-1})^\vee \otimes V/V_1^\perp$ and  $\SO(E_2) = (P_{n-1}^\perp/\widetilde{L}_{n})^\vee \otimes V/V_1^\perp$. Moreover, we have
    \[ L_n/P_{n-1} \cong V/V_1^\perp \quad \textnormal{ and }  \quad \quad P_{n-1}^\perp/\widetilde{L}_n \cong V/V_1^\perp \]
    on $U_1(n,V)$, which implies the formula
    \[v_1^*\Delta_n = \alpha^*\widetilde{\Delta}_n \otimes (V/V_1^\perp)^{\otimes 2} \otimes \det V^\vee\]
    by taking the determinant.
    The result follows from the identity $\det V = \det P_{n-1}^\perp \otimes \det P_{n-1}^\vee$, which is a consequence of the exact sequence $0 \to P_{n-1}^\perp \to V \to P_{n-1}^\vee \to 0$ on $\Bl_2$.
\end{proof}

\begin{remark}\label{rmk:lg_lambda}
    In view of Equation \eqref{eqn:two_step_blow_up_lambda}, Lemma \ref{lem:picard-bl} implies that $\lambda_1(\Delta_{n}) =1$ and $\lambda_2(\Delta_{n})=1$.
\end{remark}

\subsection{Generalized Lagrangian flag schemes}
In this section, we introduce the (generalized) Lagrangian flag varieties, which arise as intermediate schemes in the construction of an additive basis for the (Hermitian) $K$-theory of Lagrangian Grassmannians.

\begin{definition}
    Let $k\geq 0$ be an integer, $\mathbf{d} = (d_0,d_1,d_2,\ldots, d_k) $ be a $(k+1)$-tuple of non-negative non-decreasing integers. Let $\mathbf{e} = (e_0,e_1,\ldots,e_{k-1}) $ (resp. $\mathbf{t} = (t_0,t_1,\ldots,t_{k-1})$) be a $k$-tuple of non-negative (resp. positive) non-decreasing integers (empty () when $k=0$) such that $d_j \leq n$ for any $0\leq j \leq k$, and $e_i \leq d_i, e_i \leq d_{i+1}$ and $e_i \leq n-t_i$ for any $0\leq i \leq k-1$. Assume that $V_\bullet$ is a complete (Lagrangian) flag. Define the \textit{generalized Lagrangian flag scheme} as the functor $LF_{\mathbf{d}}(\mathbf{e}; V_\bullet)_{\mathbf{t}}: \mathrm{Sch}_S \to \mathrm{Sets}$ by sending any morphism $ f: Y \to S$ to the set
    \[
        \setlength\arraycolsep{1pt} \renewcommand{\arraystretch}{1} LF_{\mathbf{d}}(\mathbf{e}; V_\bullet)_{\mathbf{t}}(Y)= \left\{ \begin{matrix}
            L_n^{(k)}, (P_{n-t_i
                }^{(i)},L_n^{(i)})_{k \geq i \geq 0} \,
        \end{matrix} \,\middle \vert\, \begin{matrix}
                            &         &                       &         & f^*V_{e_{0}}      & \subset & P_{n-t_0}^{(0)}       & \subset & L_{n}^{(0)}         & \subset & f^*V_{d_0}^{\bot} \\
                            &         &                       &         &                   &         & \cap                                                                                \\
                            &         & f^*V_{e_1}            & \subset & P_{n-t_1}^{(1)}   & \subset & L_{n}^{(1)}           & \subset & f^*V_{d_{1}}^{\bot}                               \\
                            &         &                       &         & \vdots                                                                                                            \\
            f^* V_{e_{k-1}} & \subset & P_{n-t_{k-1}}^{(k-1)} & \subset & L_{n}^{(k-1)}     & \subset & f^*V_{d_{k-1}}^{\bot}                                                               \\
                            &         & \cap                                                                                                                                                \\
                            &         & L_{n}^{(k)}           & \subset & f^*V_{d_k}^{\bot}
        \end{matrix}\right\}
    \]
    where $L_n^{(j)}$ are Lagrangian subbundles of rank $n$ inside $f^*V$ and where $P_{n-t_i}^{(i)}$ are admissible subbundles of rank $n-t_i$ inside both $L_n^{(i)}$ and $L_n^{(i+1)}$. If $V_\bullet$ is understood, we may drop the mention of $V_\bullet$ for simplicity.
\end{definition}

\begin{theorem}\label{prop:lf_regular_gorenstein} Let $k \geq 1$. The following statements hold true:
    \begin{enumerate}
        \item 	The functor $LF_{\mathbf{d}}(\mathbf{e})_{\mathbf{t}}$ is represented by a scheme $p:\mathrm{LF}_{\mathbf{d}}(\mathbf{e})_{\mathbf{t}} \to S$ with the universal bundles
              $P_{n-t_i}^{(i)}$ and $L_{n}^{(j)}$ ($0\leq i \leq k-1$ and $0\leq j \leq k$) satisfying the relations:
              \[ p^*V_{d_{i+1}}^\perp \supset  L_{n}^{(i+1)} \supset P_{n-t_i}^{(i)} \subset L_{n}^{(i)} \subset p^*V_{d_i}^\perp \quad  \textnormal{ and } \quad  p^*V_{e_i} \subset P_{n-t_i}^{(i)} \]
              where $L_n^{(j)}$ are Lagrangian subbundles of rank $n$ inside $p^*V$ and where $P_{n-t_i}^{(i)}$ are admissible subbundles inside $p^*V$ of rank $n-t_i$.
        \item  If $d_i = e_{i}$ for any $ 0 \leq i \leq k-1$, then the scheme $\mathrm{LF}_{\mathbf{d}}(\mathbf{e})_{\mathbf{t}}$ is regular. The relative dimension of $\mathrm{LF}_{\mathbf{d}}(\mathbf{e})_{\mathbf{t}}  \to S$ is equal to
              \begin{equation}\label{eq:dimension-LF}
                  \binom{n-d_{k}+1}{2} + \sum_{i=0}^{k-1}[(n-t_{i} - d_i)t_i + \binom{t_{i}+1}{2}].
              \end{equation}
        \item If $d_i -e_{i} \leq 1$ for any $ 0 \leq i \leq k-1$, then the scheme $\mathrm{LF}_{\mathbf{d}}(\mathbf{e})_{\mathbf{t}}$ is equidimensional Gorenstein of relative dimension over $S$ satisfying the formula \eqref{eq:dimension-LF}.
    \end{enumerate}
\end{theorem}
\begin{remark}
    If $t_i =1$ for any $i$, we write $\LF_{\mathbf{d}}(\mathbf{e}):= \LF_{\mathbf{d}}(\mathbf{e})_{\mathbf{t}}$. Note that if $k=0$, $\mathbf{e} = (~)$ is empty, then we let $LF_{d_0}:= LF_{\mathbf{d}}(\mathbf{e}) $, which is precisely the functor represented by the scheme $\LG^{V_{d_0}}(n,V)$.
\end{remark}
\begin{proof}
    All the statements can be proved by induction on $k$. Set $\mathbf{{d}}_{\widehat{0}} = (d_1,d_2,\ldots, d_{k})$. $\mathbf{{t}}_{\widehat{0}} = (t_{1},\ldots,t_{k-1})$ and $\mathbf{{e}}_{\widehat{0}} = (e_{1},\ldots,e_{k-1})$. Assume that the functor $LF_{\mathbf{d}_{\widehat{0}}}(\mathbf{e}_{\widehat{0}} )_{\mathbf{t}_{\widehat{0}}}$
    is represented by a scheme $Y:=\LF_{\mathbf{d}_{\widehat{0}}}(\mathbf{e}_{\widehat{0}})_{\mathbf{t}_{\widehat{0}}}$ with universal bundles
    $$(P_{n-t_i}^{(i)})_{1\leq i\leq k-1} \textnormal{ and } (L_{n}^{(j)})_{1\leq j \leq k}$$
    with the desired properties. Let $t, e$ be positive integers such that $V_e \subset L_n^{(1)}$ and $e \leq n-t$. Let $G_Y(e,t) $ be the Grassmann bundle $\Gr_{Y}^{V_{e}}(n-t,L_{n}^{(1)})$ and let $P_{n-t}^{(0)}$ be its universal bundle of rank $n-t$ over $G_Y(e,t)$. Define the scheme $\LG_{Y}(e,t)$ as the Lagrangian Grassmann bundle $\LG_{G_Y(e,t)}(t, P_{n-t}^{(0)\perp} / P_{n-t}^{(0)})$. Let $L_{n}^{(0)}$ be the universal bundle on $D:=\LG^{V_{d_0}}(n,V)$. Then we form the following fiber square
    \[
        \xymatrix{
        \mathrm{LF}_{\mathbf{d}}(\mathbf{e})_{\mathbf{t}} \ar[d] \ar[r] & \Gr_{D}^{V_{e_0}}(n-t_{0},L_{n}^{(0)}) \ar[d]^-{q}\\
        G_Y(e_0,t_0) \ar[r]^-{q'} & \Gr_{S}^{V_{e_0}}(n-t_{0}, V)
        }
    \]
    where $q$ and $q'$ are projections. The functor of points of the scheme $\mathrm{LF}_{\mathbf{d}}(\mathbf{e})_{\mathbf{t}}$ is precisely the functor $LF_{\mathbf{d}}(\mathbf{e})_{\mathbf{t}}$. This proves (i).

    For (ii), we assume that the scheme $Y$ is regular by the induction hypothesis. By identifying the functor of points, we obtain a functorial isomorphism
    \[
        \LF_{\mathbf{d}}(\mathbf{e})_{\mathbf{t}} \cong \LG_Y(e_0,t_0)
    \]
    since $d_0 = e_0$.\ The claim that $\mathrm{LF}_{\mathbf{d}}(\mathbf{e})_{\mathbf{t}}$ is a regular scheme follows from the fact that any (Lagrangian) Grassmann bundle over a regular base is smooth; hence, regular. Note that $\LG_{Y}(e_0,t_0)$ is of relative dimension
    \[
        (n_0 - t_0 - e_0)t_0 + \binom{t_0+1}{2}
    \]
    over $Y$. The dimension formula follows inductively.

    It suffices to prove (iii) when $d_0 = e_0 +1 $, and when the scheme $Y$ is equidimensional Gorenstein by the induction hypothesis.
    Note that the scheme
    $\LG_Y(e_0,t_0) \to Y$ is smooth over $Y$.
    Furthermore, observe that $\mathrm{LF}_{\mathbf{d}}(\mathbf{e})_{\mathbf{t}}$ can be identified with the closed subscheme of $\LG_Y(e_0,t_0) $ defined as the vanishing locus of the composition of the following canonical morphisms
    \begin{equation}\label{eq:section-s}
        s: L_n^{(0)}/P_{n-t_0}^{(0)} \hookrightarrow P_{n-t_0}^{(0)\perp} / P_{n-t_0}^{(0)} \to   P_{n-t_0}^{(0)\perp}/L_n^{(1)} \to V_{e_0}^\perp/L_n^{(1)} \to V_{e_0}^\perp/V_{d_0}^\perp.
    \end{equation}
    We claim that, as the zero locus of the section $s$, $\mathrm{LF}_{\mathbf{d}}(\mathbf{e})_{\mathbf{t}}$ is regularly embedded into $\LG_Y(e_0,t_0)$ of codimension $t_0$, which implies that $\mathrm{LF}_{\mathbf{d}}(\mathbf{e})_{\mathbf{t}}$ is Gorenstein, cf. \cite[Proposition 3.1.19(b)]{bruns1993cohen}. Since $s$ is a section of a rank $t_0$ vector bundle and $\LG_Y(e_0,t_0)$ is Gorenstein (hence Cohen-Macaulay), it is enough to show that  $\mathrm{LF}_{\mathbf{d}}(\mathbf{e})_{\mathbf{t}}$ is equidimensional of codimension $t_0$ inside $\LG_Y(e_0,t_0)$, cf. \cite[Theorem 2.1.2(c)]{bruns1993cohen}. To count the dimension, since $Y$ is equidimensional by the induction hypothesis, we may assume it is irreducible when computing the dimension. The scheme $\mathrm{LF}_{\mathbf{d}}(\mathbf{e})_{\mathbf{t}}$ consists of two irreducible components. The first irreducible component is given by $X_1 = \LG_Y(d_0,t_0)$.  It follows also that the relative dimension of $X_1$ over $Y$ is $(n - t_0-d_0)t_0 + \binom{t_{0}+1}{2}$. To construct the second irreducible component, we construct a projective  morphism
    $$ \Phi: \Gr^{V_{e_0}}_{\widetilde{X}}(n-t_0, P_{n-t_0+1}^{(0)}) \to  \mathrm{LF}_{\mathbf{d}}(\mathbf{e})_{\mathbf{t}} $$
    by only forgetting the strata $P^{(0)}_{n-t_0+1}$ where $\widetilde{X}:= \LG_Y(d_0,t_0-1)$, and the second irreducible component is given by the scheme-theoretical image $X_2$ of $\Phi$.
    To see that $\mathrm{LF}_{\mathbf{d}}(\mathbf{e})_{\mathbf{t}} = X_1 \cup X_2$, we note that $\mathrm{LF}_{\mathbf{d}}(\mathbf{e})_{\mathbf{t}} - X_1 \subset X_2$. Moreover, $\Phi$ induces an isomorphism
    \[ \Gr^{V_{e_0}}_{\widetilde{X}}(n-t_0, P_{n-t_0+1}^{(0)})  - \Gr_{\widetilde{X}}^{V_{d_0}}(n-t_0, P_{n-t_0+1}^{(0)}) \cong  \mathrm{LF}_{\mathbf{d}}(\mathbf{e})_{\mathbf{t}} - X_1, \] and therefore $\Phi$ is birational onto its (scheme theoretical) image. Observe that the relative dimension of $\Gr^{V_{e_0}}_{\widetilde{X}}(n-t_0, P_{n-t_0+1}^{(0)}) $ over $Y$ is equal to
    \[
        \begin{aligned}
              & (n - t_0 - d_0 + 1)(t_0-1) + \binom{t_{0}}{2} + n-t_0-e_0                               \\
            = & (n - t_0 - d_0 )t_0 - (n - t_0 - d_0) + t_0 -1 + \binom{t_{0}}{2} + (n - t_0 - d_0) + 1 \\
            = & (n - t_0 - d_0 )t_0 + \binom{t_{0} + 1}{2}
        \end{aligned}
    \]
    It follows that $X_2$ is of the same dimension as $X_1$, and therefore $\mathrm{LF}_{\mathbf{d}}(\mathbf{e})_{\mathbf{t}}$ is equidimensional of codimension $t_0$ inside $\LG_Y(e_0,t_0)$.
\end{proof}

Denote $\Delta_n^{(j)} := \det(L_n^{(j)})$ for $0\leq j \leq k$ and $\nabla_{n-t_i}^{(i)} := \det(P_{n-t_i}^{(i)})$ for $0\leq i\leq k-1$.
\begin{corollary}\label{coro:lfl_canonical_divisor}
    Suppose that $d_i -e_{i} \leq 1$ for any $ 0 \leq i \leq k-1$. Then, the relative canonical sheaf  $\omega_{\mathrm{LF}_{\mathbf{d}}(\mathbf{e})_{\mathbf{t}}/S}$ (See \cite[p.144]{hartshorne1966residues} for definition) is isomorphic to the following line bundle
    \[
        \begin{aligned}
            (\Delta_n^{(k)})^{n - d_k + 1} \det(V_{d_k})^{ d_k-n - 1}
            \prod_{i=0}^{k-1} \left( (\Delta_n^{(i)})^{t_i + e_i + 1 - d_i} (\Delta_n^{(i+1)})^{t_i + e_i - n} (\nabla_{n-t_i}^{(i)})^{n + d_i  - 2 e_i - t_i - 1} \det(V_{d_i})^{-t_i} \right).
        \end{aligned}
    \]
\end{corollary}
\begin{proof}
    Let us keep notation in the proof of Theorem \ref{prop:lf_regular_gorenstein}, and futher introduce notation:
    \[\mathrm{X} = \LG_Y(e_0,t_0), \mathrm{G} = G_Y(e_0,t_0), \mathrm{Z}= \LF_{\mathbf{d}}(\mathbf{e})_{\mathbf{t}}. \]
    Let us compute $\omega_{\mathrm{Z}/Y}$. Note that
    \begin{align}
        \label{1.1} 	\omega_{\mathrm{X}  /Y}         & = \omega_{\mathrm{X} /\mathrm{G}} \otimes \omega_{\mathrm{G}/Y}                                                                                                         \\
        \label{1.2} 	\omega_{\mathrm{X} /\mathrm{G}} & = \det(L_n^{(0)} / P_{n-t_0}^{(0)})^{\otimes t_0+1}                                                                                                                     \\
        \label{1.3} 	 \omega_{\mathrm{G}/Y}          & = \det(L_n^{(1)} / V_{e_0})^{\otimes -(n-t_0-e_0)} \otimes \det(P_{n-t_0}^{(0)} / V_{e_0})^{\otimes n-e_0}                                                              \\
        \label{1.4} 	\omega_{\mathrm{X} /Y}          & =  \Delta_n^{(0)\otimes t_0 + 1} \otimes \Delta_n^{(1)\otimes t_0 + e_0 - n} \otimes \nabla_{n-t_0}^{(0)\otimes n - t_0 - e_0 - 1} \otimes \det(V_{e_0})^{\otimes -t_0}
    \end{align}
    where \eqref{1.4} is obtained from putting \eqref{1.2} and \eqref{1.3} into \eqref{1.1}.

    There are two cases to consider: (i). If $d_0 = e_0$, then $\mathrm{Z} = \mathrm{X}$, and $\omega_{\mathrm{Z}/Y} = \omega_{\mathrm{X}/Y}$. (ii). If $d_0 = e_0 + 1$, then $\mathrm{Z}$ is a closed subscheme in $\mathrm{X}$ defined by the vanishing locus of the section defined in Equation~\eqref{eq:section-s}.
    \[ s: L_n^{(0)}/P_{n-t_0}^{(0)} \otimes (V_{e_0}^\perp/V_{d_0}^\perp)^\vee \to \SO, \]
    and we have shown that $s$ is a regular section. Therefore, we get that
    \begin{align}
        \label{1.5} 	\omega_{\mathrm{Z} /Y}           & = \omega_{\mathrm{Z} /\mathrm{X} } \otimes \omega_{\mathrm{X} /Y}                                                                                              \\
        \label{1.6} 	\omega_{\mathrm{Z} /\mathrm{X} } & = \det(L_n^{(0)}/P_{n-t_0}^{(0)})^{\vee} \otimes (V_{d_0} / V_{e_0})^{\otimes -t_0}                                                                            \\
        \label{1.7} \omega_{\mathrm{X} /Y}            & = \Delta_n^{(0)\otimes t_0} \otimes \Delta_n^{(1)\otimes t_0 + e_0 - n} \otimes \nabla_{n-t_0}^{(0)\otimes n - t_0 - e_0} \otimes \det(V_{d_0})^{\otimes -t_0}
    \end{align}
    where \eqref{1.7} is obtained from putting \eqref{1.4} and \eqref{1.6} into \eqref{1.5}.
    Combining these two cases together, we have
    \[
        \omega_{\mathrm{Z} /Y} = \Delta_n^{(0)\otimes t_0 + e_0 + 1 - d_0} \otimes \Delta_n^{(1)\otimes t_0 + e_0 - n} \otimes \nabla_{n-t_0}^{(0)\otimes n + d_0 - 2e_0 - t_0 - 1} \otimes \det(V_{d_0})^{\otimes -t_0}
    \]
    The result follows from induction.
\end{proof}

\begin{example}\label{exm:LF012-132-022}
    If $k=0$, we have $ \omega_{\LF_{d}/S} \cong (\Delta_n^{(0)})^{n-d+1} \otimes \det(V_d)^{d-1-n}$. If $k =1$, we obtain that $\omega_{\LF_{d_0,d_1 ;t_0}(e_0)/S}$ is isomorphic to the following line bundle:
    \[
        \Delta_n^{(0)\otimes t_0 + e_0 + 1 - d_0} \cdot \Delta_n^{(1)\otimes t_0 + e_0 +1 - d_1} \cdot \nabla_{n-t_0}^{(0)\otimes n + d_0 - 2e_0 - t_0 - 1} \cdot \det(V_{d_0})^{\otimes -t_0} \cdot \det(V_{d_1})^{\otimes d_1-n-1}.
    \]
    In particular, we compute a few useful examples of relative canonical sheaf for $\LF_{\mathbf{d}}(\mathbf{e})_{\mathbf{t}}$:
    \[ \begin{aligned}
            \omega_{\LF_{1,2}(0)/S}   & = \Delta_n^{(0)} \otimes (\nabla_{n-1}^{(0)})^{n-1} \otimes \det(V_2)^{1-n}
            \otimes \det(V_1)^{-1} ,                                                                                                             \\
            \omega_{\LF_{1,3}(0)_2/S} & = (\Delta_n^{(0)})^2 \otimes (\nabla_{n-2}^{(0)})^{n-2}  \otimes \det(V_3)^{2-n}
            \otimes \det(V_1)^{-2},                                                                                                              \\
            \omega_{\LF_{0,2}(0)_2/S} & =  (\Delta_n^{(0)})^3 \otimes \Delta_n^{(1)} \otimes (\nabla_{n-2}^{(0)})^{n-3}  \otimes \det(V_3)^{2-n}
            \otimes \det(V_1)^{-2}.
        \end{aligned}\]
\end{example}

\begin{remark}\label{rmk:lf_irr_comp}
    The interesting point is that the scheme $\mathrm{LF}_{\mathbf{d}}(\mathbf{e})_{\mathbf{t}} $ may not be irreducible. For example, if $d_i -e_i \leq 1$, the number of irreducible components of $\mathrm{LF}_{\mathbf{d}}(\mathbf{e})_{\mathbf{t}}$ is $2^s$ provided that there are $s$ pairs of integers $(d_i,e_{i+1})$ such that $d_i -e_{i+1} =1$.
\end{remark}

\begin{lemma}\label{lma:lf_de_structure_sheaf}
    If $d_i -e_{i} \leq 1$ for any $ 0 \leq i \leq k-1$, the scheme $\mathrm{LF}_{\mathbf{d}}(\mathbf{e})$ is reduced and the canonical map
    \[
        p^{\sharp}:\SO_S \xrightarrow{\cong} p_* \SO_{\mathrm{LF}_{\mathbf{d}}(\mathbf{e})}
    \]
    is an isomorphism.
\end{lemma}
\begin{proof}
    If $d_i = e_{i}$ for any $ 0 \leq i \leq k-1$, the scheme $\mathrm{LF}_{\mathbf{d}}(\mathbf{e})$ is regular and smooth over $S$. By \cite[\href{https://stacks.math.columbia.edu/tag/0E0L}{Lemma 0E0L}]{stacks-project} (see also \cite[Exercise 9.3.11, p.260]{FAG}), the map $p^{\sharp}:\SO_S \xrightarrow{\cong} p_* \SO_{\mathrm{LF}_0}$ is an isomorphism.

    Using the notation and the induction process in the proof of Theorem \ref{prop:lf_regular_gorenstein}, it suffices to prove for the case where $d_0 = e_0 + 1$ and $d_i = e_i$ for $i\neq 0$. The scheme $X := \mathrm{LF}_{\mathbf{d}}(\mathbf{e})$ is a closed codimension one subscheme of $\LG_Y(e_0,1)$, defined as the zero locus of the section $s$ given in \eqref{eq:section-s}. Note that $X$ is the union of two codimension one subschemes $X_1$ and $X_2$. Moreover, $X_1$ and $X_2$ intersect transversely in $X$ and are both regular. The section $s$ defined in Equation~\eqref{eq:section-s}, is the composition of two sections that determine $X_1$ and $X_2$, respectively. It follows that $X$ is reduced.

    The fact that $X$ is reduced and the section $s$ is regular, is independent of the choice of the base scheme. Then we can apply Lemma \ref{lma:proper_flat_regular_section} to conclude that the map $p:\mathrm{LF}_{\mathbf{d}}(\mathbf{e}) \to S$ is flat and has reduced and connected geometric fibers. Then the isomorphism is a direct consequence of \textit{loc. cit}.
\end{proof}
\begin{remark}
    Note that
    %		$$\LG(n,V) = \LF_0, \quad \LG^{V_1}(n,V) = \LF_1, \quad $$ 
    $E_2 = \mathrm{LF}_{1,1}(1), F_2 = \mathrm{LF}_{1,1}(0), $ and $ \Bl_2 = \mathrm{LF}_{0,1}(0)$, where $F_2$ is not irreducible.
\end{remark}

\subsection{Canonical maps between Lagrangian flag varieties}
\subsubsection{} \label{sec:tuples} Let $\epsilon_i^k = (0, \ldots, 0, 1, 0 , \ldots ,0 ) \in \mathbb{Z}^{k+1}$ be the standard $k+1$-tuple with $1$ in the $(i+1)$-th place and $0$ in other places. Let $\mathbf{1}_k:= \sum_{i=0}^k  \epsilon_i^k = (1, \ldots, 1)$ and $\mathbf{2}_k:= \sum_{i=0}^k  2 \epsilon_i^k = (2, \ldots, 2)$. If $\mathbf{d} = (d_0, d_1, \ldots, d_{k-1},d_k)$ is a $(k+1)$-tuple, let $\mathbf{d}^{<k} = (d_0, d_1, \ldots, d_{k-1})$ be the $k$-tuple obtained by deleting the last digit $d_k$ from the original $(k+1)$-tuple $\mathbf{d}$.

\subsubsection{} Suppose that we have another $(k+1)$-tuple $\mathbf{\tilde{d}} = (\tilde d_0,\tilde d_1,\ldots,\tilde d_k)$ and another $k$-tuple $\mathbf{\tilde e} = (\tilde e_0,\ldots,\tilde e_{k-1})$ such that $\tilde e_i \leq \tilde d_i$ and $\tilde e_i \leq \tilde d_{i-1}$ for any $1\leq i \leq k$. Suppose further that $d_i \geq \tilde d_i$ and $e_i \geq \tilde e_i$. Then, we have a closed embedding
\[  \iota: \mathrm{LF}_{\mathbf{d}}(\mathbf{e}) \hookrightarrow \mathrm{LF}_{\mathbf{\tilde d}}(\mathbf{ \tilde e}) \]
by the inclusion of each stratum on the functor of points.
\subsubsection{}
Let $t,s$ be integers such that $1\leq t, s\leq k$. Let $I=\{1,\ldots,t\}$ (resp.\ $J=\{s,\ldots,k\}$). Set $\mathbf{d}_I: = (d_0, \ldots, d_t)$ and $\mathbf{e}_I:=(e_1, \ldots, e_t) $ (resp. $\mathbf{d}_J: = (d_s, \ldots, d_k)$ and $\mathbf{e}_J:=(e_{s+1}, \ldots, e_k) $). Then, one can construct a natural projection
\[  \pi: \mathrm{LF}_{\mathbf{d}}(\mathbf{e})  \to \mathrm{LF}_{\mathbf{d}_I}(\mathbf{e}_I) \quad \textnormal{ (resp.\ $\alpha: \mathrm{LF}_{\mathbf{d}}(\mathbf{e})  \to \mathrm{LF}_{\mathbf{d}_J}(\mathbf{e}_J)$)} \]
by forgetting the strata which range out $I$ (resp.\ $J$).
\subsubsection{ } Suppose that $e_i,d_j \geq 1$. Then, we have an isomorphism
\[ \rho: \mathrm{LF}_{\mathbf{d}}(\mathbf{e};V_\bullet) \to \mathrm{LF}_{\mathbf{d}-\mathbf{1}_{k+1}}(\mathbf{e}-\mathbf{1}_{k}; V^1_\bullet) \]
by taking the bundles $P_{n-1}^{(i)}, L_{n-1}^{(j)}$ to the quotients $P_{n-1}^{(i)}/V_1, L_{n-1}^{(j)}/V_1$.

\begin{lemma}\label{lma:lf_hat_pic}
    Let $n\geq 3$. Consider the closed embeddings $\LF_{1,2}(1)  \xrightarrow{\iota} \LF_{1,2}(0)$, $\LF_{1,2}(0)  \xrightarrow{\iota_b} \mathrm{LF}_{0,2}(0) $, and $\LF_{1,2}(1)  \xrightarrow{\iota_a} \mathrm{LF}_{0,2}(0)$
    with $ \iota_b \circ \iota = \iota_a$. Then, the map $$\iota_a^*:\Pic(\mathrm{LF}_{0,2}(0))/2 \to \Pic(\LF_{1,2}(1))/2$$ is an isomorphism, and the map $\iota_b^*:\Pic(\mathrm{LF}_{0,2}(0))/2 \to \Pic(\mathrm{LF}_{1,2}(0))/2$ is injective.
\end{lemma}
\begin{proof}
    Note that closed embeddings $\iota_a$ and $\iota_b$ can be identified with the zero locus of $V_1 \to L_n^{(1)}/P^{(0)}_{n-1}$ and $L_n^{(0)}/P^{(0)}_{n-1} \to V/V_{1}^{\bot}$, respectively.
    The result follows from the commutative diagram
    \[
        \xymatrix{
        \LF_{1,2}(1)  \ar[r] \ar[d]_-{{\iota_a}} & \Gr^{V_1}(n-1,L_n^{(1)}) \ar[d]_-{\iota_c} \ar[dr] \\
        \mathrm{LF}_{0,2}(0) \ar[r] & \Gr(n-1,L_n^{(1)}) \ar[r] & \LF_{2}
        }
    \]
    where $L_n^{(1)}$ is the tautological bundle on $\LF_{2}$ and where the left square is a fiber square. Since $n\geq 3$, the map
    \[
        \iota_c^*: \Pic(\Gr(n-1,L_n^{(1)}))/2 \to \Pic(\Gr^{V_1}(n-1,L_n^{(1)}))/2
    \]
    is an isomorphism by \cite[Proposition 1.3]{balmer2012witt}. The map $\iota_a^*$ is then an isomorphism on the Picard group modulo two, as the morphism $\mathrm{LF}_{0,2}(0) \to \Gr(n-1,L_n^{(1)})$ is a projective bundle.
\end{proof}

%%%%%%%%%%%%%%%%%%%% K-theory of Lagrangian Grassmannians %%%%%%%%%%%%%%%%%%
\section{\texorpdfstring{$K$}{K}-theory of Lagrangian Grassmannians} \label{sec:K_theory_of_lg}
The additive basis of the $K$-theory of Lagrangian Grassmannians can be constructed via Schubert varieties, which is well known. In this section, we study the $K$-theory of Lagrangian Grassmannians via our generalized Lagrangian flag schemes associated to shifted Young diagrams. This serves as a warm-up for our computation for the Hermitian $K$-theory of Lagrangian Grassmannians, and is written in terms of generalized Lagrangian flags. Recall that $\LF_i := \LG^{V_i}(n, V)$.
\subsection{Generalized Lagrangian flag schemes}
\begin{definition}
    A \textit{special marked point} is a lattice point $a$ that sits in some horizontal segment $s_{t}$ ($t$ even) that does not touch the terminal point of $s_t$.
\end{definition}

Consider all the special marked points on each horizontal segment $s_t$, ordered from right to left
\[
    a_{t,1},\, a_{t,2},\, \ldots,\, a_{t,|s_t|},
\]
where $a_{t,1} := a_t$ is a convex corner and $a_{t,|s_t|}$ is the last special marked point on $s_t$. Observe that the lattice distance between consecutive special marked points on the same horizontal segment satisfies $|a_{t,i}a_{t,i+1}|=1$. On a horizontal segment $s_t$, the number of special marked points is $|s_t|$.
This is illustrated in Figure \ref{fig:boundary-shifted-young}.

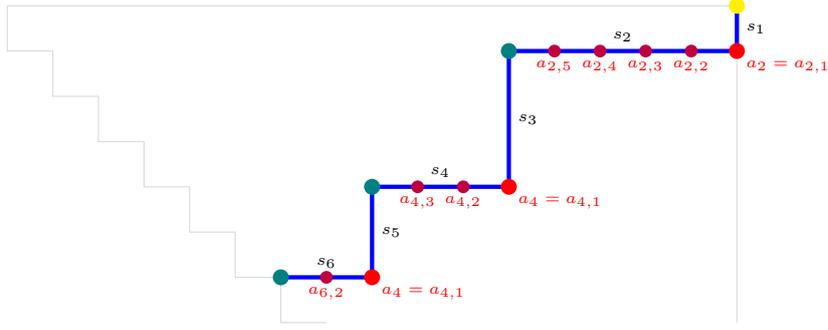
\begin{figure}[!ht]
    \centering
    \begin{tikzpicture}[scale=0.6,>=Stealth]
        \def\n{14}
        \foreach \i in {0,...,6} {
                \draw[gray!30] (\i+1 - 2,\n-1-\i) -- (\i - 2,\n-1-\i) -- (\i - 2,\n-\i);
            }
        \draw[gray!30] (\n,7) -- (\n,\n);
        \draw[gray!30] (-2,\n) -- (\n,\n);

        \coordinate (O)  at (14,14);
        \coordinate (a2) at (14,13);
        \coordinate (a3) at (9,13);
        \coordinate (a4) at (9,10);
        \coordinate (a5) at (6,10);
        \coordinate (a6) at (6,8);
        \coordinate (a7) at (4,8);

        \draw[blue,ultra thick] (O) -- (a2) -- (a3) -- (a4) -- (a5) -- (a6) -- (a7);

        \fill[yellow] (O)  circle (5pt);
        \foreach \P in {a2,a4,a6} \fill[red]  (\P) circle (5pt);
        \foreach \P in {a3,a5,a7} \fill[teal] (\P) circle (5pt);

        \node[anchor=north west,color=red] at (a2) {\tiny $a_2=a_{2,1}$};
        \foreach \x in {2,3,4,5} { \fill[purple] (15-\x,13) circle (4pt); \node[anchor=north,color=red] at (15-\x,13) {\tiny $a_{2,\x}$};}
        \node[anchor=north west,color=red] at (a4) {\tiny $a_4=a_{4,1}$};
        \foreach \x in {2,3} { \fill[purple] (10 - \x,10) circle (4pt); \node[anchor=north,color=red] at (10-\x,10) {\tiny $a_{4,\x}$};}
        \node[anchor=north west,color=red] at (a6) {\tiny $a_4=a_{4,1}$};
        \fill[purple] (5,8) circle (4pt);
        \node[anchor=north,color=red] at (5,8) {\tiny $a_{6,2}$};

        \node[anchor=west] at (14,13.5) {\tiny $s_1$};
        \node[anchor=south] at (11.5,13) {\tiny $s_2$};
        \node[anchor=west] at (9,11.5) {\tiny $s_3$};
        \node[anchor=south] at (7.5,10) {\tiny $s_4$};
        \node[anchor=west] at (6,9) {\tiny $s_5$};
        \node[anchor=south] at (5,8) {\tiny $s_6$};
    \end{tikzpicture}
    \caption{Boundary $b_\Lambda$ and special marked points.}
    \label{fig:boundary-shifted-young}
\end{figure}

Let $\tau_{\Lambda}:= 2 \cdot \big\lfloor \tfrac{l_\Lambda}{2} \big\rfloor$ (the largest even integer $\le l_\Lambda$). Then the total number of special marked points is $\Sigma_\Lambda:= \sum_{j=1}^{\tau_\Lambda/2} |s_{2j}|$.

\begin{definition}\label{def:d-lambda}
    Let $\Lambda$ be a \textit{shifted} Young diagram in $\mathfrak{A}_n$. Define
    $\mathbf{{d}}_{\Lambda} = (d_0, d_1, \ldots d_{\Sigma-1})$ (resp. $\mathbf{{d}}_{\Lambda} = (d_0, d_1, \ldots d_\Sigma)$) with $0 \leq d_i \leq n$
    to be the following $\Sigma_\Lambda$-tuple (resp.\ $(\Sigma_\Lambda + 1)$-tuple)
    \[
        \mathbf{d}_\Lambda:=
        \begin{cases}
            \big(|Oa_{2,1}|,\ldots, |Oa_{2,|s_2|}|,\ |Oa_{4,1}|,\ldots, |Oa_{4,|s_4|}|, \ldots, |Oa_{\tau_\Lambda,1}|,\ldots, |Oa_{\tau_\Lambda,|s_{\tau_\Lambda}|}|\big)       & \textnormal{if $l_\Lambda$ is even}, \\
            \big(|Oa_{2,1}|,\ldots, |Oa_{2,|s_2|}|,\ |Oa_{4,1}|,\ldots, |Oa_{4,|s_4|}|, \ldots, |Oa_{\tau_\Lambda,1}|,\ldots, |Oa_{\tau_\Lambda,|s_{\tau_\Lambda}|}|,\, n \big) & \textnormal{if $l_\Lambda$ is odd}.
        \end{cases}
    \]
\end{definition}

\begin{definition}
    Define the scheme
    \[
        \mathrm{LF}(\Lambda; {V}_\bullet):=\mathrm{LF}_{\mathbf{{d}}_{\Lambda}}(\mathbf{e}_{\Lambda}; {V}_\bullet),		\]
    where $\mathbf{e}_{\Lambda} := \mathbf{d}^{<k}_\Lambda = (d_0, \cdots, d_{k-1})$, with $\mathbf{d}_\Lambda$ a $(k+1)$-tuple. If ${V}_\bullet$ is understood, we simply write $\mathrm{LF}({\Lambda})$.
\end{definition}

The natural projection $p_\Lambda: \mathrm{LF}({\Lambda}) \to S$ is flat. Moreover, one can define a projective morphism
\begin{align*}
    f_\Lambda: \mathrm{LF}({\Lambda}) & \to \LG(n,V)
\end{align*}
by setting $f_\Lambda\big( L_n^{(0)}, (P_{n-1}^{(i)},L_n^{(i)})_{k \geq i \geq 1} \big) = L_n^{(0)} $.

\subsection{Combinatorial operations on Young diagrams}
\subsubsection{ }
Let $\Lambda \in \mathfrak{U}_n$. Define two operations
\[
    {\iota}(\Lambda) = (\Lambda_{n-1}, \Lambda_{n-2}, \ldots, \Lambda_1) \quad \textnormal{ if $\Lambda_n = n$}, \quad \quad
    v(\Lambda) = (\Lambda_{n}, \Lambda_{n-1}, \ldots, \Lambda_2) \quad  \textnormal{ if $\Lambda_n < n$}.
\]
Note that if $\Lambda_n < n$, then we must have $\Lambda_1=0$.
\subsubsection{} \label{sect:iota-c}
The set $\mathfrak{U}_n$ splits into two disjoint parts
\[ \mathfrak{U}_n = \mathfrak{U}_n^r \bigsqcup \mathfrak{U}_n^c \]
where $\mathfrak{U}_n^r:=  \{\Lambda \in \mathfrak{U}_{n} \mid \Lambda_n = n\} $ (resp. $
    \mathfrak{U}_n^c:= \{\Lambda \in \mathfrak{U}_{n} \mid \Lambda_n < n\}  $) denotes the subset of shifted Young diagrams whose topmost rows is full (resp. whose rightmost columns is empty). Note that we have a bijection of sets
\[\begin{array}{ll}
        {\iota}:\mathfrak{U}_n^r \xrightarrow{\cong} \mathfrak{U}_{n-1}   \quad
        {v}: \mathfrak{U}_n^c \xrightarrow{\cong} \mathfrak{U}_{n-1}
    \end{array} \]
where the operation $\iota$ (resp. $v$) corresponds to deleting the topmost row (resp. the rightmost column) of a shifted Young diagram.

\begin{lemma}\label{lem:lf_lambda}
    Let $\Lambda \in \mathfrak{U}_n$. The diagrams below are fiber squares
    \begin{equation} \xymatrix{
        \mathrm{LF}(\Lambda;V_\bullet)\ar[d]^-{\iota \pi} \ar[r]^-{\rho\alpha} & \mathrm{LF}(v(\Lambda);V^1_\bullet) \ar[d]^-{f_{v(\Lambda)}} \\
        \mathrm{LF}_{0,1}(0;  V_\bullet) \ar[r]^-{\rho\alpha} &\LF_1(V)
        } \quad
        \xymatrix{
        \mathrm{LF}({\iota}(\Lambda);V^1_\bullet)\ar[d]^-{f_{{\iota}(\Lambda)}} \ar[r]^-{\rho^{-1}}_{\cong} & \mathrm{LF}(\Lambda;V_\bullet) \ar[d]^-{f_{\Lambda}} \\
        \LG(n-1,V^1) \ar[r]^-{\iota \rho^{-1}} &\LG(n, V)
        }
    \end{equation}
    where $\Lambda_n < n$ for the first diagram and $\Lambda_n =n$ for the second one.
\end{lemma}
\begin{proof}
    The first diagram is a fiber square because $s_1=0$ and $d_0 = |Oa_{2,1}| =0$. This can be verified by explicitly describing the map $\pi:\LF(\Lambda) \to \LF_{0,d_1}(0)$. The operation $v(\Lambda)$ moves the origin to the left by one lattice length. It follows that removing $d_1$ from $\mathbf{d}_\Lambda$ yields $\mathbf{d}_{v(\Lambda)} + \mathbf{1}_{k-1}$, where $\mathbf{d}_\Lambda$ is a $(k+1)$-tuple and the addition is performed component-wise.
    The second diagram is commutative, as the universal bundle $L^{(0)}_n$ over $\LF(\Lambda; V_\bullet)$ is contained in $V_1^\perp$.
\end{proof}
\subsection{\texorpdfstring{$K$}{K}-theory of Lagrangian Grassmannian}
For any $\Lambda \in \mathfrak{U}_n$, we construct a map
\[ \phi_\Lambda : K(S) \xrightarrow{p_\Lambda^*} K(\mathrm{LF_{\Lambda}}) \xrightarrow{f_{\Lambda*}}  K(\LG(n,V))\]
of $K$-theory spectra.
\begin{theorem}
    Let $S$ be a regular scheme. The map
    \[
        \sum \phi_\Lambda:  \bigoplus_{\Lambda \in \mathfrak{U}_n} K(S) \xrightarrow{\simeq} K\big(\LG(n,V)\big)
    \]
    is an equivalence of spectra.
\end{theorem}
\begin{proof}
    We prove this result by induction on $n$. When $n=1$, $\LG(n,V) \cong \mathbb{P}(V)$. The map
    \[
        (\iota_*, p^*): K(S) \oplus K(S) \xrightarrow{\simeq} K(\mathbb{P}(V))
    \]
    is an equivalence of spectra.\ Now suppose that the statement is true for $n-1$. The result follows from the commutativity of the following diagram
    \begin{equation}\label{eq:$K$-theory-LG-induction}
        \xymatrix@R=60pt{
        \Big(\bigoplus\limits_{\Lambda \in \mathfrak{U}_{n}^r} K(S)\Big) \oplus \Big(\bigoplus\limits_{\Lambda \in \mathfrak{U}_{n}^c} K(S)\Big)  \ar[d]_-{\begin{pmatrix}
            \scriptstyle \sum \rho\phi_{\iota(\Lambda)} & \scriptstyle 0 \\ \scriptstyle 0 & \scriptstyle \sum \rho \phi_{v(\Lambda)}
        \end{pmatrix}}^-{\simeq} \ar[rr]^-{=} &&   \bigoplus\limits_{\Lambda \in \mathfrak{U}_{n}} K(S)  \ar[d]^-{\sum \phi_\Lambda}  \\
        K(\LG^{V_1}(n,V)) \oplus K(\LG^{V_1}(n,V)) \ar[rr]^-{ \begin{pmatrix}
            \scriptstyle \iota_{1*} & \scriptstyle \pi_* \alpha^*
        \end{pmatrix}}_-{\simeq} && K(\LG(n,V))
        }
    \end{equation}
    where the bottom map is an equivalence of spectra by Proposition \ref{prop:oriented-cohomology} and \ref{prop:blow-up-two-step-setup}.     The commutativity of Diagram \ref{eq:$K$-theory-LG-induction} follows from the following equations
    \begin{align}
        \iota_{1*} \rho f_{{\iota}(\Lambda)*} p_{{\iota}(\Lambda)}^* & = f_{\Lambda*} p_{\Lambda}^*, \label{eqn:lg_k_ind_1} \\
        \pi_* \alpha^* \rho f_{v(\Lambda)*} p_{v(\Lambda)}^*         & = f_{\Lambda*} p_{\Lambda}^*,\label{eqn:lg_k_ind_2}
    \end{align}
    for any $\Lambda \in \mathfrak{U}_{n}^{r}$ and any $\Lambda \in \mathfrak{U}_{n}^{c}$ respectively. By Lemma \ref{lem:lf_lambda}, Equation \eqref{eqn:lg_k_ind_1} holds. To verify Equation \eqref{eqn:lg_k_ind_2}, observe that for any $\Lambda \in \mathfrak{U}_{n}^{c}$, by Lemma \ref{lem:lf_lambda} we have the following commutative diagram
    $$
        \xymatrix{
        &S \\
        \mathrm{LF}(v(\Lambda); V^1) \ar[ur]^-{p_{v(\Lambda)}} \ar[d]_-{\rho  f_{v(\Lambda)}} \ar@{}[dr]|-{\diagram \label{diag:lg_k}}  & \mathrm{LF}({\Lambda}; V) \ar[d] \ar@{=}[r] \ar[l] & \mathrm{LF}({\Lambda};V) \ar[d]_-{f_{\Lambda}} \ar[ul]_-{p_{\Lambda}} \\
        \LF_1(V) & \mathrm{LF}_{0,1}(0;V) \ar[l]_-{\alpha} \ar[r]^-{\pi} &  \LF_0(V)
        }
    $$
    where Diagram $\diag{\ref{diag:lg_k}}$ is a fiber square by Lemma \ref{lem:lf_lambda}. (Recall that $B_2= \mathrm{LF}_{0,1}(0;V)$.) Although the right square in the diagram is not a fiber square, it still commutes. Since $\alpha$ is flat, Equation \eqref{eqn:lg_k_ind_2} is true by the base change formula (cf. \cite[Proposition 3.18]{thomason1990higher}).
\end{proof}

%%%%%%%%%%%%%%%%%%%% On the induction theorem %%%%%%%%%%%%%%%%%%
\section{On the induction theorem}
The aim of this section is to use Theorem \ref{thm:two_step_blow_up} to decompose the Hermitian $K$-theory of Lagrangian Grassmannians into smaller pieces using localization sequences. Throughout this section, we may assume that $V_i$ are all trivial.
\subsection{Notation}

Let $L$  be a line bundle over the regular scheme $S$. It might be more organized to write the proof of the induction theorem using some notation of generalized Lagrangian flags.

\subsubsection{} We introduce the following notation on the pairs of schemes with line bundles:

$$\mathbb{LF}_i := (\LF_i, L), \quad \widetilde{\mathbb{LF}}_i := (\LF_i, L \otimes \Delta)  $$

\subsubsection{} We will use the canonical embeddings $ \iota_1': \LF_1 \to \LF_2 $ and the following natural projections
\[\begin{array}{lll}
        \hat{\pi}:\mathrm{LF}_{0,2}(0) \to \LF_0       & \grave{\pi}:\mathrm{LF}_{0,2}(0)_2 \to \LF_0       & \pi':\mathrm{LF}_{1,2}(1) \to \LF_1       \\
        \hat{\alpha}: \mathrm{LF}_{0,2}(0) \to \LF_{2} & \grave{\alpha}: \mathrm{LF}_{0,2}(0)_2 \to \LF_{2} & \alpha': \mathrm{LF}_{1,2}(1) \to \LF_{2}
    \end{array}\]
\subsubsection{}\label{sect:padding}    We will also need the idea of \textit{padding} to construct the maps. Let $\mathring{V}:= V \bot \SO^{\oplus 2}$ with $\SO^{\oplus 2}$ equipped with the trivial symplectic form. Construct a complete flag
\[
    \mathring{V}_{\bullet} = \{0=\mathring{V}_{0} \subset \mathring{V}_{1} \subset \mathring{V}_{2} \subset \ldots \subset \mathring{V}_{n} \subset \mathring{V}_{n+1} = \mathring{V}_{n+1}^{\bot} \subset \mathring{V}_{n}^{\bot} \subset \ldots \subset \mathring{V}_{1}^{\bot} \subset \mathring{V}_{0}^{\bot} = \mathring{V}\},
\]
with $\mathring{V}_i = V_{i-1} \perp \SO $. Note that $\mathring{V}^{1} \cong V$, and we then define
$       \mathrm{L}\mathring{\mathrm{F}}_{\mathbf{d}}(\mathbf{e})_{\mathbf{t}}:= \mathrm{L}{\mathrm{F}}_{\mathbf{d}}(\mathbf{e};\mathring{V}_\bullet)_{\mathbf{t}}$ and $
    \mathrm{L}\mathring{\mathrm{F}}_{\mathbf{d}}(\mathbf{e}):= \mathrm{L}\mathring{\mathrm{F}}_{\mathbf{d}}(\mathbf{e})_{\mathbf{t}}$ if $\mathbf{t} = (1, \ldots, 1)$. Moreover, we will use the natural projections
\[\begin{array}{lll}
        \bar{\pi}:\mathrm{L}\mathring{\mathrm{F}}_{1,2}(0) \to \LFo_1 \cong \LF_0      & \acute{\pi}:\mathrm{L}\mathring{\mathrm{F}}_{1,3}(0)_2 \to \LFo_1 \cong \LF_0       & \bar{\pi}':\LFo_{2,3}(1) \to \LFo_2\cong \LF_1       \\
        \bar{\alpha}:\mathrm{L}\mathring{\mathrm{F}}_{1,2}(0) \to \LFo_{2} \cong \LF_1 & \acute{\alpha}: \mathrm{L}\mathring{\mathrm{F}}_{1,3}(0)_2 \to \LFo_{3} \cong \LF_2 & \bar{\alpha}': \LFo_{2,3}(1) \to \LFo_3\cong \LF_{2}
    \end{array}\]
where we use the canonical isomorphism $\rho: \mathrm{LF}_{i+1}(~;\mathring{V}_\bullet) \cong \LF_{i}$.

\subsection{The induction theorem}

The following theorem is the key for the proof of the main theorem.
\begin{theorem}\label{thm:lg_induc}
    Let $n\geq 2$. The following statements hold true:
    \begin{enumerate}[leftmargin=20pt,label={\rm (\alph*)},ref=(\alph*)]
        \item Suppose that $n$ is even, then the map
              \[     (\iota_{1*} \,\,\, \pi_*\alpha^*):   \begin{array}{cll}
                      \GW^{[m-n]}(\mathbb{LF}_{1}) \oplus \GW^{[m]}(\widetilde{\mathbb{LF}}_{1})  \xrightarrow{  }
                      \GW^{[m]}(\widetilde{\mathbb{LF}}_0)
                  \end{array}
              \]
              is an equivalence.
        \item Suppose that $n$ is even, then the map
              \[
                  (\iota_{1*} \,\,\, \bar{\pi}_*\bar{\alpha}^*): \GW^{[m-n]}(\widetilde{\mathbb{LF}}_{1}) \oplus  \GW^{[m]}(\mathbb{LF}_{1})      \xrightarrow{ } \GW^{[m]}(\mathbb{LF}_{0})
              \]
              is an equivalence.
        \item Suppose that $n$ is odd, then the map
              \[
                  (\iota_{*} \,\,\, \hat{\pi}_* H \hat{\alpha}^* \,\,\, \acute{\pi}_* \acute{\alpha}^*):  \GW^{[m-2n+1]}(\mathbb{LF}_{2}) \oplus K(\LF_{2}) \oplus \GW^{[m]}(\mathbb{LF}_{2}) \xrightarrow{ } \GW^{[m]}(\mathbb{LF}_0)
              \]
              is an equivalence.
        \item Suppose that $n$ is odd, then the map
              \[
                  (\iota_{*} \,\,\, \hat{\pi}_* H \hat{\alpha}^* \,\,\,  \grave{\pi}_* \grave{\alpha}^*): 	\GW^{[m-2n+1]}(\widetilde{\mathbb{LF}}_{2}) \oplus K(\LF_{2}) \oplus \GW^{[m]}(\widetilde{\mathbb{LF}}_{2})
                  \xrightarrow{ }  \GW^{[m]}(\widetilde{\mathbb{LF}}_0)
              \]
              is an equivalence.
    \end{enumerate}
\end{theorem}

\begin{proof}  Proposition \ref{prop:blow-up-two-step-setup} implies that Lagrangian Grassmannians fit into Hypothesis \ref{sect:hypothesis}, and further observe that $c_1 = n$ and $c_2 = 2$.

    (a)   Note that $\lambda_1(\Delta) = \lambda_2(\Delta) =1 $ by Remark \ref{rmk:lg_lambda}. It follows that $\lambda_1(\Delta) \equiv c_1-1 \mod 2 $ and $\lambda_2(\Delta) \equiv c_2-1 \mod 2 $. Therefore,  Theorem \ref{thm:two_step_blow_up}(i) can be applied in this case. Hence, the following localization sequence
    \[
        \GW^{[m-n]}(\mathbb{LF}_1) \xrightarrow{\iota_{1\ast}} \GW^{[m]}(\widetilde{\mathbb{LF}}_0) \xrightarrow{(\vartheta^*)^{-1} v_{1}^*} \GW^{[m]}(\widetilde{\mathbb{LF}}_1)
    \]
    is split with a right splitting given by $\pi_*\alpha^*$.

    (b) Although this case is excluded from Theorem \ref{thm:two_step_blow_up}, we can still construct an explicit splitting using generalized Lagrangian flag schemes. In view of Section \ref{sect:padding}, we  form the following diagram
    \begin{equation}\label{eqn:n-odd-hermitian-X(1)}
        \xymatrix{
        \LF_0  & \ar[l]_-{\bar{\pi}} \mathrm{L}\mathring{\mathrm{F}}_{1,2}(0) \ar[r]^-{\bar{\alpha}} & \LF_1 \\
        U_1  \ar[u]^-{v_1} \ar@{=}[r] & U_1 \ar[u]^-{v_a} \ar[ur]_-{\vartheta}
        }
    \end{equation}
    where $v_a$ is the open complement $\mathrm{L}\mathring{\mathrm{F}}_{1,2}(0) - \mathrm{L}\mathring{\mathrm{F}}_{2,2}(0)$ under the following identification   $$U_1 = \LF_0 - \LF_1 \cong \mathrm{L}\mathring{\mathrm{F}}_1 - \mathrm{L}\mathring{\mathrm{F}}_2 \cong   \mathrm{L}\mathring{\mathrm{F}}_{1,2}(1) - \mathrm{L}\mathring{\mathrm{F}}_{2,2}(1) \cong  \mathrm{L}\mathring{\mathrm{F}}_{1,2}(0) - \mathrm{L}\mathring{\mathrm{F}}_{2,2}(0), $$
    and where the left square is a fiber square.
    By Example \ref{exm:LF012-132-022}, the line bundle $$\omega_{\bar{\pi}} \cong \omega_{\mathrm{L}\mathring{\mathrm{F}}_{1,2}(0)/S} \otimes  \omega_{\mathrm{L}\mathring{\mathrm{F}}_{1}/S}^\vee$$
    is trivial in the Picard group of $\mathrm{L}\mathring{\mathrm{F}}_{1,2}(0)$ modulo two. Therefore, the composition
    \[\GW^{[m]}(\mathbb{LF}_1) \xrightarrow{\bar{\alpha}^*} \GW^{[m]}(\mathrm{L}\mathring{\mathrm{F}}_{1,2}(0), \omega_{\bar{\pi}}) \xrightarrow{\bar{\pi}_*} \GW^{[m]}(\mathbb{LF}_0) \]
    is well-defined. The result follows from the fact that $\bar{\pi}_*\bar{\alpha}^*$ is a right splitting of the following localization sequence
    \[
        \GW^{[m-n]}(\widetilde{\mathbb{LF}}_1) \xrightarrow{\iota^{}_{1\ast}} \GW^{[m]}(\mathbb{LF}_0) \xrightarrow{(\vartheta^*)^{-1} v_{1}^{*}} \GW^{[m]}(\mathbb{LF}_1),
    \]
    which can be proved by the identities
    $
        v_1^{*}  \bar{\pi}_* \bar{\alpha}^*  = v_a^* \bar{\alpha}^*	=  \vartheta^{*},
    $
    where we use the base change formula (cf. \cite[Section 5]{huang2023the})) for the first identity.

    The proofs of (c) and (d) are more complicated and will be proved below.
\end{proof}

\subsection{Proof of Theorem \ref{thm:lg_induc} (c)}
Suppose that
\begin{align}\label{eqn:lg_3_loc}
    \xymatrix@C=50pt{
    \mathbf{X}  \ar[r]^-{\iota_{1\ast}}     & \mathbf{Y} \ar[r]^-{(\vartheta^*)^{-1} v^*} & \mathbf{Z} \ar[r]^-{\partial}  & \mathbf{X}[1]} \\
    \label{eqn:lg_4_bott}
    \xymatrix@C=50pt{\mathbf{A} \ar[r]^-{F} & \mathbf{B} \ar[r]^-{H}                      & \mathbf{C} \ar[r]^-{\eta \cup} & \mathbf{A}[1]}
\end{align}
are the localization sequence and the algebraic Bott sequence, respectively, with
\begin{align*}
     & \mathbf{X} = 	\GW^{[m-n]}(\widetilde{\mathbb{LF}}_1), \quad	\mathbf{Y} =  \GW^{[m]}(\mathbb{LF}_0),  \quad	\mathbf{Z} =  \GW^{[m]}(\mathbb{LF}_1) ; \\
     & \mathbf{A} = \GW^{[m-n]}(\widetilde{\mathbb{LF}}_2), \quad \mathbf{B} = K(\LF_2), \quad \mathbf{C} = \GW^{[m-n+1]}(\widetilde{\mathbb{LF}}_2).
\end{align*}
Theorem \ref{thm:two_step_blow_up} (ii) applies to this case ($\lambda_1(\SO) \equiv c_1 -1 \mod 2, \, \lambda_2(\SO) \equiv c_2 \mod 2$). Hence, $\partial = \eta \cup \tilde{\pi}_{\ast}  \tilde{\alpha}^{\ast}.$
Sublemma \ref{sublem:case-(c)} below shows that there exist a unit $\langle u \rangle \in \GW_0^{[0]}(
    \mathbb{LF}_2
    )$ such that the following diagram
\begin{equation}\label{eq:map-of-dis-trian}
    {\xymatrix{
    \mathbf{D} \oplus \mathbf{A} \ar[rr]^-{\tiny\begin{pmatrix} 1 & 0  \\ 0 & F \\ 0 & 0 \end{pmatrix}} \ar[d]^-{\tiny\begin{pmatrix} \mathbf{d} & \mathbf{a} \end{pmatrix}} && \mathbf{D} \oplus \mathbf{B} \oplus \mathbf{E} \ar[rr]^-{\tiny\begin{pmatrix} 0 & H & 0 \\ 0 & 0 & \langle u\rangle \end{pmatrix}}  \ar[d]^-{\tiny\begin{pmatrix} \mathbf{g} & \mathbf{f} & \mathbf{h} \end{pmatrix}}  && \mathbf{C} \oplus \mathbf{E} \ar[rr]^-{\tiny\begin{pmatrix} 0 & 0  \\ \eta \cup & 0 \end{pmatrix}} \ar[d]^-{\tiny\begin{pmatrix} \mathbf{c} & \mathbf{e} \end{pmatrix}} &&   \mathbf{D}[1] \oplus \mathbf{A}[1] \ar[d]^-{\tiny\begin{pmatrix} \mathbf{d}[1] & \mathbf{a}[1] \end{pmatrix}} \\
    \mathbf{X}  \ar[rr]_-{\iota_{1*}} && \mathbf{Y} \ar[rr]_-{(\vartheta^*)^{-1} v_1^*} && \mathbf{Z} \ar[rr]_-{\partial} && \mathbf{X}[1].
    }}
\end{equation}
is a map of distinguished triangles in $\mathcal{SH}$ where
\begin{align*}
     & \mathbf{D} = \GW^{[m-2n+1]}(\mathbb{LF}_2), \quad\mathbf{E} = \GW^{[m]}(\mathbb{LF}_2),
\end{align*}
and where
\begin{align*}
     & \mathbf{a}=\pi'_{\ast} \alpha'^{*},\quad \mathbf{c} = \iota_{1*}',\quad \mathbf{d} = \iota_{1*}',\quad \mathbf{e} = \bar{\pi}'_* \bar{\alpha}'^{*}, \quad
    \mathbf{f}=\hat{\pi}_*H\hat{\alpha}^*,\quad \mathbf{g} = \iota_*,\quad \mathbf{h}=\acute{\pi}_* \acute{\alpha}^*.
\end{align*}
By Example \ref{exm:LF012-132-022}, the map $\mathbf{h}$ is well defined.\ The maps $(\mathbf{d} \quad \mathbf{a})$ and $(\mathbf{c} \quad \mathbf{e})$ are both equivalences by case (a) and (b) above, hence the middle map $(\mathbf{g} \quad \mathbf{f}\quad \mathbf{h})$ is also an equivalence.

\begin{sublemma}\label{sublem:case-(c)}
    The three squares in Diagram \eqref{eq:map-of-dis-trian} commute up to homotopy.
\end{sublemma}
\begin{proof}
    (i) The left square in Diagram \eqref{eq:map-of-dis-trian} commutes. It is enough to show that the identity
    $$\iota_{1\ast} \circ \mathbf{a} = \mathbf{f} \circ F \quad(\textnormal{i.e.\ } \iota_{1\ast} \pi'_{\ast} \alpha'^{*} = \hat{\pi}_*H\hat{\alpha}^* F) $$
    holds, which is a direct consequence of Lemma \ref{lma:lf_hat_pic} and Lemma \ref{lma:fh_codim_one}, since $\hat{\alpha}^* F = F \hat{\alpha}^*$.

    (ii) Let us show that the middle square in Diagram \eqref{eq:map-of-dis-trian} commutes. The formula
    $$v_{1}^* \circ \mathbf{f} = \vartheta^* \circ  \mathbf{c}\circ H, \quad (\textnormal{i.e.}\quad v_{1}^* \hat{\pi}_*H\hat{\alpha}^* = \vartheta^* \iota'_{1*} H) $$
    holds and can be verified through the following commutative diagram
    $$\xymatrix@C=40pt{
        &\LF_2 \ar@{=}[dr] \\
        \mathrm{LF}_{0,2}(0) \ar[d]_-{\hat{\pi}} \ar[ur]^-{\hat{\alpha}} \ar@{}[dr]|{\diagram \label{diag:lg_odd_so_car_2}} & V \ar[d] \ar[r] \ar[l] \ar@{}[dr]|{\diagram \label{diag:lg_odd_so_car_3}} & \LF_2 \ar[d]_-{\iota_{1}'}\\
        \LF_0
        %			\ar@{}[dr]|{\diagram \label{diag:lg_odd_so_car_4}} 
        & U_{1} \ar[l]_-{v_1} \ar[r]^-{\vartheta}
        %			\ar@{}[dr]|{\diagram \label{diag:lg_odd_so_car_5}}  
        & \LF_1
        }
    $$
    where Diagram $\diag{\ref{diag:lg_odd_so_car_2}}$ and $ \diag{\ref{diag:lg_odd_so_car_3}}$ are both fiber squares.
    We observe, by \cite[Lemma 3.15]{huang2023the}, that
    \[ v_{1}^* \hat{\pi}_*H\hat{\alpha}^* = Hv_{1}^* \hat{\pi}_*\hat{\alpha}^*,  \quad  \vartheta^* \iota'_{1*} H = H\vartheta^* \iota'_{1*}.   \]
    Moreover, by applying the base change formula (cf. \cite[Proposition 3.18]{thomason1990higher}) to Diagram $\diag{\ref{diag:lg_odd_so_car_2}}$ and $\diag{\ref{diag:lg_odd_so_car_3}}$, we have $v_{1}^* \hat{\pi}_*\hat{\alpha}^* = \vartheta^* \iota'_{1*} $ in the $K$-theory $K(U_1)$ .\

    To complete the proof, we must show that the following formula holds in $\GW^{[m]}(U_1)$:
    $$v_1^*\circ \mathbf{h} = \vartheta^* \circ \mathbf{e} (\langle u\rangle \cup -), \quad ( \textnormal{i.e.\ } \quad v_{1}^* \acute{\pi}_* \acute{\alpha}^* = \vartheta^* \bar{\pi}'_* \bar{\alpha}'^{*} (\langle u\rangle \cup -)).$$
    Consider the following diagram
    $$
        \xymatrix@C=40pt{
        \LF_0 & U_{1} \ar[l]_-{v_1} \ar@{=}[r]& U_{1} \ar[r]^-{\vartheta} & \LF_1 \\
        \mathrm{L}\mathring{\mathrm{F}}_{1,3}(0)_2 \ar[u]^-{\acute{\pi}} \ar[ddr]_-{\acute{\alpha}} & {U}_c \ar[u]^-{\pi_a} \ar[l]_-{v_c} & {U}_d \ar[u]^-{\pi_b}\ar[r]^-{\vartheta_d} & \mathrm{L}\mathring{\mathrm{F}}_{2,3}(1) \ar[u]_-{\bar{\pi}'} \ar[ddl]^-{\bar{\alpha}'} \\
        & W \ar@{=}[r] \ar[u]^-{\pi_c} \ar[d]_-{\wp} & W \ar[u]^-{\pi_d} \ar[d]_-{\wp}\\
        & \LF_2 \ar@{=}[r] & \LF_2
        }
    $$
    where the top left and right squares are both fiber squares. Note that $U_1$ has different models
    \[U_1:= \LF_0- \LF_1  \cong  \LFo_1 -\LFo_2 \cong \LFo_{1,2}(1) - \LFo_{2,2}(1)\]
    where the latter is the one that we really want to use. Let $\mathring{n}:= n+1$. Let $P_{\mathring{n}-1}^{(0)}$, $L_{\mathring{n}}^{(0)}$ and $L_{\mathring{n}}^{(1)}$ be universal bundles on $U_1$ that satisfy
    \begin{equation}\label{eq:4.6'}
        \mathring{V}_2^\perp \not\supset  L_{\mathring{n}}^{(0)} \supset P_{\mathring{n}-1}^{(0)} \subset L_{\mathring{n}}^{(1)} \subset \mathring{V}_2^\perp, \quad  \mathring{V}_1 \subset  P_{\mathring{n}-1}^{(0)}
    \end{equation}
    where $P_{\mathring{n}-1}^{(0)} = L_{\mathring{n}}^{(0)} \cap \mathring{V}_2^\perp$.
    Let $P_{\mathring{n}-1}^{(1)}$, $L_{\mathring{n}}^{(1)}$ and $\tilde{L}_{\mathring{n}}^{(2)}$ denote universal bundles over $\LFo_{2,3}(1)$ such that
    \begin{equation}\label{eq:4.7'}
        \mathring{V}_1 \subset   P_{\mathring{n}-1}^{(1)} \subset  L_{\mathring{n}}^{(1)} \subset \mathring{V}_2^\perp,  \quad  P_{\mathring{n}-1}^{(1)} \subset \tilde{L}_{\mathring{n}}^{(2)} \subset \mathring{V}_3^\perp.  \end{equation}
    Combining the relations \eqref{eq:4.6'} and \eqref{eq:4.7'}, we see that $U_d = U_1 \times_{\LF_1} \LFo_{2,3}(1) $ is the scheme with universal bundles $P_{\mathring{n}-1}^{(0)}$, $L_{\mathring{n}}^{(0)}$,  $L_{\mathring{n}}^{(1)}$,  $P_{\mathring{n}-1}^{(1)}$, and $\tilde{L}_{\mathring{n}}^{(2)}$ such that
    \begin{equation}\label{eq:4.8'}  	\mathring{V}_2^\perp \not\supset  L_{\mathring{n}}^{(0)} \supset P_{\mathring{n}-1}^{(0)} \subset L_{\mathring{n}}^{(1)} \subset \mathring{V}_2^\perp, \quad    P_{\mathring{n}-1}^{(0)}  \supset 	\mathring{V}_1 \subset   P_{\mathring{n}-1}^{(1)} \subset  L_{\mathring{n}}^{(1)} ,  \quad  P_{\mathring{n}-1}^{(1)} \subset \tilde{L}_{\mathring{n}}^{(2)} \subset \mathring{V}_3^\perp.\end{equation}
    Furthermore, let $P_{\mathring{n}-2}^{(0)}$, $L_{\mathring{n}}^{(0)}$ and ${L}_{\mathring{n}}^{(2)}$ denote universal bundles over $\LFo_{1,3}(0)_2$ such that
    \begin{equation}\label{eq:4.9'}
        \mathring{V}_1^\perp \supset  L_{\mathring{n}}^{(0)} \supset  P_{\mathring{n}-2}^{(0)} \subset {L}_{\mathring{n}}^{(2)} \subset \mathring{V}_3^\perp.  \end{equation}
    Combining the relations \eqref{eq:4.6'} and \eqref{eq:4.9'}, we see that $U_c = U_1 \times_{\LF_0} \LFo_{1,3}(0)_2 $ is the scheme with universal bundles $P_{\mathring{n}-1}^{(0)}$, $L_{\mathring{n}}^{(0)}$,  $L_{\mathring{n}}^{(1)}$,  $P_{\mathring{n}-2}^{(0)}$, and ${L}_{\mathring{n}}^{(2)}$ such that
    \begin{equation}\label{eq:4.10'}  		\mathring{V}_2^\perp \not\supset  L_{\mathring{n}}^{(0)} \supset P_{\mathring{n}-1}^{(0)} \subset L_{\mathring{n}}^{(1)} \subset \mathring{V}_2^\perp, \quad  \mathring{V}_1 \subset  P_{\mathring{n}-1}^{(0)}, \quad L_{\mathring{n}}^{(0)} \supset  P_{\mathring{n}-2}^{(0)} \subset {L}_{\mathring{n}}^{(2)} \subset \mathring{V}_3^\perp
    \end{equation}
    where we note the additional relation  $P_{\mathring{n}-2}^{(0)} \subset P_{\mathring{n}-1}^{(0)} \subset L_{\mathring{n}}^{(0)}$.
    Define $W$ to be the closed subscheme of $U_c \times_{U_1} U_b$ such that $\tilde{L}_{\mathring{n}}^{(2)} = {L}_{\mathring{n}}^{(2)}$.

    Alternatively, the scheme $W$ can be viewed as the closed subscheme of $\Gr_{{U}_d}(\mathring{n}-2, P_{\mathring{n}-1}^{(0)})$ determined by the zero locus of the morphism
    $
        s_1:P_{\mathring{n}-2}^{(0)} \to (P_{\mathring{n}-1}^{(1)})^{\perp} / L_{\mathring{n}}^{(2)}
    $, or equivalently the closed subscheme of $\Gr_{{U}_c}^{\mathring{V}_1}(\mathring{n}-1,L_{\mathring{n}}^{(2)})$ defined by
    $
        s_2:P_{\mathring{n}-1}^{(1)}/\mathring{V}_1 \to (P_{\mathring{n}-2}^{(0)})^{\perp} / (P_{\mathring{n}-1}^{(0)})^\perp
    $. (To see that $s_1$ and $s_2$ are locally defined by regular sequences, we prove that $W \to {U}_c$ and $W\to U_d$ are both of relative dim $0$. Note that $\dim {U}_c = \dim {U}_d = \dim U_1$.  On the one hand, $W$ is a closed subscheme of the fiber product ${U}_c \times_{U_1} {U}_d $, and therefore  each irreducible component of $W$ has dimension less or equal to $ \dim {U}_d$. On the other hand, the scheme $W$ is a closed subscheme of $\Gr_{{U}_d}(n-2,P_{\mathring{n}-1}^{(0)})$ defined by the morphism $s_1$ which can be regarded as a section of a rank $n-2$ bundle. Hence, the dimension of each irreducible component of $W$ must be greater or equal to $\dim {U}_d$.) 	Therefore, we obtain
    \begin{align*}
        \omega_{W/{U}_d} & = \det (P_{\mathring{n}-2}^{(0)})^{\otimes \mathring{n}} \det (P_{\mathring{n}-1}^{(0)})^{\otimes 2 - \mathring{n}} \det (P_{\mathring{n}-1}^{(1)})^{\otimes 2 - \mathring{n}} \det (L_{\mathring{n}}^{(2)})^{\otimes \mathring{n}-2},  \\
        \omega_{W/{U}_c} & = \det (P_{\mathring{n}-2}^{(0)})^{\otimes 2 - \mathring{n}} \det (P_{\mathring{n}-1}^{(0)})^{\otimes \mathring{n} - 2} \det (P_{\mathring{n}-1}^{(1)})^{\otimes \mathring{n}} \det (L_{\mathring{n}}^{(2)})^{\otimes 2 - \mathring{n}}
    \end{align*}
    which implies that $\omega_{W/{U}_d} $ and $\omega_{W/{U}_c} $ are both trivial in the Picard group of $W$ modulo two. Applying the base change formula \cite[Theorem 5.2, Corollary 9.13]{huang2023the} and Theorem \ref{thm:flat_pullback_finite_pushforward}, we see that
    \begin{equation}\label{eq:pull-push}
        v_{1}^* \acute{\pi}_* \acute{\alpha}^* =  (\pi_a\circ \pi_{c})_* \big( \wp^*(-) \cup \langle a \rangle \big), \quad  \quad \vartheta^* \bar{\pi}'_* \bar{\alpha}'^{*} = (\pi_b\circ \pi_{d})_* \big( \wp^*(-) \cup \langle b \rangle\big)
    \end{equation}
    for some units $a,b \in \mathbb{G}_m(W) $ by Corollary \ref{cor:push-pull-identiy}. Note that $\pi_a\circ \pi_{c} = \pi_b\circ \pi_{d}$. The result follows if we can show that $\wp^*(-) \cup \langle a \rangle  = \wp^*(- \cup \langle a' \rangle)   $ for some $a' \in \mathbb{G}_m(\LF_2)$. It suffices to prove the surjectivity of the maps $\wp^*: \mathbb{G}_m(\LF_2) \to \mathbb{G}_m(W)$ on group of multiplicative units. Let us argue by the following commutative diagram
    \[ \xymatrix{\mathbb{G}_m(S) \ar[d] \ar[r] & \mathbb{G}_m(\LFo_{2,3}(1)) \ar[r]^-{\vartheta_d^*} &  \mathbb{G}_m(U_d) \ar[d]^-{\pi_d^*} \\
        \mathbb{G}_m(\LF_2)  \ar[rr]^-{\wp^*} && \mathbb{G}_m(W) } \]
    where the two unlabeled maps are bijective by Lemma \ref{lma:lf_de_structure_sheaf}. The map $\vartheta_d^*$ is bijective since $\vartheta_d$ is an affine bundle over a reduced scheme (cf. Lemma \ref{lma:affine_bundle_unit}). The map $\pi_d^*$ is bijective since it is a local complete intersection inside a projective bundle, (cf. Lemma \ref{lma:push-pull-identiy-2}). Hence, the map  $\wp^*$ is bijective.
    \footnote{
        It is worth noting that the unit $\langle a \rangle$ appearing in $\GW^{[0]}_0(W)$ represents a specific type of alignment discussed in \cite{balmer2012bases}. Although $W$ is not regular and therefore lies beyond the scope of \textit{loc. cit.}, the technique developed there remains applicable due to the surjectivity of the map $\wp^* : \mathbb{G}_m(\LF_2) \to \mathbb{G}_m(W)$.
    }

    (iii) The right square in Diagram \eqref{eq:map-of-dis-trian} commutes. Observe that $ \partial \mathbf{e} = \partial (\vartheta^*)^{-1} v_1^* \mathbf{h} = 0$. Therefore, it suffices to show that the formula
    $$	\tilde{\pi}_* \tilde{\alpha}^* \iota'_{1*} =  \pi'_* \alpha'^{*}$$
    holds.  Form the following diagram
    $$
        \xymatrix@C=40pt{
        \LF_1 \ar@{=}[d] & \LF_{1,2}(1) \ar[l]_-{\pi'} \ar[r]^-{\alpha'} \ar[d] \ar@{}[dr]|{\diagram \label{diag:lg_odd_so_car_6}} & \LF_2 \ar[d]_-{\iota_{1}'}\\
        \LF_1 & E_2 \ar[l]_-{\tilde{\pi}} \ar[r]^- {\tilde{\alpha}} & \LF_1
        }
    $$
    where all the squares are commutative, $\hat{\iota}$ is a codimension one regular closed embedding, and Diagram $\diag{\ref{diag:lg_odd_so_car_6}}$ is a fiber square. The key point is to apply the usual base change formula (cf. \cite[Corollary 9.13]{huang2023the}) applied to Square $\diag{\ref{diag:lg_odd_so_car_6}}$.
\end{proof}

\subsection{Proof of Theorem \ref{thm:lg_induc} (d)}
The proof of this case is very similar to the case (c) above. Let us keep using the symbols in \eqref{eqn:lg_3_loc},  \eqref{eqn:lg_4_bott} and \eqref{eq:map-of-dis-trian}, but instead with the following notation
\begin{align*}
     & \mathbf{X} = 	\GW^{[m-n]}(\mathbb{LF}_1), \quad	\mathbf{Y} =  \GW^{[m]}(\widetilde{\mathbb{LF}}),  \quad	\mathbf{Z} =  \GW^{[m]}(\widetilde{\mathbb{LF}}_1) ; \\
     & \mathbf{A} = \GW^{[m-n]}(\mathbb{LF}_2), \quad \mathbf{B} = K(\LF_2), \quad \mathbf{C} = \GW^{[m-n+1]}(\mathbb{LF}_2);                                        \\
     & \mathbf{D} = \GW^{[m-2n+1]}(\widetilde{\mathbb{LF}}_2), \quad  \mathbf{E} = \GW^{[m]}(\widetilde{\mathbb{LF}}_2);
\end{align*}
\[\mathbf{a} = \bar{\pi}_* \bar{\alpha}^*,\quad \mathbf{d} = \iota'_{1*}, \quad \mathbf{c} = \iota'_{1*},\quad \mathbf{e} = \pi'_{\ast} \alpha'^{*}, \quad \mathbf{f}=\hat{\pi}_*H\hat{\alpha}^*,\quad \mathbf{g} = \iota_*,\quad \mathbf{h}=\grave{\pi}_* \grave{\alpha}^*.\]
Theorem \ref{thm:two_step_blow_up} (iii) applies to this case ($\lambda_1(\Delta) \equiv c_1  \mod 2, \, \lambda_2(\Delta) \equiv c_2-1 \mod 2$), and we get
\[ \partial = \eta \cup \tilde{\pi}_{2*}\tilde{\iota}_1^*\pi_{2*} \alpha^*= \eta \cup \check{\pi}_* \check{\alpha}^* \]
where the last equality is obtained by the base change formula.
Under these new notation, we need to prove that Sublemma \ref{sublem:case-(c)} still holds, and the result follows. (i) The left square of Diagram \eqref{eq:map-of-dis-trian} commutes (i.e. $	\iota_{1\ast} \bar{\pi}_* \bar{\alpha}^* = \hat{\pi}_*H\hat{\alpha}^* F \label{eqn:lg_odd_delta_1}$) for a similar reason as in case (c) in view of the following diagram
\[
    \xymatrix@C=40pt{
    \LF_0& \mathrm{LF}_{0,2}(0) \ar[l]_-{\hat{\pi}} \ar[r]^-{\hat{\alpha}} & \LG^{2}\\
    \LF_1 \ar[u]^-{\iota_1} & \mathrm{L}{\mathrm{F}}_{1,2}(0) \ar[l]_-{\bar{\pi}} \ar[r]^-{\bar{\alpha}} \ar[u]^-{\iota} & \LF_2 \ar@{=}[u]
    }
\]
(ii) The middle square of Diagram \eqref{eq:map-of-dis-trian} also commutes (i.e. $	v_{1}^* \hat{\pi}_*H\hat{\alpha}^* = \vartheta^* \iota'_{1*} H $ and $v_{1}^* \grave{\pi}_* \grave{\alpha}^* = \vartheta^* \pi'_{\ast} \alpha'^{*}$). The first formula is deduced similarly to the proof of Sublemma \ref{sublem:case-(c)}. The second formula can be obtained from the following diagram
\[
    \xymatrix{
    \LF_0 & U_{1} \ar[l]_-{v_1} \ar@{=}[r]& U_{1} \ar[r]^-{\vartheta} & \LF_1 \\
    \mathrm{LF}_{0,2}(0)_2 \ar[u]^-{\grave{\pi}} \ar[dr]_-{\grave{\alpha}} & {U}_a \ar[r]^-{\Pi} \ar[d]_-{\alpha_a}\ar[u]^-{\pi_a} \ar[l]_-{v_a} & U_b \ar[d]_-{\alpha_b}  \ar[u]^-{\pi_b}\ar[r]^-{\vartheta_b} & \mathrm{LF}_{1,2}(1) \ar[u]_-{{\pi}'} \ar[dl]^-{{\alpha}'} \\
    & \LF_2 \ar@{=}[r] & \LF_2
    }
\]
where the top left and right squares are both fiber squares. Then we note that $U_1 := \LF_0- \LF_1 \cong \LF_{0,1}(0) - \LF_{1,1}(0) $. In the latter model, we let $P_{{n}-1}^{(0)}$, $L_{{n}}^{(0)}$ and $L_{{n}}^{(1)}$ be the universal bundles on $U_1$ satisfying the relation
\begin{equation}\label{eq:4.6}
    V_1^\perp \not \supset  L_{{n}}^{(0)} \supset P_{{n}-1}^{(0)} \subset L_{{n}}^{(1)} \subset V_1^\perp.
\end{equation}
Recall that $ P_{n-1}^{(0)} = L_{{n}}^{(0)} \cap V_1^\perp  $ and $L_n^{(1)} =(P_{n-1}^{(0)})^\perp \cap V_1^\perp $.
Let $P_{{n}-1}^{(1)}$, $L_{{n}}^{(1)}$ and $L_{{n}}^{(2)}$ be universal bundles over $\LF_{1,2}:= \LF_{1,2}(1)$ such that
\begin{equation}\label{eq:4.7}
    V_1 \subset   P_{{n}-1}^{(1)} \subset  L_n^{(1)} \subset V_1^\perp,  \quad  P_{{n}-1}^{(1)} \subset L_n^{(2)} \subset V_2^\perp.  \end{equation}
Combining the relations \eqref{eq:4.6} and \eqref{eq:4.7}, we see that $U_b = U_1 \times_{\LF_1} \LF_{1,2}(1) $ is the scheme with universal bundles $P_{{n}-1}^{(0)}$, $L_{{n}}^{(0)}$,  $L_{{n}}^{(1)}$,  $P_{{n}-1}^{(1)}$, and $L_{{n}}^{(2)}$ such that
\begin{equation}\label{eq:4.8}    V_1^\perp \not \supset  L_{{n}}^{(0)} \supset P_{{n}-1}^{(0)} \subset L_{{n}}^{(1)} \subset V_1^\perp,  \quad  V_1 \subset   P_{{n}-1}^{(1)} \subset  L_n^{(1)} \subset V_1^\perp,  \quad  P_{{n}-1}^{(1)} \subset L_n^{(2)} \subset V_2^\perp. \end{equation}
Furthermore, ${U}_a = \LF_{0,2}(0)_2 - \LF_{1,2}(0)_2 $ is isomorphic to $(\LF_{0,1}(0) - \LF_{1,1}(0)) \times_{\LF_0} \mathrm{LF}_{0,2}(0)_2 $ with universal bundles $ L_n^{(0)}$, $P_{n-2}^{(0)}$, $L_n^{(1)} $, $P_{{n}-1}^{(0)}$, and $L_{{n}}^{(2)}$ such that
\begin{equation}\label{eq:4.10} V_1^\perp \not \supset L_n^{(0)} \supset P_{n-1}^{(0)} \supset P_{n-2}^{(0)} \subset L_n^{(2)} \subset V_2^\perp, \quad V_1^\perp \supset L_n^{(1)} \supset P_{n-1}^{(0)}  \end{equation}
Observe that, over $(\LF_{0,1}(0) - \LF_{1,1}(0)) \times_{\LF_0} \mathrm{LF}_{0,2}(0)_2 $, there are more  hidden relations that
\begin{equation}\label{eq:4.11} P_{n-2}^{(0)} \subset P_{n-1}^{(1)} \subset L_n^{(1)}, \quad V_1 \subset P_{n-1}^{(1)} \subset L_{n}^{(2)}.  \end{equation}
where $P_{n-1}^{(1)} = V_1 \oplus P_{n-2}^{(0)}$ is the unique admissible subbundle of rank $n-1$ satisfying these properties, since $V_1 \not \subset P_{n-2}^{(0)}$ and $V_1 \cap P_{n-2}^{(0)} = 0$. Construct the morphism $\Pi: {U}_a \to U_b$ by forgetting the strata $P_{n-2}^{(0)}$. The morphism $\Pi: {U}_a \to U_b$ is projective. To illustrate, consider the Grassman bundle $\Gr_{U_b}(n-2, P_{n-1}^{(0)})$ and its universal bundle $P_{n-2}^{(0)}$. Now, the scheme ${U}_a$ is the closed subscheme of $\Gr_{U_b}(n-2, P_{n-1}^{(0)})$ defined by the zero locus of the morphism
\begin{equation}\label{eq:4.12}  P_{n-2}^{(0)} \to L_n^{(1)}/ P_{n-1}^{(1)} \end{equation}
which is a regular section by our usual dimension count. Therefore, $\Pi$ is a projective morphism and
\begin{equation}\label{eq:4.13} \omega_{\Pi} = \det(P_{n-2}^{(0)})^{n-2} \cdot \det(P_{n-1}^{(0)})^{2-n} \cdot  \det (L_{n}^{(1)})^{n-2} \otimes \det(P_{n-1}^{(1)})^{2-n}  \end{equation}
By the relations \eqref{eq:4.10} and \eqref{eq:4.11}, we get a fiber square of vector bundles
\begin{equation}\label{eq:4.14}
    \xymatrix{
    P_{n-2}^{(0)} \ar[d] \ar[r] & P_{n-1}^{(0)} \ar[d]\\
    P_{n-1}^{(1)}  \ar[r] &  L_n^{(1)}
    }
\end{equation}
over ${U}_a$, and hence

\begin{equation}\label{eq:4.15}
    P_{n-1}^{(0)}/ P_{n-2}^{(0)}  \cong L_n^{(1)}/ P_{n-1}^{(1)}
\end{equation}

Putting \eqref{eq:4.13} and \eqref{eq:4.15} together, we  conclude that the canonical sheaf $\omega_\Pi$ is trivial in the Picard group of ${U}_a$ modulo two.
By Example \ref{exm:LF012-132-022}, we see that the maps $\grave{\pi}_*\grave{\alpha}^*$ and $\pi'_*(\alpha')^*$ are both well-defined. By the base change formula, we see that $(\pi_a)_*\alpha_a^* = v_1^* \grave{\pi}_*\grave{\alpha}^* $ and $(\pi_b)_*\alpha_b^* = \vartheta^* \pi'_*\alpha'^* $ are also well-defined. It follows that $(\pi_a)_*\alpha_a^* = (\pi_b)_*\alpha_b^* $, which is obtained from $\Pi_*\Pi^* = 1$ (cf. Remark  \ref{rmk:alignment-remark} or \cite[Proposition 3.15]{balmer2012witt}).

(iii) The right square of Diagram \ref{eq:map-of-dis-trian} commutes. This requires showing that $\check{\pi}_* \check{\alpha}^* \iota_{1*}' = \bar{\pi}_*\bar{\alpha}^*$ which is a consequence of the base change formula (cf. Theorem \ref{thm:flat_pullback_finite_pushforward}).

%%%%%%%%%%%%%%%%%%%% Proof of the main theorem: The additive basis %%%%%%%%%%%%%%%%%%%%
\section{Proof of the main theorem: The additive basis} \label{sec:proof_of_the_main_theorem}

\subsection{Almost even and \texorpdfstring{$K$}{K}-even shifted Young diagrams}
In this section, we provide alternative ways of describing almost even and $K$-even Young diagrams.
\begin{definition}[\cite{martirosian2021witt}]
    A shifted Young diagram $\Lambda$ in $\digamma_{n}$ is called \textit{almost even} if all the inner segments have even length, except that the last inner segment $s_{l_{\Lambda}}$ has odd length of $s_{l_{\Lambda}}$ is an inner segments. (Equivalently, the segments satisfy: (i) $|s_i|$ is even for every $i$ with $3 \leq i \leq l_\Lambda - 1$, (ii) $|s_2|$ is even if $|s_1| \neq 0 $, (iii) $|s_{l_\Lambda}|$ is odd). This is also equivalent to requiring that the index $w_\Lambda$ of $\Lambda$ (Definition \ref{def:shifted_young_index}) equals $l_{\Lambda}$.  Let $\mathfrak{A}_n$ denote the set of all almost even shifted Young diagrams in $\digamma_n$.
\end{definition}

\begin{definition}[See also: Figure \ref{fig:K_even}]\label{def:K_even}
    A shifted Young diagram $\Lambda$ in $\digamma_{n}$ is called \textit{$K$-even} if its index $w_\Lambda$ (Definition \ref{def:shifted_young_index}) is even. Equivalently, there exists an even integer $w=w_\Lambda$ such that
    \begin{enumerate}
        \item $|Oa_{t}| \equiv n-1 \pmod 2$ for each $2 \leq t \leq w$ (whenever $|Oa_t|\neq 0$), and
        \item $|Oa_{w+1}| \equiv n \pmod 2$.
    \end{enumerate}
    Let $\mathfrak{E}_n$ denote the set of all $K$-even shifted Young diagrams in $\digamma_n$.
\end{definition}

\begin{figure}[!ht]
    \begin{tikzpicture}[scale=0.3]
        \draw[blue, ultra thick] (19,16) -- (19,13);
        \draw[blue, ultra thick] (19,13) -- (11,13) -- (11,9) -- (5,9) -- (5,7) -- (3,7);

        \filldraw [red] (19,13) circle (10pt);
        \filldraw [teal] (11,13) circle (10pt);
        \filldraw [red] (11,9) circle (10pt);
        \filldraw [teal] (5,9) circle (10pt);
        \filldraw [red] (5,7) circle (10pt);
        \filldraw [yellow] (19,16) circle (10pt);

        \path (18.5,16.5)
        node[text=black,anchor=base west] {\tiny $O$};
        \path (18.8,14.3)
        node[text=black,anchor=base west] {\tiny $s_1$};
        \path (18.8,12.3)
        node[text=red,anchor=base west] {\tiny $a_{2}$};
        \path (10.7,13.3)
        node[text=teal,anchor=base east] {\tiny $a_{3}$};
        \path (13.7,12.1)
        node[text=black,anchor=base west] {\tiny $s_2$};
        \path (11,10.6)
        node[text=black,anchor=base west] {$\cdots$};
        \path (10.9,8.2)
        node[text=red,anchor=base west] {\tiny $a_{w_{\Lambda}}$};
        \path (6.7,8.2)
        node[text=black,anchor=base west] {\tiny $s_{w_{\Lambda}}$};
        \path (4.8,9)
        node[text=teal,anchor=base east] {\tiny $a_{w_{\Lambda}+1}$};

        \node[draw, align=center] at (12,4.2)
        {\tiny when $|Oa_{2}| = |s_{1}| \neq 0$ and $|Oa_{2}| \equiv n-1\pmod 2$};
    \end{tikzpicture} \quad
    \begin{tikzpicture}[scale=0.3]
        \draw[blue, ultra thick] (19,16) -- (17,16);
        \draw[blue, ultra thick] (17,16) -- (17,13) -- (11,13) -- (11,9) -- (5,9) -- (5,7) -- (3,7);

        \filldraw [teal] (17,16) circle (10pt);
        \filldraw [red] (17,13) circle (10pt);
        \filldraw [teal] (11,13) circle (10pt);
        \filldraw [red] (11,9) circle (10pt);
        \filldraw [teal] (5,9) circle (10pt);
        \filldraw [red] (5,7) circle (10pt);
        \filldraw [yellow] (19,16) circle (10pt);

        \path (18.5,16.5)
        node[text=black,anchor=base west] {\tiny $O$};
        \path (17,16.5)
        node[text=black,anchor=base west] {\tiny $s_2$};
        \path (16.8,14.3)
        node[text=black,anchor=base west] {\tiny $s_3$};
        \path (16.7,16.3)
        node[text=teal,anchor=base east] {\tiny $a_{3}$};
        \path (16.8,12.3)
        node[text=red,anchor=base west] {\tiny $a_{4}$};
        \path (10.7,13.3)
        node[text=teal,anchor=base east] {\tiny $a_{5}$};
        \path (13.7,12.1)
        node[text=black,anchor=base west] {\tiny $s_4$};
        \path (11,10.6)
        node[text=black,anchor=base west] {$\cdots$};
        \path (10.9,8.2)
        node[text=red,anchor=base west] {\tiny $a_{w_{\Lambda}}$};
        \path (6.7,8.2)
        node[text=black,anchor=base west] {\tiny $s_{w_{\Lambda}}$};
        \path (4.8,9)
        node[text=teal,anchor=base east] {\tiny $a_{w_{\Lambda}+1}$};

        \node[draw, align=center] at (12,4.2)
        {\tiny when $|Oa_{2}| = |s_{1}| = 0$};
    \end{tikzpicture}
    \caption{Both pictures show the boundaries of $K$-even shifted Young diagrams. Definition \ref{def:K_even} also requires that $|Oa_{3}|, |Oa_{4}|, \ldots, |Oa_{w_\Lambda}|\equiv n-1 \pmod 2$
    and $ |Oa_{w_\Lambda+1}| \equiv n \pmod 2$.}
    \label{fig:K_even}
\end{figure}
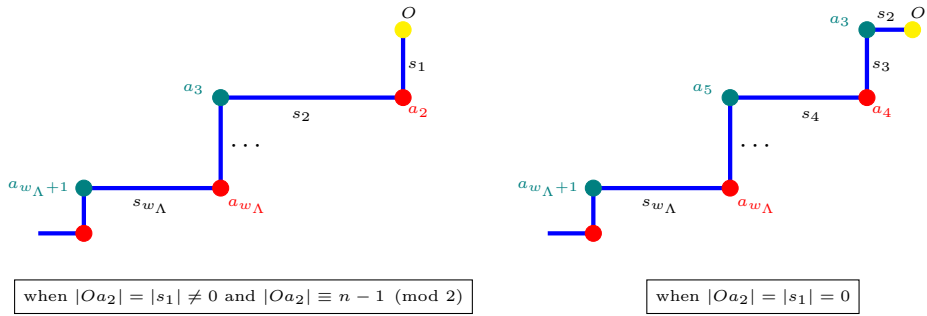

\subsection{Construction of the maps} The aim of this section is to construct the underlying map in the statement of our theorem explicitly. The goal is to construct generalized Lagrangian flags with regard to the almost even and $K$-even Young diagrams. Our models are found by inductively applying Theorem \ref{thm:lg_induc}.

\subsubsection{ }   	Let $\Lambda$ be a shifted Young diagram in the shifted $n$-frame $ \digamma_{n}$.\ Let $s_t$ be a horizontal segment of $\Lambda$, and let $a_{t,1} =a_t$ be its convex corner. Let us order any set
\begin{equation*}
    \{a_{t,1}, a_{t,2}, \ldots, a_{t,v_t}\}
\end{equation*}
of special marked points on $s_t$ from right to left in view of Figure  \ref{fig:boundary-shifted-young}. Let $\mathcal{S}_t := \{a_{t,1}, a_{t,2}, \ldots, a_{t,v_t}\}$ be a set of special marked points for each horizontal segment $s_t$ with $1 \leq t \leq l$. Let $\mathcal{S} $ denote the union of all such sets $\mathcal{S}_t$, and let $\Sigma = |\mathcal{S}|$ and $\tau_\Lambda:= 2\lfloor\frac{l_\Lambda}{2} \rfloor$. Let
$\mathbf{{d}}_{\Lambda} = (d_0, d_1, \ldots d_{\Sigma-1})$ (resp. $\mathbf{{d}}_{\Lambda} = (d_0, d_1, \ldots d_\Sigma)$) ($0 \leq d_i \leq n$) denote the $\Sigma$-tuple (resp. ($\Sigma+1$)-tuple)
\[
    \mathbf{d}_\Lambda:= \begin{cases}
        \big(|Oa_{2,1}|,\ldots, |Oa_{2,v_1}|,|Oa_{4,1}|,\ldots, |Oa_{4,v_4}|, \ldots \ldots,|Oa_{\tau_\Lambda,1}|,\ldots, |Oa_{\tau_\Lambda,v_{\tau_\Lambda}}| \big)    & \textnormal{if $l_\Lambda$ is even} \\
        \big(|Oa_{2,1}|,\ldots, |Oa_{2,v_1}|,|Oa_{4,1}|,\ldots, |Oa_{4,v_4}|, \ldots \ldots,|Oa_{\tau_\Lambda,1}|,\ldots, |Oa_{\tau_\Lambda,v_{\tau_\Lambda}}|, n \big) & \textnormal{if $l_\Lambda$ is odd}
    \end{cases} \]
Define $\mathbf{t}_{\Lambda}$ to be the following $(\Sigma-1)$-tuple (resp. $\Sigma$-tuple)
\[
    \mathbf{t}_{\Lambda}:= \begin{cases}
        \big(|a_{2,1}a_{2,2}|, |a_{2,2}a_{2,3}|, \ldots, |a_{2,v_{2}-1}a_{2,v_2}|, |a_{2,v_2}a_{3}|,|a_3a_{3,1}| \ldots, |a_{\tau,v_{\tau}-1}a_{\tau,v_\tau}|\big) & \textnormal{if $l_\Lambda$ is even} \\
        \big(|a_{2,1}a_{2,2}|, |a_{2,2}a_{2,3}|, \ldots, |a_{2,v_{2}-1}a_{2,v_2}|, |a_{2,v_2}a_{3}|,|a_3a_{3,1}| \ldots, |a_{\tau,v_{\tau}}a_{\tau+1}|\big)     & \textnormal{if $l_\Lambda$ is odd}
    \end{cases} \]
Define $\mathbf{e}_{\Lambda}$ to be the $(\Sigma-1)$-tuple (resp. $\Sigma$-tuple) obtained from $\mathbf{d}_\Lambda$ by deleting the last digit.  Geometrically speaking, $\mathbf{{d}}_{\Lambda}$ captures the lattice length from the origin to the special marked points and $\mathbf{t}_{\Lambda}$ collects the distances between two consecutive special marked points (including terminal points) horizontally.

\subsubsection{ } Construct the following schemes
\[\LF_{\mathcal{S}}^0{(\Lambda)} := \LF_{\mathbf{{d}}_{\Lambda}+\mathbf{1}}(\mathbf{e}_{\Lambda}+\mathbf{2}- \mathbf{t}_{\Lambda};\mathring{V}_\bullet)_{\mathbf{t}_{\Lambda}}, \quad
    \LF_{\mathcal{S}}^1(\Lambda) := \LF_{\mathbf{d}_{\Lambda}+\mathbf{1}}(\mathbf{e}_{\Lambda}+\mathbf{2}- \mathbf{t}_{\Lambda} - \epsilon_1, \mathring{V}_\bullet)_{\mathbf{t}_{\Lambda}} \]
where $\epsilon_1$ is defined in Section \ref{sec:tuples}.
These two types of schemes will be important for constructing generalized Lagrangian flags showed up in our additive basis for Hermitian $K$-theory of Lagrangian Grassmannians. There are canonical maps
\[  S \stackrel{p_\Lambda}\longleftarrow \LF^0_{\mathcal{S}}{(\Lambda)} \stackrel{f_\Lambda}\longrightarrow \LF_0, \quad S \stackrel{p_\Lambda}\longleftarrow \LF^1_{\mathcal{S}}{(\Lambda)} \stackrel{f_\Lambda}\longrightarrow \LF_0\]
where $p_\Lambda$ is the projection and $f_\Lambda$ is the composition of the projection $\LF^i_{\mathcal{S}}{(\Lambda)} \xrightarrow{\pi}  \LF_{d_0} $ followed by the inclusion $ \iota: \LF_{d_0} \hookrightarrow \LF_0 $ ($i =0,1$).

\subsubsection{}
A special marked point $a$ on the segment $s_t$ is called \textit{even} (resp. \textit{odd}) if the lattice length between $a$ and the corner $a_t$ is even (resp. odd).
An ordered set $\mathcal{S}_t:=\{a_{t,1}, a_{t,2}, \ldots, a_{t,v_t}\}$ of special marked points on the segment $s_t$ is called of type $(n)$ ($n \in \{1,2,3\}$) if it satisfies the rule ($n$) below:
\begin{enumerate}
    \item [(1)] $\mathcal{S}_t$ consists of all the special marked points on $s_t$.
    \item [(2)] $\mathcal{S}_t$ consists of all the even special marked points on $s_t$.
    \item [(3)] $\mathcal{S}_t$ consists of all the odd special marked points together with $a_{t,1}$ on $s_t$.
\end{enumerate}
See Figure \ref{fig:boundary}.

\begin{figure}[!ht]
    \centering
    \begin{tikzpicture}[scale=0.3,>=Stealth]
        \coordinate (O)  at (19,15);
        \coordinate (a2) at (19,13);
        \coordinate (a3) at (13,13);
        \coordinate (a4) at (13,10);
        \coordinate (a5) at (9,10);
        \coordinate (a6) at (9,8);
        \coordinate (a7) at (7,8);

        \draw[blue, ultra thick] (O)--(a2)--(a3)--(a4)--(a5)--(a6)--(a7);

        \fill[yellow] (O)  circle (10pt);
        \foreach \P in {a2,a4,a6} \fill[red]  (\P) circle (10pt);
        \foreach \P in {a3,a5} \fill[teal] (\P) circle (10pt);

        \foreach \X/\name in {17/{a_{2,3}},15/{a_{2,5}}}{
        \fill[purple] (\X,13) circle (8pt);
        \node[anchor=north,color=red] at (\X,13) {\tiny $\name$};
        }
        \fill[purple] (11,10) circle (8pt);
        \node[anchor=north,color=red] at (11,10) {\tiny $a_{4,3}$};

        \node[anchor=west]        at (19,15.5) {\tiny $O$};
        \node[anchor=north west,color=red] at (a2) {\tiny $a_2=a_{2,1}$};
        \node[anchor=north west,color=red] at (a4) {\tiny $a_4=a_{4,1}$};

        \node[anchor=west]  at (19,14) {\tiny $s_1$};
        \node[anchor=south] at (16,13)   {\tiny $s_2$};
        \node[anchor=west]  at (13,11.5) {\tiny $s_3$};
        \node[anchor=south] at (11,10)   {\tiny $s_4$};
        \node[anchor=west]  at (9,8.5)     {\tiny $\ldots$};
    \end{tikzpicture}
    \quad
    \begin{tikzpicture}[scale=0.3,>=Stealth]
        \coordinate (b2) at (18,13);
        \coordinate (a3) at (13,13);
        \coordinate (a4) at (13,10);
        \coordinate (a5) at (9,10);
        \coordinate (a6) at (9,8);
        \coordinate (a7) at (7,8);

        \draw[blue, ultra thick] (b2)--(a3)--(a4)--(a5)--(a6)--(a7);

        \fill[red]  (b2) circle (10pt);
        \fill[red]  (a4) circle (10pt);
        \fill[red]  (a6) circle (10pt);
        \foreach \P in {a3,a5} \fill[teal] (\P) circle (10pt);

        \foreach \X/\name in {17/{a_{2,2}},15/{a_{2,4}}}{
        \fill[purple] (\X,13) circle (8pt);
        \node[anchor=north,color=red] at (\X,13) {\tiny $\name$};
        }
        \fill[purple] (11,10) circle (8pt);
        \node[anchor=north,color=red] at (11,10) {\tiny $a_{4,2}$};

        \node[anchor=south] at (16,13)   {\tiny $s_2$};
        \node[anchor=west]  at (13,11.5) {\tiny $s_3$};
        \node[anchor=south] at (11,10)   {\tiny $s_4$};
        \node[anchor=west]  at (9,8.5)     {\tiny $\ldots$};

        \node[anchor=north west,color=red] at (b2) {\tiny $a_2=a_{2,1}$};
        \node[anchor=north west,color=red] at (a4) {\tiny $a_4=a_{4,1}$};
    \end{tikzpicture}

    \caption{The first picture illustrates that $s_2$ is of type (2). The second picture illustrates that $s_2$ is of type (3).}
    \label{fig:boundary}
\end{figure}
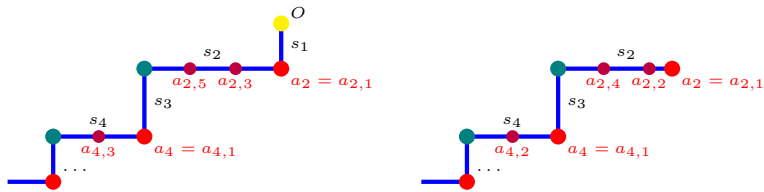

\begin{remark}
    One immediately gets $v_t = |s_t|$ for type (1), $v_t = \lceil \frac{|s_t|}{2} \rceil$ for type (2) and (3).
\end{remark}
\begin{remark}
    If $\mathcal{S}_t$ are all of type (1), then $\LF^0_{\mathcal{S}}{(\Lambda)}  \cong \LF{(\Lambda)}$.
\end{remark}

\subsubsection{ }		Let $w$ be an even integer. Let $\Lambda$ be a shifted Young diagram in $\digamma_n$. Define
\[
    \mathrm{L}{\mathrm{F}}^{a}_w{(\Lambda)}:=
    \LF_{\mathcal{S}{(w)} }^{0}{(\Lambda)}, \quad \mathrm{L}{\mathrm{F}}^b_w{(\Lambda)}:=  \LF_{\mathcal{\widetilde{S}}{(w)} }^{1}{(\Lambda)}\]
where
\begin{itemize}
    \item [a)] $\mathcal{S}{(w)}$ is the set of special marked points such that  $\mathcal{S}(w)_t$ is of type (1) for $t > w$ and $\mathcal{S}(w)_t$ is of type (2) for $t \leq w$;
    \item [b)] $\mathcal{\widetilde{S}}{(w)}$ is the set of special marked points such that  $\mathcal{S}(w)_t$ is of type (1) for $t > w$ and $\mathcal{S}(w)_t$ is of type (2) for $t \leq w$, except for $S(w)_2$ which is of type (3).
\end{itemize}

\subsubsection{ }
Let $\Lambda$ be an almost even Young diagram in $\digamma_{n}$. Construct the maps
\[\begin{aligned}
         & \xi_\Lambda^0: \GW^{[m]}(S) \xrightarrow{p_\Lambda^*} \GW^{[m-|\Lambda|]}(\LF^a({\Lambda}), \omega_{f_\Lambda}) \xrightarrow{f_{\Lambda * }} \GW(\LF_0)                           & \textnormal{if $n$ is odd;}                \\
         & \xi_\Lambda^1: \GW^{[m]}(S) \xrightarrow{p_\Lambda^*} \GW^{[m-|\Lambda|]}(\LF^a({\Lambda}), \omega_{f_\Lambda}\otimes \Delta) \xrightarrow{f_{\Lambda * }} \GW(\widetilde{\LF}_0) & \textnormal{if $n$ is even and $s_1 > 0$;} \\
         & \xi_\Lambda^0: \GW^{[m]}(S, V_1) \xrightarrow{p_\Lambda^*}  \GW^{[m-|\Lambda|]}(\LF^b{(\Lambda)}, \omega_{f_\Lambda}) \xrightarrow{f_{\Lambda * }} \GW(\LF_0)                     & \textnormal{if $n$ is even and $s_1 = 0$;}
    \end{aligned}\]
where $\LF^a(\Lambda) := \LF^a_{l_\Lambda}(\Lambda) $ and $\LF^b(\Lambda):= \LF^b_{l_\Lambda}(\Lambda) $. These maps are well-defined by Corollary \ref{coro:lfl_canonical_divisor}.
\subsubsection{ }
Let $\Lambda$ be a $K$-even Young diagram in $\digamma_{n}$ of index $\omega_\Lambda$. Construct the maps
\[\begin{aligned}
         & \mu_\Lambda^1: K(S) \xrightarrow{p_\Lambda^*} K(\mathrm{L}{\mathrm{F}}^a{(\Lambda)}) \xrightarrow{f_{\Lambda*}} K(\LF_0) \xrightarrow{ H} \GW^{[m]}(\widetilde{\LFb}_0) & \textnormal{for any $n$;}                                  \\
         & \mu_\Lambda^0: K(S) \xrightarrow{p_\Lambda^*} K(\LF^a{(\Lambda)}) \xrightarrow{f_{\Lambda*}} K(\LF_0) \xrightarrow{ H} \GW^{[m]}({\LFb}_0)                              & \textnormal{if n is odd, or if $n$ is even and $s_1 > 0$;} \\
         & \mu_\Lambda^0: K(S) \xrightarrow{p_\Lambda^*} K(\LF^b{(\Lambda)}) \xrightarrow{f_{\Lambda*}} K(\LF_0) \xrightarrow{ H} \GW^{[m]}({\LFb}_0)                              & \textnormal{if $n$ is even and $s_1 = 0$;}
    \end{aligned}\]
where $\LF^a(\Lambda) := \LF^a_{w_\Lambda}(\Lambda) $ and $\LF^b(\Lambda) := \LF^b_{w_\Lambda}(\Lambda)$, and where $H$ is the hyperbolic map.
\subsection{The main theorem}
Recall notation in Section \ref{sect:iota-c}. Let $\mathfrak{H}$ be a subset of $\mathfrak{U}_n$. Further set
\[\mathfrak{H}^r :=\mathfrak{H}\cap \iota^{-1}\mathfrak{U}_n  \quad \mathfrak{H}^c :=  \mathfrak{H} \cap v^{-1}\mathfrak{U}_{n-1}  \]
\[ \mathfrak{H}^{rr} :=\mathfrak{H}\cap \iota^{-1}\iota^{-1}\mathfrak{U}_{n-2} \quad \mathfrak{H}^{cr} :=\mathfrak{H}\cap v^{-1}\iota^{-1}\mathfrak{U}_{n-2} \quad \mathfrak{H}^{cc} :=\mathfrak{H}\cap v^{-1}v^{-1}\mathfrak{U}_{n-2}. \]
The case $\mathfrak{H}= \mathfrak{E}_{n}$ or $\mathfrak{A}_n$ will be of particular interest.
\begin{theorem}\label{thm:main-theo} The following statements hold:
    \begin{enumerate}[leftmargin=20pt,label={\rm (\alph*)},ref=(\alph*)]
        \item Suppose that $n$ is even, then the map
              \[   (\Theta^1, \Omega^1): \bigoplus_{\Lambda \in \mathfrak{E}_n \backslash \mathfrak{A}_n^r } K(S) \oplus \bigoplus_{\Lambda \in \mathfrak{A}_{n}^r} \GW^{[m-|\Lambda|]}(S) \xrightarrow{ (\sum \mu_\Lambda^1, \sum\xi^1_\Lambda)} \GW^{[m]}(\widetilde{\LFb}_0)
              \]
              is an equivalence.
        \item Suppose that $n$ is even, then the map
              \[
                  (\Theta^0, \Omega^0): \bigoplus_{\Lambda \in \mathfrak{E}_n \backslash \mathfrak{A}_n^c } K(S) \oplus \bigoplus_{\Lambda \in \mathfrak{A}_{n}^c} \GW^{[m-|\Lambda|]}(S, V_1) \xrightarrow{ (\sum \mu_\Lambda^0, \sum \xi_\Lambda^0)} \GW^{[m]}(\LFb_0)
              \]
              is an equivalence.
        \item Suppose that $n$ is odd, then the map
              \[
                  (\Theta^0, \Omega^0): \bigoplus_{\Lambda \in \mathfrak{E}_n \backslash \mathfrak{A}_n} K(S) \oplus \bigoplus_{\Lambda \in \mathfrak{A}_n} \GW^{[m-|\Lambda|]}(S) \xrightarrow{ (\sum \mu_\Lambda^0, \sum \xi_\Lambda^0)} \GW^{[m]}(\mathbb{LF}_0)
              \]
              is an equivalence.
        \item Suppose that $n$ is odd, then the map
              \[
                  \Theta^1: \bigoplus_{\Lambda \in \mathfrak{E}_n } K(S) \xrightarrow{ \sum \mu_\Lambda^1} \GW^{[m]}(\widetilde{\mathbb{LF}}_0)
              \]
              is an equivalence.
    \end{enumerate}
\end{theorem}
\begin{proof}
    The proof of this result is parallel to the proof provided in \cite{huang2023the}, which adopt induction and base-change formula repetitively.

    (d). Note that $\mathfrak{E}_n = \mathfrak{E}_n^{rr} \cup \mathfrak{E}_n^{cr} \cup \mathfrak{E}_n^{cc}$. Form the following diagram
    \[
        \xymatrix{
        \big( \bigoplus\limits_{\Lambda \in \mathfrak{E}_n^{rr}} K(S) \big) \oplus  \big( \bigoplus\limits_{\Lambda \in \mathfrak{E}_n^{cr}} K(S) \big) \oplus \big( \bigoplus\limits_{\Lambda \in \mathfrak{E}_n^{cc}} K(S) \big)
        \ar[rrr]^-{	=} \ar[dd]^-{\mathrm{diag}(\sum \rho\mu_{\iota\iota(\Lambda)}^1, \sum \rho\phi_{\iota v(\Lambda)} , \sum \rho \mu_{vv(\Lambda)}^1)} &&&  \bigoplus\limits_{\Lambda \in \mathfrak{E}_n} K(S) \ar[dd]^-{\sum \mu^1_\Lambda} \\ \\
        \GW^{[m-2n+1]}(\widetilde{\mathbb{LF}}_{2}) \oplus K(\LF_{2}) \oplus \GW^{[m]}(\widetilde{\mathbb{LF}}_{2})
        \ar[rrr]^-{	(\iota_{*} \,\,\, \hat{\pi}_* H \hat{\alpha}^* \,\,\,  \grave{\pi}_* \grave{\alpha}^*)} &&& \GW^{[m]}(\widetilde{\mathbb{LF}}_0)
        }\]
    where the bottom map is an equivalence by Theorem \ref{thm:lg_induc} and the left map is an equivalence by the induction hypothesis. The technical part is to show that this diagram is commutative. It suffices to prove the following formulas hold true:
    \begin{align}
        \iota_* \rho \mu_{\iota \iota(\Lambda)}^{1} = \mu_{\Lambda}^{1}               & \quad  \textnormal{if $\Lambda \in \mathfrak{E}_n^{rr}$}  \label{5.1} \\
        \hat{\pi}_* H \hat{\alpha}^* \rho \phi_{\iota v(\Lambda)} = \mu_{\Lambda}^{1} & \quad  \textnormal{if $\Lambda \in \mathfrak{E}_n^{cr}$}  \label{5.2} \\
        \grave{\pi}_* \grave{\alpha}^* \rho \mu_{vv(\Lambda)}^{1} = \mu_{\Lambda}^{1} & \quad  \textnormal{if $\Lambda \in \mathfrak{E}_n^{cc}$}  \label{5.3}
    \end{align}
    The formula \eqref{5.1} can be observed directly from the diagram
    \begin{equation}\label{eq:diag-ii}
        \xymatrix{
        & S  &  \\
        \LF^a(\iota \iota(\Lambda); V^2) \ar[ur]^-{p_{\iota \iota(\Lambda)}} \ar[d]_-{\rho  f_{\iota \iota(\Lambda)}}  \ar[rr] & &\LF^a(\Lambda;V) \ar[d]^-{f_{\Lambda}} \ar[ul]_-{p_{\Lambda}} \\
        \LF_2(V)  \ar[rr]^-{\iota}   &&\LF_0(V)
        }		\end{equation}
    where the lower square is a fiber square.
    The formulas \eqref{5.2} and \eqref{5.3} are proved by applying the base-change formulas to the following diagrams
    \begin{equation}\label{eq:diag-iv-vv}
        \xymatrix{
        & S  &  \\
        \LF(\iota v(\Lambda);V^2) \ar[ur]^-{p_{\iota v(\Lambda)}} \ar[d]_-{\rho f_{\iota v(\Lambda)}}  &  \LF^a(\Lambda;V) \ar[r]^-{=}  \ar[d] \ar[l] &\LF^a(\Lambda;V) \ar[d]^-{f_{\Lambda}} \ar[ul]_-{p_{\Lambda}} \\
        \LF_2 & \ar[l]_-{\hat{\alpha}} \LF_{0,2}(0) \ar[r]^-{\hat{\pi}} &\LF_0
        } \quad
        \xymatrix{
        & S  &  \\
        \LF^a(vv(\Lambda);V^2) \ar[ur]^-{p_{vv(\Lambda)}} \ar[d]_-{\rho f_{vv(\Lambda)}}  &  \LF^a(\Lambda;V) \ar[r]^-{=}  \ar[d] \ar[l] &\LF^a(\Lambda;V) \ar[d]^-{f_{\Lambda}} \ar[ul]_-{p_{\Lambda}} \\
        \LF_2 & \ar[l]_-{\grave{\alpha}} \LF_{0,2}(0)_2 \ar[r]^-{\grave{\pi}}   &\LF_0
        }
    \end{equation}
    respectively where the left squares in both diagrams are fiber squares.

    (c) Observe that
    $\mathfrak{A}_{n} = \mathfrak{A}_{n}^{rr} \cup \mathfrak{A}_{n}^{cc}$ and
    $\mathfrak{E}_{n} \backslash \mathfrak{A}_n = (\mathfrak{E}_{n}\backslash \mathfrak{A}_n)^{rr} \cup \mathfrak{E}_n^{cr} \cup (\mathfrak{E}_{n,l}\backslash \mathfrak{A}_n)^{cc}. $
    Form the diagram
    \[
        \xymatrix{
        \big( \bigoplus\limits_{\Lambda \in \mathfrak{A}_n^{rr}} \GW^{[m-2n+1-|\Lambda|]}(S) \big) \oplus 0 \oplus  \big( \bigoplus\limits_{\Lambda \in \mathfrak{A}_n^{cc}} \GW^{[m-|\Lambda|]}(S)\big)
        \ar[rrr]^-{	=} \ar[dd]^-{\mathrm{diag}\big(\sum \rho\xi_{\iota\iota(\Lambda)}^0, 0,\sum \rho \xi_{vv(\Lambda)}^0\big)} &&&   \bigoplus\limits_{\Lambda \in \mathfrak{A}_n} \GW^{[m-|\Lambda|]}(S) \ar[dd]^-{\sum \xi^0_\Lambda} \\ \\
        \GW^{[m-2n+1]}({\mathbb{LF}}_{2}) \oplus K(\LF_{2}) \oplus \GW^{[m]}({\mathbb{LF}}_{2})
        \ar[rrr]^-{	(\iota_{*} \,\,\, \hat{\pi}_* H \hat{\alpha}^* \,\,\,  \acute{\pi}_* \acute{\alpha}^*)} &&& \GW^{[m]}({\mathbb{LF}}_0)  \\\\
        \big( \bigoplus\limits_{\Lambda \in (\mathfrak{E}_{n}\backslash \mathfrak{A}_n)^{rr}} K(S)
        \oplus \big) \oplus  \big( \bigoplus\limits_{\Lambda \in \mathfrak{E}_n^{cr}} K(S) \big) \oplus \big( \bigoplus\limits_{\Lambda \in (\mathfrak{E}_{n}\backslash \mathfrak{A}_n)^{cc}} K(S) \big)
        \ar[rrr]^-{	=} \ar[uu]_-{\mathrm{diag}(\sum \rho\mu_{\iota\iota(\Lambda)}^0 , \sum \rho\phi_{\iota v(\Lambda)}, \sum \rho \mu_{vv(\Lambda)}^0)} &&&  \bigoplus\limits_{\Lambda \in \mathfrak{E}_n} K(S)   \ar[uu]_-{\sum \mu^0_\Lambda}
        }
    \]
    where the middle map is a stable equivalence by Theorem \ref{thm:lg_induc} and where summing up the upper-left and the lower-left maps yields a stable equivalence by the induction hypothesis.  Thus, summing the top and the bottom squares yields the results, if we can prove the commutativity of both squares. It suffices to prove the following formulas hold true:
    \begin{align}
        \iota_* \rho \xi_{\iota \iota(\Lambda)}^{0} = \xi_{\Lambda}^{0}               & \quad  \textnormal{if $\Lambda \in \mathfrak{A}_n^{rr}$}\label{5.6}                               \\
        \acute{\pi}_* \acute{\alpha}^* \rho \xi_{vv(\Lambda)}^{0} = \xi_{\Lambda}^{0} & \quad  \textnormal{if $\Lambda \in \mathfrak{A}_n^{cc}$} \label{5.7}                              \\
        \iota_* \rho \mu_{\iota \iota(\Lambda)}^{0} = \mu_{\Lambda}^{0}               & \quad  \textnormal{if $\Lambda \in (\mathfrak{E}_n\backslash \mathfrak{A}_n)^{rr}$}  \label{5.8}  \\
        \hat{\pi}_* H \hat{\alpha}^* \rho \phi_{\iota v(\Lambda)} = \mu_{\Lambda}^{0} & \quad  \textnormal{if $\Lambda \in \mathfrak{E}_n^{cr}$}  \label{5.9}                             \\
        \acute{\pi}_* \acute{\alpha}^* \rho \mu_{vv(\Lambda)}^{0} = \mu_{\Lambda}^{0} & \quad  \textnormal{if $\Lambda \in (\mathfrak{E}_n\backslash \mathfrak{A}_n)^{cc}$}  \label{5.10}
    \end{align}
    The formulas \eqref{5.6} and \eqref{5.8} (resp. \eqref{5.9}) are obtained from the same way as in the case (d), using the base-change formulas applied to Diagram \eqref{eq:diag-ii} (resp. the left diagram in \eqref{eq:diag-iv-vv}). The formulas \eqref{5.8} and \eqref{5.10} can be proved similarly as in the case (d) but with the following modified diagram instead
    \[
        \xymatrix{
        & S  &  \\
        \LF^a(vv(\Lambda);V^2) \ar[ur]^-{p_{vv(\Lambda)}} \ar[d]_-{\rho  f_{vv(\Lambda)}}  &  \LF^a(\Lambda;V) \ar[r]^-{=}  \ar[d] \ar[l] &\LF^a(\Lambda;V) \ar[d]^-{f_{\Lambda}} \ar[ul]_-{p_{\Lambda}} \\
        \LF_2 & \ar[l]_-{\acute{\alpha}} \LFo_{1,3}(0)_2 \ar[r]^-{\acute{\pi}}   &\LF_0
        }		\]
    where the left square is a fiber square.

    (b)	Note that $\mathfrak{E}_n \backslash \mathfrak{A}_n^c = \mathfrak{E}_n^r \cup (\mathfrak{E}_n \backslash \mathfrak{A}_n)^c$ and $\mathfrak{E}_n \backslash \mathfrak{A}_n^r = \mathfrak{E}_n^r \cup (\mathfrak{E}_n \backslash \mathfrak{A}_n)^r$. Let us construct the diagram
    \[
        \xymatrix@C=16pt{
        \big(\bigoplus\limits_{\Lambda \in \mathfrak{E}_n^r} K(S) \big) \oplus \big( \bigoplus\limits_{\Lambda \in (\mathfrak{E}_n \backslash \mathfrak{A}_n)^c} K(S)  \oplus \bigoplus\limits_{\Lambda \in \mathfrak{A}_n^c} \GW^{[m-|\Lambda|]}(S) \big) \ar[dd]^-{\mathrm{diag}\big(\sum \rho \mu^0_{ \iota(\Lambda)}, \big(\sum \rho \mu^0_{v(\Lambda)}, \sum \rho \xi^0_{v(\Lambda)}\big)\big)} \ar[rr] &&  \bigoplus\limits_{\Lambda \in \mathfrak{E}_n \backslash \mathfrak{A}_n^c} K(S)  \oplus \bigoplus\limits_{\Lambda \in \mathfrak{A}_n^c} \GW^{[m]}(S)  \ar[dd]^-{(\sum \mu^0_\Lambda,  \sum\xi^0_\Lambda)} \\ \\
        \GW^{[m-n]}(\widetilde{\mathbb{LF}}_{1}) \oplus \GW^{[m]}({\mathbb{LF}}_{1})  \ar[rr]^-{ (\iota_{1*} \,\,\, \bar{\pi}_*\bar{\alpha}^*)} &&
        \GW^{[m]}({\mathbb{LF}}_0)}
    \]
    where the bottom map is an equivalence and the left map is an equivalence. The key step is to show that the diagram is commutative, that is, the following formulas hold true:
    \begin{align}
        \bar{\pi}_* \bar{\alpha}^* \rho \xi^0_{v(\Lambda)} =\xi^0_\Lambda & \quad  \textnormal{if } \Lambda \in \mathfrak{A}_n^c, \label{eqn:5_11}                             \\
        \bar{\pi}_* \bar{\alpha}^* \rho \mu^0_{v(\Lambda)} =\mu^0_\Lambda & \quad  \textnormal{if } \Lambda \in (\mathfrak{E}_n \backslash \mathfrak{A}_n)^c, \label{eqn:5_12} \\
        \iota_{1*} \rho \mu^0_{\iota(\Lambda)} =\mu^0_\Lambda             & \quad  \textnormal{if } \Lambda \in \mathfrak{E}_n^r. \label{eqn:5_13}
    \end{align}
    The formulas \eqref{eqn:5_11} and \eqref{eqn:5_12} are proved by applying the base-change formulas to the following diagram
    \[
        \xymatrix{
        & S  &  \\
        \LF^b(v(\Lambda);V^1) \ar[d]_-{f_{v(\Lambda)}} \ar[ur]^-{p_{v(\Lambda)}}  &  \LF^b(\Lambda;V) \ar[r]^-{=}  \ar[d] \ar[l] &\LF^b(\Lambda;V) \ar[d]^-{f_{\Lambda}} \ar[ul]_-{p_{\Lambda}} \\
        \LF_1 & \ar[l]_{\bar{\alpha}} \LF_{1,2}(0) \ar[r]^-{\bar{\pi}}   &\LF_0
        }		\]
    where the left square is a fiber product. The formula \eqref{eqn:5_13} is derived from the following diagrams, using base-change formulas
    \[
        \xymatrix{
        & S  &  \\
        \LF^b(\iota (\Lambda); V^1) \ar[ur]^-{p_{\iota (\Lambda)}} \ar[d]_-{\rho f_{\iota (\Lambda)}}  \ar[rr] & &\LF^b(\Lambda;V) \ar[d]^-{f_{\Lambda}} \ar[ul]_-{p_{\Lambda}} \\
        \LF_1(V)  \ar[rr]^-{\iota}   &&\LF_0(V)
        }		\]
    where the bottom square is a fiber product.

    (a) The main step is to construct the diagram
    \[
        \xymatrix@C=16pt{
        \big(	\bigoplus\limits_{\Lambda \in (\mathfrak{E}_n \backslash \mathfrak{A}_n)^r} K(S) \oplus \bigoplus\limits_{\Lambda \in \mathfrak{A}_n^r} \GW^{[m-|\Lambda|]}(S) \big) \oplus \big( \bigoplus\limits_{\Lambda \in \mathfrak{E}_n^c} K(S) \big) \ar[dd]^-{\mathrm{diag}\big(\big(\sum \rho \mu^1_{ \iota(\Lambda)},  \sum \rho \mu^1_{v(\Lambda)}\big), \sum \rho \xi^1_{v(\Lambda)}\big)} \ar[rr] &&  \bigoplus\limits_{\Lambda \in \mathfrak{E}_n \backslash \mathfrak{A}_n^{r}} K(S)  \oplus \bigoplus\limits_{\Lambda \in \mathfrak{A}_n^{r}} \GW^{[m]}(S)  \ar[dd]^-{(\sum \mu^1_\Lambda , \sum \xi^1_\Lambda)} \\ \\
        \GW^{[m-n]}(\mathbb{LF}_{1}) \oplus \GW^{[m]}(\widetilde{\mathbb{LF}}_{1})  \ar[rr]^-{ (\iota_{1*} \,\,\, \pi_*\alpha^*)} &&
        \GW^{[m]}(\widetilde{\mathbb{LF}}_0)}
    \]
    and adopt the same strategy as above. The proof of the commutativity of this diagram is similar to the case (b), but using $\LF^a(\Lambda)$ and $\LF_{0,1}(0)$ instead. We do not repeat details.
\end{proof}

\subsection{Examples of shifted Young diagrams}
Figures~\ref{fig:exp_almost_even} and~\ref{fig:exp_K_even} illustrate various examples of almost even and $K$-even shifted Young diagrams, respectively.

\begin{figure}[!h]
    \centering

    \begin{tikzpicture}[scale=0.4]
        \pgfmathsetmacro{\nnn}{3}
        \pgfmathsetmacro{\gap}{\nnn + 1.5}
        \foreach \t in {0,...,3}{
                \foreach \y in {1,...,\nnn}{
                        \draw (\gap*\t + \y-1,\nnn - \y) -- (\nnn+\gap*\t,\nnn - \y);}
                \draw (\gap*\t,\nnn ) -- (\nnn+\gap*\t,\nnn );
                \foreach \x in {1,...,\nnn}{
                        \draw (\x - 1+\gap*\t,\nnn - \x) -- (\x - 1+\gap*\t,\nnn);}
                \draw (\nnn + \gap*\t,0) -- (\nnn + \gap*\t,\nnn);
            }
        \drawalmostyoung{\gap}{\nnn}{0}{1,2,3}
        \drawalmostyoung{\gap}{\nnn}{1}{1,2,2}
        \drawalmostyoung{\gap}{\nnn}{2}{1,0,0}
    \end{tikzpicture}

    \vspace{10pt}

    \begin{tikzpicture}[scale=0.4]
        \pgfmathsetmacro{\nnn}{4}
        \pgfmathsetmacro{\gap}{\nnn + 1}
        \foreach \t in {0,...,3}{
                \foreach \y in {1,...,\nnn}{
                        \draw (\gap*\t + \y-1,\nnn - \y) -- (\nnn+\gap*\t,\nnn - \y);}
                \draw (\gap*\t,\nnn ) -- (\nnn+\gap*\t,\nnn );
                \foreach \x in {1,...,\nnn}{
                        \draw (\x - 1+\gap*\t,\nnn - \x) -- (\x - 1+\gap*\t,\nnn);}
                \draw (\nnn + \gap*\t,0) -- (\nnn + \gap*\t,\nnn);
            }
        \drawalmostyoung{\gap}{\nnn}{0}{1,2,3,4}
        \drawalmostyoung{\gap}{\nnn}{1}{1,2,3,3}
        \drawalmostyoung{\gap}{\nnn}{2}{1,2,1,1}
        \drawalmostyoung{\gap}{\nnn}{3}{1,1,1,1}
    \end{tikzpicture}

    \vspace{10pt}

    \begin{tikzpicture}[scale=0.4]
        \pgfmathsetmacro{\nnn}{4}
        \pgfmathsetmacro{\gap}{\nnn + 1}
        \foreach \t in {0,...,3}{
                \foreach \y in {1,...,\nnn}{
                        \draw (\gap*\t + \y-1,\nnn - \y) -- (\nnn+\gap*\t,\nnn - \y);}
                \draw (\gap*\t,\nnn ) -- (\nnn+\gap*\t,\nnn );
                \foreach \x in {1,...,\nnn}{
                        \draw (\x - 1+\gap*\t,\nnn - \x) -- (\x - 1+\gap*\t,\nnn);}
                \draw (\nnn + \gap*\t,0) -- (\nnn + \gap*\t,\nnn);
            }
        \drawalmostyoung{\gap}{\nnn}{0}{1,2,3,0}
        \drawalmostyoung{\gap}{\nnn}{1}{1,2,2,0}
        \drawalmostyoung{\gap}{\nnn}{2}{1,0,0,0}
    \end{tikzpicture}

    \vspace{10pt}

    From top to bottom: Diagrams in $\mathfrak{A}_{3}$, $\mathfrak{A}_{4}^{r}$, and $\mathfrak{A}_{4}^{c}$.
    \caption{Examples of almost even Young diagrams.}
    \label{fig:exp_almost_even}
\end{figure}

\begin{figure}[!h]
    \centering

    \begin{tikzpicture}[scale=0.4]
        \pgfmathsetmacro{\nnn}{3}
        \pgfmathsetmacro{\gap}{\nnn + 1.5}
        \foreach \t in {0,...,1}{
                \foreach \y in {1,...,\nnn}{
                        \draw (\gap*\t + \y-1,\nnn - \y) -- (\nnn+\gap*\t,\nnn - \y);}
                \draw (\gap*\t,\nnn ) -- (\nnn+\gap*\t,\nnn );
                \foreach \x in {1,...,\nnn}{
                        \draw (\x - 1+\gap*\t,\nnn - \x) -- (\x - 1+\gap*\t,\nnn);}
                \draw (\nnn + \gap*\t,0) -- (\nnn + \gap*\t,\nnn);
            }
        \drawalmostyoung{\gap}{\nnn}{0}{1,2}
        \drawalmostyoung{\gap}{\nnn}{1}{1,1}
    \end{tikzpicture}

    \vspace{10pt}

    \begin{tikzpicture}[scale=0.4]
        \pgfmathsetmacro{\nnn}{3}
        \pgfmathsetmacro{\gap}{\nnn + 1.5}
        \foreach \t in {0,...,3}{
                \foreach \y in {1,...,\nnn}{
                        \draw (\gap*\t + \y-1,\nnn - \y) -- (\nnn+\gap*\t,\nnn - \y);}
                \draw (\gap*\t,\nnn ) -- (\nnn+\gap*\t,\nnn );
                \foreach \x in {1,...,\nnn}{
                        \draw (\x - 1+\gap*\t,\nnn - \x) -- (\x - 1+\gap*\t,\nnn);}
                \draw (\nnn + \gap*\t,0) -- (\nnn + \gap*\t,\nnn);
            }
        \drawalmostyoung{\gap}{\nnn}{0}{1,2,2}
        \drawalmostyoung{\gap}{\nnn}{1}{1,2}
        \drawalmostyoung{\gap}{\nnn}{2}{1,1}
    \end{tikzpicture}

    \vspace{10pt}

    \begin{tikzpicture}[scale=0.4]
        \pgfmathsetmacro{\nnn}{4}
        \pgfmathsetmacro{\gap}{\nnn + 1.5}
        \foreach \t in {0,...,5}{
                \foreach \y in {1,...,\nnn}{
                        \draw (\gap*\t + \y-1,\nnn - \y) -- (\nnn+\gap*\t,\nnn - \y);}
                \draw (\gap*\t,\nnn ) -- (\nnn+\gap*\t,\nnn );
                \foreach \x in {1,...,\nnn}{
                        \draw (\x - 1+\gap*\t,\nnn - \x) -- (\x - 1+\gap*\t,\nnn);}
                \draw (\nnn + \gap*\t,0) -- (\nnn + \gap*\t,\nnn);
            }
        \drawalmostyoung{\gap}{\nnn}{0}{1,2,3,3}
        \drawalmostyoung{\gap}{\nnn}{1}{1,2,3,1}
        \drawalmostyoung{\gap}{\nnn}{2}{1,2,2,1}
        \drawalmostyoung{\gap}{\nnn}{3}{1,1,1,1}
        \drawalmostyoung{\gap}{\nnn}{4}{1,2}
        \drawalmostyoung{\gap}{\nnn}{5}{1,1}
    \end{tikzpicture}

    \vspace{10pt}

    \begin{tikzpicture}[scale=0.4]
        \pgfmathsetmacro{\nnn}{4}
        \pgfmathsetmacro{\gap}{\nnn + 1.5}
        \foreach \t in {0,...,5}{
                \foreach \y in {1,...,\nnn}{
                        \draw (\gap*\t + \y-1,\nnn - \y) -- (\nnn+\gap*\t,\nnn - \y);}
                \draw (\gap*\t,\nnn ) -- (\nnn+\gap*\t,\nnn );
                \foreach \x in {1,...,\nnn}{
                        \draw (\x - 1+\gap*\t,\nnn - \x) -- (\x - 1+\gap*\t,\nnn);}
                \draw (\nnn + \gap*\t,0) -- (\nnn + \gap*\t,\nnn);
            }
        \drawalmostyoung{\gap}{\nnn}{0}{1,2,3,1}
        \drawalmostyoung{\gap}{\nnn}{1}{1,2,2,1}
        \drawalmostyoung{\gap}{\nnn}{2}{1,2,2,0}
        \drawalmostyoung{\gap}{\nnn}{3}{1,2}
        \drawalmostyoung{\gap}{\nnn}{4}{1,1}
    \end{tikzpicture}

    \vspace{10pt}

    From top to bottom: Diagrams in $\mathfrak{E}_{3} \setminus \mathfrak{A}_{3}$, $\mathfrak{E}_{3}$, $\mathfrak{E}_{4} \setminus \mathfrak{A}_{4}^{c}$, and $\mathfrak{E}_{4} \setminus \mathfrak{A}_{4}^{r}$.

    \caption{Examples of $K$-even Young diagrams.}
    \label{fig:exp_K_even}
\end{figure}

\newpage
\thispagestyle{plain}
{\tiny
    \begin{landscape}
        \centering
        \vspace*{\fill}
        \begin{figure}[htpb]
            \centering
            \[
                \begingroup \setlength\arraycolsep{1pt} \renewcommand{\arraystretch}{1}
                \vcenter{\xymatrix@C=20pt@R=100pt{
                &{\left\{L_n \mid V_{1} \subset L_n\right\}} \ar[r]^-{\iota_1} &\{L_n \mid L_n \subset V \} & \ar[l]_-{v_1} \{ L_n \mid V_{1} \nhd\, {L_n} \} \ar[ld]_-{\tilde{v}_1} \ar@{=}[r] & \{ L_n \mid V_{1} \nhd\, {L_n} \} \ar[dl]_-{w} \ar[dd]^{\vartheta}\\
                {\left\{{ (P_{n-1},L_n)} \middle \vert \begin{matrix}
                    V_{1}                   \\
                    \cap                    \\
                    P_{n-1} & \subset & L_n
                \end{matrix}\right\}} \ar[r]^-{\kappa}  & {\left\{{ (P_{n-1},L_n)} \middle \vert \begin{matrix}
                            P_{n-1} & \subset & L_n          \\
                                    &         & \cap         \\
                                    &         & V_{1}^{\bot}
                        \end{matrix}\right\}} \ar[r]^-{\tilde{\iota}_1} \ar[u]^-{\tilde{\pi}_1} & {\left\{{ (P_{n-1},L_n)} \middle \vert \begin{matrix}
                            P_{n-1} & \subset & L_n \\
                            \cap                    \\
                            V_{1}^{\bot}
                        \end{matrix}\right\}}  \ar[u]^-{\pi_1} & \ar[l]_-{v_2} {\left\{{ (P_{n-1},L_n)} \middle \vert \begin{matrix}
                    V_{1} & \nhd & P_{n-1}      & \subset & L_n \\
                          &      & \cap                         \\
                          &      & V_{1}^{\bot}
                \end{matrix}\right\}} \ar[ld]_-{\tilde{v}_2}\\
                {\left\{{ (P_{n-1},L_n,L'_n)} \middle \vert \begin{matrix}
                    V_{1} & \subset & P_{n-1} & \subset & L_n       & \subset & V_1^\perp \\
                          &         & \cap                                                \\
                          &         & L'_n    & \subset & V_1^\perp
                \end{matrix}\right\}}  \ar[u]^-{\tilde{\pi}_2} \ar[r] &   {\left\{{ (P_{n-1},L_n,L'_n)} \middle \vert \begin{matrix}
                             &  & P_{n-1} & \subset & L_n       & \subset & V_1^\perp \\
                             &  & \cap                                                \\
                             &  & L'_n    & \subset & V_1^\perp
                        \end{matrix}\right\}}  \ar[r]& {\left\{{ (P_{n-1},L_n,L'_n)} \middle \vert \begin{matrix}
                             &  & P_{n-1} & \subset & L_n         \\
                             &  & \cap                            \\
                             &  & L'_n    & \subset & V_{1}^\perp
                        \end{matrix}\right\}} \ar[u]^-{\pi_2}\ar@{-->}[rr]^-{\alpha} && {\{L'_n \mid V_{1} \subset L'_n\}}
                }}
                \endgroup
            \]
            \caption{Functor of points of Diagram \ref{eqn:two_step_blow_up_lg}}
            \label{fig:lg_blow_up_functor_of_points}
        \end{figure}
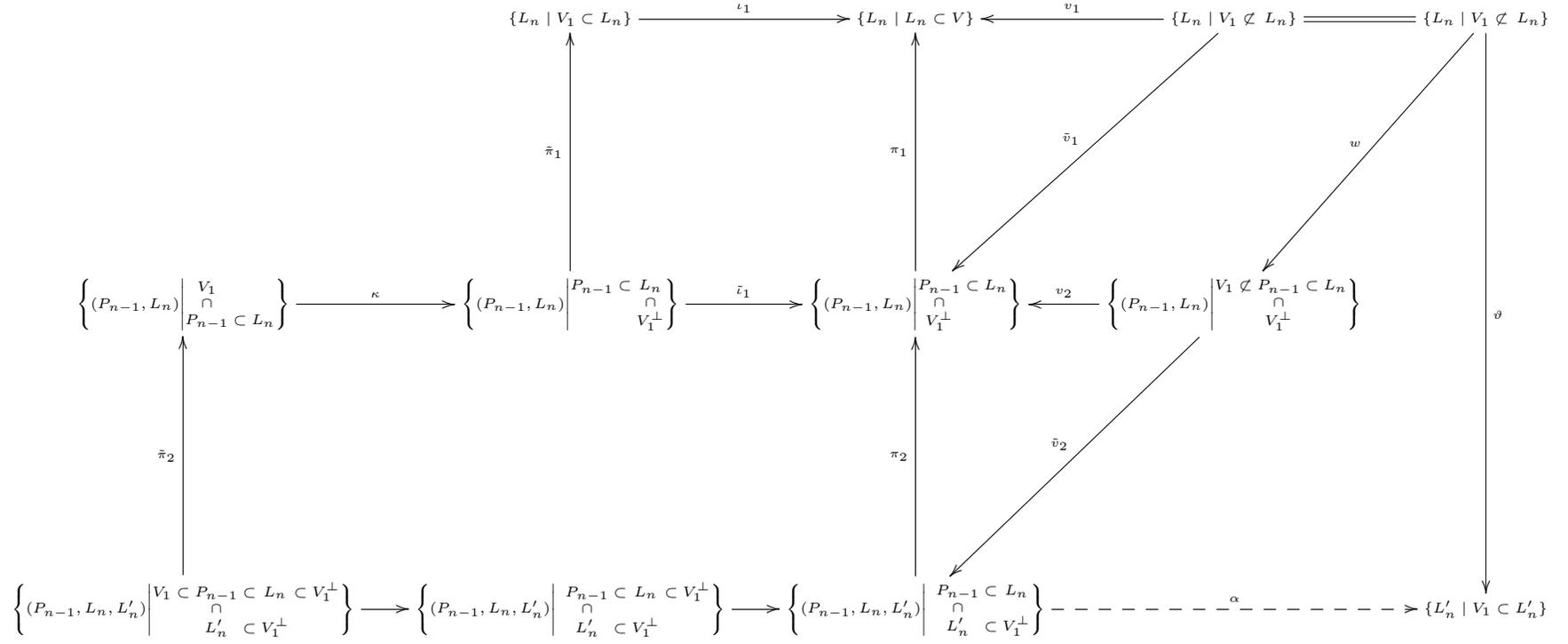
        \vspace*{\fill}
    \end{landscape}
}
\appendix

%%%%%%%%%%%%%%%%%%%% Regular embeddings and fiber squares %%%%%%%%%%%%%%%%%%%%
\section{Regular embeddings and fiber squares}
Suppose that $X$ is a Cohen-Macaulay scheme.  Let $i:Z\hookrightarrow X$ and $j:W\hookrightarrow X$ be two regular embeddings of codimension $d$ and $d'$ respectively. Form the following fiber square.
\begin{equation}\label{eq:fiber}
    \xymatrix{T \ar[d]^-{v}\ar[r]^-{u} & W \ar[d]^-{j} \\
    Z \ar[r]^-{i} & X
    }
\end{equation}

\begin{lemma}\label{lma:meet-proper-tor-independent}
    Assume that $Z$ and $W$ meet properly in $X$, i.e., $\mathrm{codim}_W T = \mathrm{codim}_X Z $ (or equivalently, $ \mathrm{codim}_Z T = \mathrm{codim}_W X$). Then, the following statements hold true.
    \begin{enumerate}
        \item The square $\eqref{eq:fiber}$ is tor-independent in the sense of \cite[Definition 3.10.2]{lipman2009notes}.
        \item The closed embeddings $u$ and $v$ are regular embeddings of codimension $d$ and $d'$ respectively.
    \end{enumerate}
\end{lemma}
\begin{proof}
    Let $x\in X$ be a point, and let $\mathrm{Spec}(R) \subset X$ be an open neighborhood of $x$. Recall the definition of regular embedding, cf. \cite[Appendix B.7.1]{fulton1998intersection}. There are regular sequences $\mathbf{s} = \{s_1, s_2 ,\ldots , s_{d}\}$ and $\mathbf{t} = \{t_1, t_2 ,\ldots , t_{d'}\}$ in $R$ such that $i^{-1}(\mathrm{Spec}(R)) = \mathrm{Spec}(R/(\mathbf{s}))$ and $j^{-1}(\mathrm{Spec}(R)) = \mathrm{Spec}(R/(\mathbf{t}))$. Note that
    \[
        v^{-1} i^{-1}(\mathrm{Spec}(R)) = u^{-1} j^{-1}(\mathrm{Spec}(R)) = \mathrm{Spec}(R/(\mathbf{s}, \mathbf{t})).
    \]
    Since $Z$ and $W$ meet properly in $X$, we conclude that $T$ has codimension $d+d'$ in $X$. Therefore
    \[
        \mathrm{dim} R/(\mathbf{s}, \mathbf{t}) = \mathrm{dim}\mathrm{R} - d - d'
    \]
    and since $R$ is Cohen-Macaulay, $\mathbf{s} \cup \mathbf{t}$ is a regular sequence in $R$, cf. \cite[Theorem 2.1.2(c)]{bruns1993cohen}. Then the Koszul complex $K^{\bullet}(R;\mathbf{s}, \mathbf{t}) $ provides a projective resolution of $  R/(\mathbf{s}, \mathbf{t})$ (See \cite[Corollary 1.6.14(b)]{bruns1993cohen}). It follows  that
    \[
        \mathrm{Tor}^{\bullet}_{R}(R/(\mathbf{s}), R/(\mathbf{t})) \simeq K^{\bullet}(R;\mathbf{s}) \otimes_{R} K^{\bullet}(R;\mathbf{t}) \simeq K^{\bullet}(R;\mathbf{s}, \mathbf{t})
    \]
    whose cohomology is non-trivial only in  degree $0$. Hence, $\mathrm{Tor}_i^R(R/(\mathbf{s}), R/(\mathbf{t}))=0$ for $i > 0$, and the square \eqref{eq:fiber} is tor-independent. Since  $\mathbf{s} \cup \mathbf{t}$ is a regular sequence in $R$,  the image of $\mathbf{s}$ in $R/(\mathbf{t})$ is also a regular sequence of length $d$. Therefore, $u$ is a regular embedding of codimension $d$. A similar proof yields that $v$ is also a regular embedding of codimension $d'$.
\end{proof}

\begin{lemma}\label{lma:tor-independent-eclate}
    Let $g:X_1 \to X_2$ be a morphism of schemes. Let $Z_2 \subset X_2$ be a closed subscheme such that the following fiber product
    \begin{equation}\label{eqn:tor-independent-eclate}
        \xymatrix{
        Z_1 \ar[r] \ar[d] & X_1 \ar[d]^-{g} \\
        Z_2 \ar[r] & X_2
        }
    \end{equation}
    is tor-independent, where $Z_1 := g^{-1}(Z_2)$. Let $X'_{i}$ be the blowup of $Z_i$ in $X_i$ with exceptional fiber $E_i$ ($i=1,2$). Then there exists a cube
    \begin{equation}\label{eqn:cube-tor-independent-eclate}
        \xymatrix@R=10pt@C=10pt{
        & Z_1 \ar[rr] \ar'[d][dd]
        & & X_1 \ar[dd]^-{g}
        \\
        E_1 \ar[ur]\ar[rr]\ar[dd]
        & & X'_1 \ar[ur]\ar[dd]
        \\
        & Z_2 \ar'[r][rr]
        & & X_2
        \\
        E_2 \ar[rr]\ar[ur]
        & & X_2' \ar[ur]
        }
    \end{equation}

    of schemes with all the faces are fiber products.
\end{lemma}
\begin{proof}
    The proof is similar to the one of \cite[\href{https://stacks.math.columbia.edu/tag/0805}{Lemma 0805}]{stacks-project}. Let $\mathcal{I}_2$ be the ideal sheaf of $Z_2 \subset X_2$. Then the ideal sheaf $\mathcal{I}_1$ of $Z_1 \subset X_1$ is the image of $\sigma: g^* \mathcal{I}_2 \to \SO_{X_1}$. Note that $X_1 \times _{X_2} X_2'$ is the relative Proj of $\bigoplus _{n \geq 0} g^*\mathcal{I}_2^{\otimes n}$. Let $\mathcal{F}_\bullet$ be a flat resolution of $\mathcal{I}_2$. Since the square \eqref{eqn:tor-independent-eclate} is tor-independent, in derived category $\mathrm{D}(X_1)$,
    \[
        \SO_{X_1}/\mathcal{I}_1 \cong \mathbf{L} g^\ast (\SO_{X_2} / \mathcal{I}_2) \cong (\ldots \xrightarrow{d_{-3}} g^\ast\mathcal{F}_{-1} \xrightarrow{d_{-2}} g^\ast\mathcal{F}_{0} \xrightarrow{d_{-1}} \SO_{X_1})
    \]
    is concentrated at degree $0$, therefore the map $\sigma: g^\ast \mathcal{I}_2 \cong g^\ast\mathcal{F}_{0}/\mathrm{im}(d_{-2}) \to \SO_{X_1}$ with the image $\mathcal{I}_1$ is injective. Then the maps $\sigma^{\otimes n}: g^* \mathcal{I}_2^{\otimes n} \cong (g^* \mathcal{I}_2)^{\otimes n} \to \SO_{X_1}$ are also injective. Thus, we see that $X'_1$ can be identified with $X_1 \times _{X_2} X_2'$. Moreover, $\SO_{X_1'}(E_1) \cong \SO_{X_1'}(-1)$ can be identified with the pullback of $\SO_{X_2'}(E_2) \cong \SO_{X_2'}(-1)$ along $X_1' \to X_2'$, so $E_1 = X_1' \times _{X_2'} E_2$. Therefore, we have shown that all faces of the cube \eqref{eqn:cube-tor-independent-eclate}, except the left one, are fiber squares. The result then follows from \cite[Exercises 5.2(b), p.115]{awodey2010category}.
\end{proof}

%%%%%%%%%%%%%%%%%%%% Base change formula for the Gorenstein case %%%%%%%%%%%%%%%%%%%%
\section{Base change formula for the Gorenstein case}\label{sec:base_change_gorenstein}

Denote by $\GW(X, I)$ the coherent Grothendieck–Witt theory for any scheme $X$ with residue complex $I$; cf. \cite[Definition 2.6]{huang2023the}. Suppose that all the schemes are Gorenstein divisorial throughout this section.

\subsection{The map \texorpdfstring{$\rho$}{rho}}\label{sect:rho}
The functorial injective resolution $\rho: L \to EL$ (cf. \cite[Remark 4.10]{huang2023the})  induces a map
\[
    \rho:\GW^{[n]}(S,L) \to \GW^{[n]}(S,EL)
\]
for any line bundle $L$ over $S$. This map is also denoted by $\rho$. When $S$ is regular, the map $\rho$ is an equivalence.

\subsection{Pushforward map}
Let $Y$ be a Gorenstein divisorial scheme, and $\iota: X \to Y$ be a regular embedding of codimension $c$. By \cite[Proposition III.7.2, p.179]{hartshorne1966residues}(or \cite[\href{https://stacks.math.columbia.edu/tag/0BVA}{Lemma 0BVA}]{stacks-project}), the scheme $X$ is Gorenstein. The scheme $X$ is divisorial, since any subscheme of a divisorial scheme is still divisorial by definition (cf. \cite[Definition II.2.2.5]{SGA6}).
Assume that $L$ is a line bundle on $Y$. Then there is an isomorphism $\chi: E\iota^{\natural}L \xrightarrow{\cong} \iota^{\Delta}EL[c]$, cf. \cite[Theorem VI.3.1(c), p.318 and Corollary III.7.3, p.180]{hartshorne1966residues}.
\begin{definition}
    Construct the pushforward map
    \[\iota_*:\GW^{[n-c]} (X, \iota^{\natural}L) \to \GW^{[n]}(Y,EL)\]
    as the following composition
    \[\GW^{[n-c]} (X, \iota^{\natural}L) \xrightarrow{\rho} \GW^{[n-c]} (X, E\iota^{\natural}L) \xrightarrow[\simeq]{\chi} \GW^{[n]} (X, \iota^{\Delta}EL) \xrightarrow{\iota_*} \GW^{[n]}(Y,EL).   \]
\end{definition}

\begin{remark}
    The map $\iota_*:\GW^{[n-c]} (X, \iota^{\natural}L) \to \GW^{[n]}(Y,EL)$ roots in the non-singular exact dg form functor
    \[
        (\iota_\ast, \tilde{\zeta}) : (\Ch^b(\V(X)), \sharp_{\iota^{\natural}L[-c]}, \mathrm{quis}, \can) \to ( \Ch^b_c(\M(Y)) , \sharp_{EL}, \mathrm{quis}, \can),
    \]
    with $\tilde{\zeta} : \iota_\ast \sharp_{\iota^{\natural}L[-c]} \to \sharp_{EL}\iota_\ast $ the composition
    \[
        \iota_\ast\left[-, \iota^{\natural}L[-c]\right]_X \xrightarrow{\iota_\ast(\rho)} \iota_\ast\left[-, E\iota^{\natural}L[-c]\right]_X \xrightarrow{\iota_*(\chi)} \iota_\ast\left[-, \iota^{\Delta}EL\right]_X \xrightarrow{\phi} \left[\iota_\ast (-), \iota_\ast \iota^{\Delta}EL\right]_Y \xrightarrow{\mathrm{Tr}} \left[\iota_\ast (-), EL\right]_Y.
    \]
\end{remark}
\subsection{Projection formula} The projection formula for regular schemes has been proved in \cite[Proposition 6.6]{huang2023the}. The proof in \textit{loc.\ cit.}\ can be applied to the singular case, for which we record the following result:
\begin{proposition}[Projection formula]\label{prop:projection_formula_finite}
    Let $\iota: Z \to X$ be a finite morphism of Gorenstein divisorial schemes. Let $I$ be a residue complex on $X$ and let $L$ be a line bundle on $X$.
    Then the following diagram of Grothendieck-Witt spectra
    \[
        \xymatrix{
        \GW^{[n]}(Z, \iota ^\Delta I) \times \GW^{[m]}(X, L) \ar[rr]^-{\iota_\ast \times \mathrm{id}} \ar[d]_-{\cup \circ (\mathrm{id} \times \iota^\ast)} && \GW^{[n]}(X, I) \times \GW^{[m]}(X, L) \ar[d]_-{\cup} \\
        \GW^{[n+m]}(Z, \iota ^\Delta I \otimes \iota^\ast L) \ar[r]^-{\simeq} & \GW^{[n+m]}(Z, \iota ^\Delta (I \otimes L)) \ar[r]^-{\iota_\ast} & \GW^{[n+m]}(X, I\otimes L)
        }
    \]
    commutes up to homotopy. In a formula, we have that
    \[ \iota_* (\alpha \cup \iota^* \beta) = \iota_*(\alpha) \cup \beta \]
    for any $\alpha \in \GW^{[n]}_i(Z, \iota^\Delta I)$ and $\beta \in \GW^{[m]}_j(X,L)$.
\end{proposition}

\subsection{Base-change formula}

Let $L$ be a line bundle on $Y$. Assume that $Y$ is a regular scheme. Suppose that we have a fiber square of schemes
\[\xymatrix{X \ar[r]^-{\iota} & Y  \\
    V \ar[r]^-{\tilde{\iota}} \ar[u]^-{\tilde{\pi}} & Z \ar[u]^-{\pi}
    }\]
Assume that $X,V,Z$ are Gorenstein schemes and that $\iota$ is a local complete intersection (cf. \cite[p.138]{hartshorne1966residues}). Suppose further that the square is tor-independent.
\begin{proposition}\label{prop:base_change_formula_regular_and_locus}
    Assume that $\iota$ is a regular embedding of codimension $c$, and that $\pi: Z\to Y$ is the zero locus of a regular section $s: \mathcal{E}^\vee \to \SO_Y$, where $\mathcal{E}$ is a vector bundle on $Y$ of rank $d$. Suppose further that $X$ and $Z$ meet properly in $Y$.
    For any line bundle $L$ over $Y$, the following diagram
    \[
        \xymatrix{
        \GW^{[n - c]}(X, \iota^{\natural} L) \ar[r]^-{\iota_*\rho} \ar[d]_-{\gamma\tilde{\pi}^*} & \GW^{[n]}(Y, EL) \ar[d]_-{\rho\pi_*\rho^{-1}} \\
        \GW^{[n-c]}(V, \tilde{\iota}^{\natural}\pi^* L) \ar[r]^-{\tilde{\iota}_*\rho} & \GW^{[n]}(Z, E\pi^*L),
        }
    \]
    commutes up to homotopy where $\rho: 1 \to E$ is the natural isomorphism sending any line bundles $L$ to its associated residue complex $EL$ and where 	$
        \gamma: \tilde{\pi}^*\iota^{\natural} L \to \tilde{\iota}^{\natural}\pi^*L
    $ is the natural isomorphism. By safely omitting the mentions of $\rho$ and $\gamma$, the formula
    \[
        \pi^* \iota_* = \tilde{\iota}_* \tilde{\pi}^*
    \]
    holds.
\end{proposition}
\begin{proof}
    By the d\'evissage theorem (cf. \cite[Theorem 5.1]{xie20a}), it suffices to prove that
    \begin{equation}\label{eqn:base_change_two_regular}
        \pi_* \rho \pi^* \rho^{-1} \iota_* \rho = \pi_*  \tilde{\iota}_* \rho \tilde{\pi}^*
    \end{equation}
    The projection formula (cf. Proposition \ref{prop:projection_formula_finite}) with $\alpha=\rho(1)$ implies that the left hand side of \eqref{eqn:base_change_two_regular} is equal to
    \[
        \pi_* \rho \pi^* \rho^{-1} \iota_* \rho = \rho (\kappa(\mathcal{E}))\cup (\rho^{-1} \iota_* \rho ) = \kappa(\mathcal{E})\cup \iota_* \rho,
    \]
    where $\kappa(\mathcal{E})$ is the Koszul complex associated to the section $s: \mathcal{E}^\vee \to \SO_Y$. The first equality holds by the projection formula and $\iota_*(\rho(1)) = \rho(\kappa(\mathcal{E}))$, which follows from \cite[\S 6.7]{huang2023the}, even though $Z$ is not assumed to be regular.
    By the projection formula and \cite[\S 6.7]{huang2023the} again, the right-hand side of \eqref{eqn:base_change_two_regular} is equal to
    \[
        \pi_*  \tilde{\iota}_* \rho \tilde{\pi}^* = \iota_* \tilde{\pi}_* \rho \tilde{\pi}^* = \iota_* (\rho (\kappa(\iota^*\mathcal{E}))\cup -) = \iota_* ( \kappa(\iota^*\mathcal{E})\cup \rho) = \iota_* ( \iota^*\kappa(\mathcal{E})\cup \rho) = \kappa(\mathcal{E}) \cup \iota_* \rho.
    \]
    The result follows.
\end{proof}

\subsection{On the flat base change formula}
Let
\[
    \xymatrix{{X'} \ar[r]^-{\bar{g}} \ar[d]_-{\bar{f}}& X \ar[d]_-{f}\\
    Y' \ar[r]^-{g}& Y}
\]
be a fiber square of schemes, where $f$ is finite and $g$ is flat. We have a canonical natural transformation
$$\varepsilon : g^* f_* \to \bar{f}_*\bar{g}^*$$
from $\M(X)$ to $\M(Y')$, which becomes isomorphic when restricted to the subcategory of quasi-coherent sheaves, cf. \cite[Section 12.2, p.327]{gortz2020algebraic}.

\begin{lemma}\label{lma:gamma}
    Suppose that there are residue complex on $X$, and the pullback $g^*, \bar{g}^*$ preserve dualizing complexes (e.g. when all the four schemes are Gorenstein).
    Let $\mathrm{Res}(Y)$(resp. $\mathrm{Res}(X')$) be the additive category of residue complexes on $Y$ (resp. $X'$), cf. \cite[p.304]{hartshorne1966residues} (see also \cite[Definition 2.2]{huang2023the}). There is a natural isomorphism
    \[
        \gamma: E\bar{g}^\ast f^\Delta \xrightarrow{\cong} \bar{f}^\Delta Eg^\ast
    \]
    from $\mathrm{Res}(Y)$ to $\mathrm{Res}(X')$ such that the following diagram
    \[
        \xymatrix{
        \bar{f}_\ast \bar{g}^\ast f^\Delta \ar[d]_-{\rho} \ar[r]^-{\varepsilon^{-1}}_-{\cong}& g^\ast f_\ast f^\Delta \ar[r]^-{\mathrm{Tr}} & g^\ast \ar[d]^-{\rho}  \\
        \bar{f}_\ast E \bar{g}^\ast f^\Delta \ar[r]^-{\gamma} & 	\bar{f}_\ast \bar{f}^\Delta Eg^\ast  \ar[r]^-{\mathrm{Tr}} & Eg^\ast
        }
    \]
    commutes.
\end{lemma}
\begin{proof}
    Note that $f$ is finite, by \cite[Theorem III.6.3]{hartshorne1966residues}, the isomorphism exists in derived categories, and can be uniquely lifted to an isomorphism between residue complexes. The statement and its proof in \cite[Lemma 4.8]{huang2023the} remain valid when $\tilde{J}$ is more generally considered a Cousin complex with respect to the same filtration associated with $J$. By \cite[Lemma VI.4.1]{hartshorne1966residues}, for any residue complex $I$ on Y, $f_\ast f^\Delta I$ is a Cousin complex with respect to the same filtration associated with $I $. Therefore there is an isomorphism
    \[
        \mathrm{Hom}_{Z^0\Ch^b_c(\M(X))}(g^\ast f_\ast f^\Delta I, Eg^\ast I) \cong \mathrm{Hom}_{D^b_c(\M(X))}(g^\ast f_\ast f^\Delta I, Eg^\ast I).
    \]
    This implies that the commutative diagram in \cite[Proposition III.6.6(2)]{hartshorne1966residues} can be uniquely lifted to the diagram given in this lemma.
\end{proof}

\begin{theorem}[Flat base-change formula]\label{thm:flat_pullback_finite_pushforward}
    With the assumption in Lemma \ref{lma:gamma}, the diagram
    \[
        \xymatrix{\GW^{[n]}\left(X^\prime, \bar{f}^\Delta Eg^\ast   I\right) \ar[d]_{\bar{f}_\ast} & \ar[l]_{\gamma}^{\simeq} \GW^{[n]}\left(X^\prime, E\bar{g}^\ast f^\Delta  I\right) & \ar[l]_(0.43){\bar{g}^\ast} \GW^{[n]}(X, f^\Delta I)  \ar[d]_{f_\ast} \\
        \GW^{[n]}(Y', Eg^\ast I) && \ar[ll]_-{g^\ast} \GW^{[n]}(Y, I)}
    \]
    of Grothendieck-Witt spectra commutes up to homotopy.
\end{theorem}
\begin{proof}
    All we need to do is demonstrate that the diagram
    \[
        \xymatrix{\left(\Ch^b_{c}(\M(X')), \sharp_{\bar{f}^\Delta Eg^\ast I} \right) \ar[d]_{(\bar{f}_\ast, \zeta)} & \ar[l]_{(\mathrm{id}, \gamma)} \left((\Ch^b_{c}(\M(X')), \sharp_{E\bar{g}^\ast f^\Delta  I}\right) & \ar[l]_-{(\bar{g}^*, \widetilde{\beta})} \left(\Ch^b_{c}(\M(X)) \sharp_{f^\Delta I}\right)  \ar[d]_{(f_\ast, \zeta)} \\
        \left(\Ch^b_{c}(\M(Y')), \sharp_{Eg^\ast I}\right) && \ar[ll]_-{(g^\ast, \widetilde{\beta})} \left(\Ch^b_{c}(\M(Y)), \sharp_{I}\right)}
    \]
    commutes up to the natural quasi-isomorphism $\varepsilon$.\ The argument is finished, if we can prove the diagram
    \[
        \xymatrix{g^\ast f_\ast [A, f^\Delta I] \ar[r]^-{\zeta} \ar[d]_-{\varepsilon} & g^\ast [f_\ast A, I] \ar[rr]^-{\widetilde{\beta}} && [g^\ast f_\ast A, Eg^\ast I] \\
        \bar{f}_\ast \bar{g}^\ast[A, f^\Delta I] \ar[r]^-{\widetilde{\beta}} & \bar{f}_\ast [\bar{g}^\ast A, E\bar{g}^\ast f^\Delta I] \ar[r]^-{\gamma} & \bar{f}_\ast[\bar{g}^\ast A, \bar{f}^\Delta Eg^\ast I]  \ar[r]^-{\zeta} & [\bar{f}_\ast \bar{g}^\ast A, Eg^\ast I] \ar[u]^-{\varepsilon^{\sharp}}
        }
    \]
    commutes for every $A\in \Ch^b_{c}(\M(X))$. This follows from the commutative diagram
    $$
        \xymatrix@C=14pt{\bar{f}_\ast \bar{g}^\ast[A, f^\Delta I] \ar[d]_-{\beta} \ar@{}[rrrrd]!UL|-{\diagram \label{diag:base_dual_1}} &&&& g^\ast f_\ast [A, f^\Delta I] \ar[d]_-{\phi} \ar[llll]_-{\varepsilon}\\
        \bar{f}_\ast [\bar{g}^\ast A, \bar{g}^\ast f^\Delta I] \ar[d]_-{\rho} \ar[r]^-{\phi} & [\bar{f}_\ast \bar{g}^\ast A, \bar{f}_\ast \bar{g}^\ast f^\Delta I] \ar[d]_-{\rho} \ar[r]^-{\varepsilon^{-1}} \ar@{}[ddr]|-{\diagram \label{diag:base_dual_2}} & [\bar{f}_\ast \bar{g}^\ast A, g^\ast f_\ast f^\Delta I] \ar[d]_-{\mathrm{Tr}} \ar[r]^-{\varepsilon^{\sharp}} & [g^\ast f_\ast A, g^\ast f_\ast f^\Delta I] \ar[d]_-{\mathrm{Tr}} & g^\ast [f_\ast A, f_\ast f^\Delta I] \ar[d]_-{\mathrm{Tr}} \ar[l]_-{\beta}\\
        \bar{f}_\ast[\bar{g}^\ast A, E\bar{g}^\ast f^\Delta I] \ar[r]^-{\phi} \ar[d]_-{\gamma} & [\bar{f}_\ast \bar{g}^\ast A, \bar{f}_\ast E\bar{g}^\ast f^\Delta I] \ar[d]_-{\gamma} & [\bar{f}_\ast \bar{g}^\ast A, g^\ast I] \ar[r]^-{\varepsilon^{\sharp}} \ar[d]_-{\rho} & [g^\ast f_\ast A, g^\ast I] \ar[d]_-{\rho} & g^\ast [f_\ast A, I] \ar[l]_-{\beta} \\
        \bar{f}_\ast[\bar{g}^\ast A, \bar{f}^\Delta Eg^\ast I] \ar[r]^-{\phi} & [\bar{f}_\ast \bar{g}^\ast A, \bar{f}_\ast \bar{f}^\Delta Eg^\ast I] \ar[r]^-{\mathrm{Tr}} & [\bar{f}_\ast \bar{g}^\ast A, Eg^\ast I] \ar[r]^-{\varepsilon^{\sharp}} & [g^\ast f_\ast A, Eg^\ast I]}
    $$
    where $\diag{\ref{diag:base_dual_1}}$ commutes by \cite[Lemma 4.6.5]{lipman2009notes}, $\diag{\ref{diag:base_dual_2}}$ commutes by Lemma \ref{lma:gamma}, and other squares commute by naturality.
\end{proof}
\begin{remark}
    We conjecture that Theorem \ref{thm:flat_pullback_finite_pushforward} remains valid when the map $f$ is more generally proper and the map $g$ is flat. However, we have not yet been able to locate a fully rigorous proof of the base change compatibility of the trace map, as stated in \cite[Corollary VII.3.4(b)(TRA4), p.383]{hartshorne1966residues}, which plays a crucial role in the argument. There are partial results in this direction, such as \cite[Theorem 3.6.5]{conrad2000grothendieck} - which is a central goal of that work - for the case where $f$ is a proper Cohen–Macaulay map. However, the general case where $f$ is proper remains out of reach.
\end{remark}
\begin{lemma}\label{lma:fh_codim_one}
    Let $X,Z$ and $S$ be regular schemes. Consider the diagram
    \[
        \xymatrix{
        & E \ar[d]_-{f} \ar[dl]_-{g} \ar[dr]^-{p}\\
        X & Z \ar[l]_-{h} \ar[r]^-{q}& S
        }
    \]
    where both triangles are commutative. Assume that $f$ is a regular embedding of codimension one, and that $g,h$ are proper morphisms of relative dimension $n,n+1$ respectively.
    Suppose that $f^*:\Pic(Z)/2 \to \Pic(E)/2$ is injective and there is a line bundle $L$ over $X$ and a line bundle $L'$ over $S$ such that $p^*L' = g^{\natural}L$ in $\Pic(E)/2$. Then $q^*L' = h^{\natural}L \otimes \SO(E)$ in $\Pic(Z)/2$
    and the following diagram
    \[
        \xymatrix{
        \GW^{[m+n]}(S,L') \ar[r]^-{q^*} \ar[d]_-{p^*} & \GW^{[m+n]}(Z,q^*L') \ar[r]^-{F} & K(Z) \ar[r]^-{H} & \GW^{[m+n+1]}(Z, h^{\natural}L) \ar[d]_-{h_*}\\
        \GW^{[m+n]}(E,p^*L') \ar[r]^-{\simeq} & \GW^{[m+n]}(E,g^{\natural}L)  \ar[r]^-{\rho} & \GW^{[m+n]}(E,g^{\Delta}EL) \ar[r]^-{g_*}& \GW^{[m]}(X, L)
        }
    \]
    commutes.
\end{lemma}
\begin{proof}
    Firstly, we have
    \[
        f^* q^* L' = p^*L' = g^{\natural}L = f^*h^{\natural}L \otimes \omega_{f}  = f^*(h^{\natural}L \otimes \SO(E))
    \]
    in $\Pic(E)/2$. Secondly, by injectivity assumption of $f^*$, we have $q^*L' = h^{\natural}L \otimes \SO(E)$ in $\Pic(Z)/2$.
    By Proposition \ref{prop:projection_formula_finite}, we observe that
    \[
        g_* \rho p^* = h_* f_* \rho f^* q^* = h_* (f_*(1) \cup q^*)
    \]
    in $\GW^{[m]}(X,L)$.\ Since $f$ is a regular embedding of codimension one,
    \[
        f_*(1) = e(\SO(E)) = H(1)
    \]
    in $\GW^{[1]}_{0}(Z, \SO(E)^{\vee})$. Finally, by \cite[Lemma 2.18]{huang2023the}, we conclude
    \[
        h_* (f_*(1) \cup q^*) = h_* (H(1) \cup q^*) = h_* HF q^*.
    \]
    The result follows.
\end{proof}

%%%%%%%%%%%%%%%%%%%% On the connecting homomorphism via two step blow-up %%%%%%%%%%%%%%%%%%%%
\section{On the connecting homomorphism via two step blow-up}\label{sec:two_step_blow_up}

The connecting homomorphism for Witt groups via a one-step blow-up was originally studied in \cite{balmer2009geometric} (see also \cite{huang2023the} for the corresponding case in Hermitian $K$-theory). However, in practice, many examples do not satisfy the one-step blow-up assumption, for instance, the Lagrangian Grassmannians considered in this paper. To overcome this issue, Martirosian \cite{martirosian2021witt} extended the setup to two-step blow-ups in the context of Witt groups. In this section, we undertake a further study of the connecting homomorphism in Hermitian $K$-theory via two-step blow-ups.

\subsection{Setup}\label{sect:Setup} Let $\iota: Z_1 \hookrightarrow X$ be a regular immersion of codimension $c_1 \geq 2$. Let $B_1$ be the blow-up of $X$ along $Z_1$, and $E_1$ be the exceptional fiber. Let $U_1 := X - Z_1 \cong B_1 - E_1$ be the unaltered open complement. Furthermore, we assume that $\kappa:Z_2 \to E_1$ is a closed embedding such that the composition $\tilde{\iota} \circ \kappa: Z_2 \hookrightarrow B_1$ is a regular immersion  of codimension $c_2 \geq 2$. Let $B_2$ be the blow-up of $B_1$ along $Z_2$, and $E_2$ be the exceptional fiber. Define $U_1 := B - Z_2 \cong B_2 - E_2$ as the unaltered open complement. We have a commutative diagram
\[
    \xymatrix{
    &Z_1 \ar[r]^-{\iota_1} &X & \ar[l]_-{v_1} U_1 \ar[ld]^(.45){\tilde{v}_1} \ar@{=}[r] & U_{1} \ar[ld]^-{w}\\
    Z_2 \ar[r]^-{\kappa} \ar@/_1pc/[rr]_-{\iota_2} & E_1 \ar[r]^-{\tilde{\iota}_1} \ar[u]^-{\tilde{\pi}_1} & B_1 \ar[u]^-{\pi_1} & \ar[l]^-{v_2} U_2 \ar[ld]^-{\tilde{v}_2}\\
    E_2 \ar[rr]^-{\tilde{\iota}_2} \ar[u]^-{\tilde{\pi}_2} && B_2 \ar[u]^-{\pi_2}
    }
\]
with the usual morphisms.

\subsection{Hypothesis}\label{sect:hypothesis}
Suppose that there exists a scheme $Y$ and an auxiliary  morphism $\alpha: B_2 \to Y$ such that the composition $\vartheta := \alpha \circ \tilde{v}_2 \circ  w_1$ forms an affine bundle, i.e. a Zariski locally trivial bundle. This setup can be illustrated in the following diagram:
\begin{equation}\label{eq:XUBY}
    \xymatrix{
    &Z_1 \ar[r]^-{\iota_1} &X &\ar[l]_-{v_1} U_1 \ar[ld]^(.45){\tilde{v}_1} \ar@{=}[r] & U_{1} \ar[ld]^-{w} \ar[dd]^{\vartheta} \\
    Z_2 \ar[r]^-{\kappa} \ar@/_1pc/[rr]_-{\iota_2} & E_1 \ar[r]^-{\tilde{\iota}_1} \ar[u]^-{\tilde{\pi}_1} & B_1 \ar[u]^-{\pi_1} & \ar[l]^-{v_2} U_2 \ar[ld]^-{\tilde{v}_2}\\
    E_2 \ar[rr]^-{\tilde{\iota}_2} \ar[u]^-{\tilde{\pi}_2} && B_2 \ar[u]^-{\pi_2}  \ar@{-->}[rr]^-{\alpha}&& Y
    }
\end{equation}
Set $\pi= \pi_1 \circ \pi_2$ and $\tilde{\pi} = \alpha \circ \tilde{\iota}_2$.
\begin{proposition}[Oriented cohomology theory]\label{prop:oriented-cohomology}
    Under Setup \ref{sect:Setup} and Hypothesis \ref{sect:hypothesis}.\ Assume that we have an oriented cohomology theory with support $H^*$, which is homotopy invariant for regular schemes. Suppose further that $H^*$ admits push-forward along proper morphisms satisfying the flat base-change formula.
    Then, the localization long exact sequence  \vspace{-7pt}
    \begin{equation}\label{eq:local-X-Z}
        \xymatrix{ \cdots \ar[r]^-{\partial} & H^*_{Z_1}(X) \ar[r] & H^*(X) \ar[r]^-{v_1^*} & H^*(U_1) \ar[r]^-{\partial} & H^{*+1}_{Z_1}(X) \ar[r] & \cdots \cdots }
    \end{equation}
    reduces to split short exact sequences
    \vspace{-7pt}
    $$\xymatrix{0 \ar[r] & H^*_{Z_1}(X) \ar[r] & H^*(X) \ar[r]^-{v_1^*} & H^*(U_1) \ar[r] & 0 } \vspace{-7pt} $$
    with a right splitting given by $ \pi_{*} \circ \alpha^* \circ (\vartheta^*)^{-1}$.
\end{proposition}
\begin{proof}
    Note that
    $$v_1^* \circ \pi_* = v_1^* \circ {\pi_1}_* \circ {\pi_2}_* = \tilde{v}_1^* \circ {\pi_2}_* = w^* \circ v_2^* \circ {\pi_2}_* = w^* \circ \tilde{v}_2^*$$
    where for the second and last equality we use the base change formula. It follows that
    $ v_1^* \circ \pi_* \circ \alpha^* = w^* \circ \tilde{v}_2^*  \circ \alpha^* = \vartheta^*.$ The result follows since $\vartheta$ is an affine bundle, whence $\vartheta^*$ is an isomorphism.
\end{proof}

\begin{lemma}
    The following statements hold for Picard groups of the blow-ups $B_1$ and $B_2$.
    \begin{itemize}
        \item [(i)] The map
              \[
                  (f, \pi_1^*) : \mathbb{Z} \oplus \Pic(X)  \rightarrow  \Pic(B_1)
              \]
              is an isomorphism, where the first map $f$ sends an integer $n$ to $\SO(E_1)^{\otimes n}$.
        \item [(ii)] The map
              \[
                  (f,g,\pi^*): \mathbb{Z} \oplus \mathbb{Z} \oplus \Pic(X)  \rightarrow  \Pic(B_2)
              \]
              is an isomorphism, where the first map $f$ sends an integers $n$ to $\pi_2^*\SO(E_1)^{\otimes n}$, while the second map $g$ sends an integer $m$ to $\SO(E_2)^{\otimes m}$.
    \end{itemize}
\end{lemma}
\begin{proof}
    For the proof of (i), we refer to \cite[p. 115 Proposition 6.7]{fulton1998intersection}, in view of the condition $c_1 \geq 2$. See also \cite[Appendix A.6]{balmer2009geometric}. For the proof of (ii), note that the map $(g,\pi_2^*): \mathbb{Z} \oplus \Pic(B_1) \to \Pic(B_2) $ is an isomorphism by $(i)$, and we use the compatibility of pullbacks with compositions.
\end{proof}
The right-hand square of Diagram (\ref{eq:XUBY}) yields the following diagram on Picard groups
\[
    \Pic\left( \vcenter{\xymatrix@R=35pt@C=35pt{X & U_1 \ar@{_{(}->}[l]_-{v_1} \ar[d]^-{\vartheta} \ar@{_{(}->}[dl]|-*+{\scriptstyle{\tilde{v}_{2}\circ w}} \\
            B_2 \ar[u]^-{\pi} \ar[r]_-{\alpha} & Y}} \right) \cong \raisebox{-9pt}{$\vcenter{\xymatrix@R=35pt{\Pic(X)\ar@{=}[r] \ar[d]_{\left(\begin{smallmatrix}1\\0\\0\end{smallmatrix}\right)} & \Pic(X) \ar@{=}[d]  \\ \Pic(X)\oplus \mathbb{Z} \oplus \mathbb{Z} \ar[ur]|-{(\begin{smallmatrix}1&0&0\end{smallmatrix})} & \Pic(X) \ar[l]^-{\left(\begin{smallmatrix}1\\ \lambda_1 \\ \lambda_2 \end{smallmatrix}\right)}}}$}
\]
where the map $\Pic(X) \cong \Pic(Y) \xrightarrow{\alpha^*} \Pic(B_2) \cong \Pic(X)\oplus \mathbb{Z} \oplus \mathbb{Z}$ must be of the form $\left(\begin{smallmatrix}1\\ \lambda_1\\ \lambda_2\end{smallmatrix}\right)$.

In this paper, we fix the following notation on various line bundles. If $v:U\to X$ is a morphism of schemes and $L$ is a line bundle on $X$, then we define $L_{\mid U}:= v^*L$. Furthermore, we additionally set
$$L_Y := (\vartheta^*)^{-1} L_{\mid U_1}, \quad\quad L_{B_2} := \alpha^*L_Y , \quad\quad L_{E_2} := \tilde{\pi}^*L_Y $$
where $Y,B_2,E_2$ do not admit a morphism to $X$. Note that
\begin{equation}\label{eqn:two_step_blow_up_lambda}
    L_{B_2} \cong \pi^* L \otimes \pi_2^*\SO(E_1)^{\otimes \lambda_1(L)} \otimes \SO(E_2)^{\otimes \lambda_2(L)} .
\end{equation}

\subsection{On the localization sequences}  The blow-up setup (Setup \ref{sect:Setup}) together with pushforwards and pullbacks bring us the following sequence of maps of localization sequences of Grothendieck-Witt spectra

\begin{equation}\label{eqn:seq-of-loc-seq}
    \xymatrix@C=15pt{
    \GW^{[m-c_1]}(Z_1, \iota_1^{\natural}L) \ar[r]^-{\iota_{1*}} & \GW^{[m]}(X, L) \ar[r]^-{v_1^*} & \GW^{[m]}(U_1, L_{\mid U_1}) \ar[r]^-{\partial_{1}} & \GW^{[m-c_1]}(Z_1, \iota_1^{\natural}L)[1]\\
    \GW^{[m-1]}(E_1, \tilde{\iota}_1^{\natural}\pi^{\natural}_1L) \ar[r]^-{\tilde{\iota}_{1*}} \ar[u]^-{\tilde{\pi}_{1\ast}} & \GW^{[m]}(B_1, \pi_1^{\natural}L) \ar[r]^-{\tilde{v}_1^*} \ar[u]^-{\pi_{1\ast}} & \GW^{[m]}(U_1, (\pi_1^{\natural}L)_{\mid U_1}) \ar[r]^-{\tilde\partial_{1}} \ar[u]^-{\simeq} & \GW^{[m-1]}(E_1, \tilde{\iota}_1^{\natural}\pi^{\natural}_1L)[1] \ar[u]^-{\tilde{\pi}_{1\ast}} \\
    \GW^{[m-c_2]}(Z_2, \iota_2^{\natural}\pi_1^{\natural}L) \ar[r]^-{\iota_{2*}} \ar[u]^-{\kappa_{\ast}} & \GW^{[m]}(B_1, \pi_1^{\natural}L) \ar[r]^-{v_2^*} \ar@{=}[u] & \GW^{[m]}(U_2, (\pi_1^{\natural}L)_{\mid U_2}) \ar[r]^-{\partial_{2}} \ar[u]^-{w^{\ast}} & \GW^{[m-c_2]}(Z_2, \iota_2^{\natural}\pi_1^{\natural}L) [1] \ar[u]^-{\kappa_{\ast}} \\
    \GW^{[m-1]}(E_2, \tilde{\iota}_2^{\natural} \pi^{\natural}L) \ar[r]^-{\tilde{\iota}_{2*}} \ar[u]^-{\tilde{\pi}_{2\ast}} & \GW^{[m]}(B_2, \pi^{\natural}L) \ar[r]^-{\tilde{v}_2^*} \ar[u]^-{\pi_{2\ast}} & \GW^{[m]}(U_2,(\pi^{\natural}L)_{\mid U_2}) \ar[r]^-{\tilde\partial_{2}} \ar[u]^-{\simeq} & 	\GW^{[m-1]}(E_2, \tilde{\iota}_2^{\natural} \pi^{\natural}L)[1] \ar[u]^-{\tilde{\pi}_{2\ast}}
    }
\end{equation}
\begin{lemma}\label{lma:seq-of-loc-seq}
    Every square in Diagram \eqref{eqn:seq-of-loc-seq} commutes up to homotopy.
\end{lemma}
\begin{proof}
    The first and the last horizontal ladder diagrams commute by \cite[Lemma 7.8]{huang2023the}. We explain the commutativity of the second horizontal ladder diagram, which stems from the following ladder diagram
    \begin{equation}\label{eq:ladder-categorical}
        \xymatrix@C=20pt{
        \GW^{[m]}_{E_1}(B_1, \pi^{\natural}_1L) \ar[r]^-{\mathrm{ext}} & \GW^{[m]}(B_1, \pi^{\natural}_1L) \ar[r]^-{\tilde{v}_1^*} & \GW^{[m]}(U_1, (\pi^{\natural}_1L)_{\mid U_1}) \ar[r]^-{\partial} & \GW^{[m]}_{E_1}(B_1, \pi^{\natural}_1L)[1] \\
        \GW^{[m]}_{Z_2}(B_1, \pi^{\natural}_1L) \ar[r]^-{\mathrm{ext}} \ar[u]^-{\mathrm{ext}} & \GW^{[m]}(B_1, \pi^{\natural}_1L) \ar[r]^-{v_2^*} \ar@{=}[u] & \GW^{[m]}(U_2, (\pi^{\natural}_1L)_{\mid U_2}) \ar[r]^-{\partial} \ar[u]^-{w^{\ast}} & \GW^{[m]}_{Z_2}(B_1, \pi^{\natural}_1L)[1] \ar[u]^-{\mathrm{ext}}
        }
    \end{equation}
    of localization sequences, where the maps $\mathrm{ext}$ represent the extension of support. All the squares in the ladder diagram \eqref{eq:ladder-categorical} commute for categorical reasons similar to those in the proof of \cite[Lemma 7.8]{huang2023the}. To illustrate, given a pretriangulated dg category with weak equivalences and duality $(\mathcal{A},\mathrm{quis})$ and let $\mathrm{quis} \subset v \subset w \subset Z^{0}\mathcal{A}$ be two larger set of weak equivalences such that both $(\mathcal{A},w)$ and $(\mathcal{A},v)$ are also pretriangulated dg categories with weak equivalences and dualities. Then we have the following canonical map of squares:

    \[
        \xymatrix@R=10pt@C=10pt{
        & (\mathcal{A}, \mathrm{quis}) \ar[rr]^-{=} \ar'[d][dd]
        & & (\mathcal{A}, \mathrm{quis}) \ar[dd]
        \\
        (\mathcal{A}^{v}, \mathrm{quis}) \ar[ur]\ar[rr]\ar[dd]
        & & (\mathcal{A}^w, \mathrm{quis}) \ar[ur]\ar[dd]
        \\
        & (\mathcal{A}, v) \ar'[r][rr]
        & & (\mathcal{A}, w)
        \\
        (\mathcal{A}^{v},v) \ar[rr]\ar[ur]
        & & (\mathcal{A}^w,w) \ar[ur]
        }
    \]

    The commutativity of squares in Diagram \eqref{eq:ladder-categorical} can be obtained by setting $\mathcal{A}=\Ch^b(\V(B_1))$, defining $\mathrm{quis} \subset Z^0 \mathcal{A}$ as the set of quasi-isomorphisms, and letting $v \subset Z^0 \mathcal{A}$ (resp. $w \subset Z^0 \mathcal{A}$) be the set of morphisms which become quasi-isomorphisms after pullback along $v_2$ (resp. $\tilde{v}_1$).

    In particular, the commutativity of the third square in the second row of Diagram \eqref{eqn:seq-of-loc-seq} is a consequence of the commutativity of the following squares:
    \[
        \xymatrix{
        \GW^{[m]}(U_1, (\pi^{\natural}_1L)_{\mid U_1}) \ar[r]^-{\partial} & \GW^{[m]}_{E_1}(B_1, \pi^{\natural}_1L)[1] & \GW^{[m-1]}(E_1, \tilde{\iota}_{1}^{\natural}\pi^{\natural}_1L)[1] \ar[l]_(.49){\tilde{\iota}_{1*}}^(.49){\simeq}\\
        \GW^{[m]}(U_2, (\pi^{\natural}_1L)_{\mid U_2}) \ar[r]^-{\partial} \ar[u]^-{w^*} & \GW^{[m]}_{Z_2}(B_1, \pi^{\natural}_1L)[1] \ar[u]^-{\mathrm{ext}} & \GW^{[m-c_2]}(Z_2, \iota_{2}^{\natural}\pi^{\natural}_1L)[1] \ar[l]_(.49){\iota_{2*}}^(.49){\simeq} \ar[u]^-{\kappa_*}
        }
    \]
    where the right square commutes by the fact that pushforwards are compatible with composition.
\end{proof}

\subsection{On the geometric description of the connecting homomorphism}
The numbers $\lambda_1$ and $\lambda_2$, defined in Equation \eqref{eqn:two_step_blow_up_lambda}, allow us to study the connecting homomorphism $\partial_1$ in Diagram \eqref{eqn:seq-of-loc-seq}, which is the aim of this section. We record the following result:

\begin{theorem}\label{thm:two_step_blow_up}
    Suppose that the assumptions of Setup \ref{sect:Setup} and Hypothesis \ref{sect:hypothesis} hold. Then the following statements hold true.
    \begin{enumerate}[leftmargin=30pt]
        \item If $\lambda_1(L)\equiv c_1-1 \bmod 2$ and $\lambda_2(L)\equiv c_2-1 \bmod 2$, then the sequence of abelian groups
              \[\xymatrix{0 \ar[r] & \GW^{[m-c_1]}_i(Z_1, \iota_{1}^{\natural} L) \ar[r]^-{\iota_{1*}} & \GW^{[m]}_i(X,L) \ar[r]^-{v_1^*} & \GW^{[m]}_i(U_1,L_{\mid U_1}) \ar[r] & 0 }\]
              is split exact for each $i \in \mathbb{Z}$, with an explicit right splitting given by $\pi_{\ast} \circ \alpha^{\ast} \circ (\vartheta^*)^{-1}$.
        \item If $\lambda_1(L)\equiv c_1-1 \bmod 2$ and $\lambda_2(L)\equiv c_2 \bmod 2$, then the diagram

              \[
                  \xymatrix{
                  \GW^{[m]}(U_1, L_{\mid U_1}) \ar[r]^-{\partial_1} \ar[d]_-{(\vartheta^*)^{-1}}^-{\simeq} & \GW^{[m-c]}(Z_1, \iota_{1}^{\natural} L)[1] & \GW^{[m-c_1+1]}(Z_1, \iota_{1}^{\natural} L) \ar[l]_-{\eta \cup}\\
                  \GW^{[m]}(Y, L_Y) \ar[r]^-{\tilde{\pi}^{\ast}} & \GW^{[m]}(E_{2}, L_{E_2}) \ar[r]^-{\simeq} & \GW^{[m]}(E_{2}, \tilde{\iota}_2^{\natural} \pi^{\natural}L) \ar[u]^-{(\tilde{\pi}_1 \kappa \tilde{\pi}_2)_{\ast}}
                  }
              \]
              commutes.
        \item If $\lambda_1(L)\equiv c_1 \bmod 2$ and $\lambda_2(L)\equiv c_2-1 \bmod 2$, then the diagram
              \[
                  \xymatrix{
                  \GW^{[m]}(U_1, L_{\mid U_1}) \ar[r]^-{\partial_1} \ar[d]_-{\alpha^{\ast} (\vartheta^*)^{-1}} & \GW^{[m-c_1]}(Z_1, \iota_{1}^{\natural} L)[1] & \GW^{[m-c_1+1]}(Z_1, \iota_{1}^{\natural} L) \ar[l]_-{\eta \cup}\\
                  \GW^{[m]}(B_2, L_{B_2}) \ar[r]^-{\pi_{2 \ast}} & \GW^{[m]}(B_1, \SO(E_{1}) \otimes \pi_1^{\natural}L) \ar[r]^-{\tilde{\iota}_{1}^{\ast}} & \GW^{[m]}(E_{1},\tilde{\iota}_1^{\natural}\pi^{\natural}_1L) \ar[u]^-{\tilde{\pi}_{1\ast}}
                  }
              \]
              commutes where $ L_{B_2} \cong \SO(E_1)\otimes \pi^\natural L$ on $B_2$ and where $\tilde{\iota}_{1}^{\ast}(\SO(E_{1}) \otimes \pi_1^{\natural}L) \cong \tilde{\iota}_1^{\natural}\pi^{\natural}_1L$ on $E_1$.
    \end{enumerate}
\end{theorem}
To prove this theorem we recall the following result:
\begin{lemma}\label{lem:connecting-codim-one}
    Let $B$ be a scheme with a prime divisor $\tilde\iota : E \hookrightarrow B$. Let $\SO(E)$ be the line bundle on $B$ associated to $E$. Let $\tilde{v}:U\to B$ be the open immersion of the open complement $U=B\backslash E$ and let $L_{\mid U}:=\tilde{v}^* L$ for any $L \in \Pic(B)$. Then  the connecting homomorphism $\partial$ in the localization sequence
    \[
        \GW^{[m-1]}(E, \tilde\iota^{\natural}L) \xrightarrow{\tilde\iota_*} \GW^{[m]}(B, L) \xrightarrow{\tilde v^*} \GW^{[m]}(U, L_{\mid U}) \xrightarrow{\partial} \GW^{[m-1]}(E, \tilde\iota^{\natural}L)[1],
    \]
    fits into the following commutative diagram:
    \[
        \xymatrix{\GW^{[m]}(B, L \otimes \SO(E)) \ar[r]^-{\tilde\iota^\ast} \ar[d]_-{\tilde v^\ast} & \GW^{[m]}(E, \tilde \iota^\ast(L \otimes \SO(E))) \ar[r]^-{\simeq} & \GW^{[m]}(E,  \tilde\iota^{\natural}L) \ar[d]_-{\eta \cup} \\
        {\GW^{[m]}(U,\tilde v^\ast \left(L\otimes \SO(E)\right))} \ar[r]^-{\simeq} & {\GW^{[m]}\left(U,L_U\right)} \ar[r]^-{\partial} &  {\GW^{[m-1]}(E, \tilde\iota^{\natural}L)}[1] }
    \]
\end{lemma}
\begin{proof}
    See \cite[Lemma 7.6]{huang2023the}.
\end{proof}

\begin{proof}[Proof of Theorem \ref{thm:two_step_blow_up}]
    (i). This case can be treated similarly to the oriented cohomology theory case, cf. Proposition \ref{prop:oriented-cohomology}. Indeed, in order to makes sure the splitting  $\pi_{\ast} \circ \alpha^{\ast} \circ (\vartheta^*)^{-1}$ is well-defined, we need to check the twists are compatible. In other words, we must check that we can compose the map
    $$\alpha^{\ast} \circ (\vartheta^*)^{-1}: \GW^{[m]}(U_1,L_{\mid U_{1}}) \to \GW^{[m]}(B_{2}, L_{B_2})$$ with the pushforward $$\pi_*: \GW^{[m]}(B_2, \pi^{\natural}L) \to \GW^{[m]}(X, L).$$
    By \cite[Proposition A.11(iii)]{balmer2009geometric}, we have
    \begin{equation}\label{eqn:relative_twist_pi}
        \pi^{\natural}L = \pi^*
        L \otimes \omega_{\pi} \cong  \pi^*
        L \otimes \pi_{2}^*\omega_{\pi_{1}} \otimes \omega_{\pi_2} \cong \pi^*
        L \otimes \pi_2^*\SO(E_1)^{\otimes c_1-1} \otimes \SO(E_2)^{\otimes c_2-1}.
    \end{equation}
    The conditions $\lambda_1(L)\equiv c_1-1 \bmod 2$ and $\lambda_2(L)\equiv c_2-1 \bmod 2$ imply that $\pi^{\natural}L = \alpha^*L_Y$ in the group $\Pic(B_2)/2$ by the isomorphism \eqref{eqn:two_step_blow_up_lambda}.

    (ii). By the condition that $\lambda_1(L)\equiv c_1-1 \bmod 2$ and $\lambda_2(L)\equiv c_2 \bmod 2$, we have
    \[
        L_{B_2} \cong \pi^{\natural}L\otimes \SO(E_2) = \pi^*(L) \otimes (\pi_2^*\SO(E_1))^{\otimes c_1 - 1} \otimes \SO(E_2)^{\otimes c_2}
    \]
    in the group $\Pic(B_2)/2$.

    Now, let us examine the following diagram:
    $$
        \xymatrix{
        \GW^{[m]}(Y, L_{Y}) \ar[ddd]_-{\alpha^*} \ar@{}[dddr]|-{\diagram \label{diag:two_step_conn_ii_pullback} } \ar[r]^-{\vartheta^*}_-{\simeq} & \GW^{[m]}(U_1, L_{\mid U_1}) \ar[r]^-{\partial_{1}} \ar@{}[dr]|-{\diagram \label{diag:two_step_conn_ii_partial} } & \GW^{[m-c_1]}(Z_1, \iota_{1}^{\natural}L)[1]\\
        & \GW^{[m]}(U_2,(\pi^{\natural}L)_{\mid U_2}) \ar[r]^-{\tilde\partial_{2}} \ar[u]^-{w^*} \ar@{}[ddr]|-{\diagram \label{diag:two_step_conn_ii_codim_one} } & \GW^{[m-1]}(E_2, \tilde{\iota}_2^{\natural}\pi^{\natural}L)[1] \ar[u]^-{(\tilde{\pi}_1 \kappa \tilde{\pi}_2)_{\ast}} \\
        & \GW^{[m]}(U_2,\tilde{v}_2^*(\pi^{\natural}L\otimes \SO(E_2))) \ar[u]^-{\simeq} & \GW^{[m]}(E_2, \tilde{\iota}_2^{\natural}\pi^{\natural}L) \ar[u]^-{\eta \cup} \\
        \GW^{[m]}(B_2, L_{B_2}) \ar[r]^-{\simeq} & \GW^{[m]}(B_2,\pi^{\natural}L\otimes \SO(E_2)) \ar[u]^-{\tilde{v}_2^*} \ar[r]^-{\tilde{\iota}_2^*} & \GW^{[m]}(E_2,\tilde{\iota}_2^*(\pi^{\natural}L\otimes \SO(E_2))) \ar[u]^-{\simeq}
        }
    $$
    where Diagram $\diag{\ref{diag:two_step_conn_ii_pullback}}$ commutes by the compatibility of pullback with composition, Diagram $\diag{\ref{diag:two_step_conn_ii_partial}}$ commutes by Lemma \ref{lma:seq-of-loc-seq} and Diagram $\diag{\ref{diag:two_step_conn_ii_codim_one}}$ commutes by Lemma \ref{lem:connecting-codim-one}.
    Note that cupping with the Bott element commutes with pushforwards. The result follows.

    (iii). By the condition that $\lambda_1(L)\equiv c_1 \bmod 2$ and $\lambda_2(L)\equiv c_2 - 1 \bmod 2$, we have
    \[
        L_{B_2} \cong \pi_2^*\SO(E_{1}) \otimes \pi^{\natural}L = \pi^*(L) \otimes (\pi_2^*\SO(E_1))^{\otimes c_1} \otimes \SO(E_2)^{\otimes c_2 - 1}
    \]
    in the group $\Pic(B_2)/2$.

    Now, let us examine the following diagram:
    $$\xymatrix{
        \GW^{[m]}(Y, L_Y)  \ar[d]_-{\alpha^*} \ar[r]^-{\vartheta^*}_-{\cong} \ar@{}[ddr]|-{\diagram \label{diag:two_step_conn_iii_pullback}}& \GW^{[m]}(U_1, L_{\mid U_1})  \ar[r]^-{\partial_{1}} \ar@{}[dr]|-{\diagram \label{diag:two_step_conn_iii_partial} } &  \GW^{[m-c_1]}(Z_1, \iota_1^{\natural}L)[1]  \\
        \GW^{[m]}(B_2, L_{B_2}) \ar[d]_-{\simeq} & \GW^{[m]}(U_1, (\pi_1^{\natural}L)_{\mid U_1}) \ar[r]^-{\tilde{\partial}_1}  \ar[u]^-{\simeq} \ar@{}[ddr]|-{\diagram \label{diag:two_step_conn_iii_codim_one} }& \GW^{[m-1]}(E_{1},  \tilde{\pi}_1^{\natural}\iota_1^{\natural}L)[1]  \ar[u]^-{\tilde{\pi}_{1\ast}}\\
        \GW^{[m]}(B_2, \pi_2^*\SO(E_{1}) \otimes \pi^{\natural}L) \ar[r]^-{\tilde{v}_2^*} \ar@{}[dr]|-{\diagram \label{diag:two_step_conn_iii_base_change} }& \GW^{[m]}(U_2, v_2^*(\SO(E_{1}) \otimes  \pi_1^{\natural}L))  \ar[u]_-{w^*}    & \GW^{[m]}(E_{1},  \tilde{\pi}_1^{\natural}\iota_1^{\natural}L) \ar[u]^-{\eta\cup} \\
        \GW^{[m]}(B_2, \pi_2^*\SO(E_{1}) \otimes \pi^{\natural}L) \ar@{=}[u] \ar[r]^-{\pi_{2 \ast}} & \GW^{[m]}(B_1, \SO(E_{1}) \otimes  \pi_1^{\natural}L)  \ar[u]_{v_2^*} \ar[r]^-{\tilde{\iota}_{1}^{\ast}} & \GW^{[m]}(E_{1}, \tilde{\iota}_{1}^{\ast}(\SO(E_{1}) \otimes \pi_1^{\natural}L)) \ar[u]^-{\simeq}
        }$$
    where Diagram $\diag{\ref{diag:two_step_conn_iii_pullback}}$ commutes by the compatibility of pullback with composition, Diagram $\diag{\ref{diag:two_step_conn_iii_partial}}$ commutes by Lemma \ref{lma:seq-of-loc-seq}, Diagram $\diag{\ref{diag:two_step_conn_iii_codim_one}}$ commutes by Lemma \ref{lem:connecting-codim-one}, Diagram $\diag{\ref{diag:two_step_conn_iii_base_change}}$ commutes by flat base change formula, cf. \cite[Proposition 6.12]{huang2023the}.
    Note that cupping with the Bott element commutes with pushforwards. The result follows.
\end{proof}

%%%%%%%%%%%%%%%%%%%% On the pushforward of the trivial bundle %%%%%%%%%%%%%%%%%%%%
\section{On the pushforward of the trivial bundle}\label{sec:pushforward}
In \cite[Lemme VII.3.5, p.441]{SGA6}, it is shown that if $\pi:B \to X$ is the blow-up along a regular closed immersion, then $\mathbf{R}\pi_*\SO_{B} = \SO_X$, which forms the basis for applying the lax similitude in \cite{balmer2012bases} to Grassmannians \cite{balmer2012witt} (cf. \cite[Proposition 3.15]{balmer2012witt}). In this paper, we encounter examples where $\pi: B\to X$ may not be a blow-up, but rather a projective morphism of relative dimension zero, and where $B$ and $X$ may be reducible. The aim of this section is to study if we can still have $\mathbf{R}\pi_*\SO_{B} = \SO_X$ and apply lax similitude in the situation (cf.  Sublemma \ref{sublem:case-(c)} (ii)) that we need.
\subsection{Cohomology of universal bundle}
Suppose that ${V }$ is a vector bundle of rank $n+1$ over $Y$. Let ${P }_n$ be the universal bundle over $p: \Gr(n,{V }) \to Y$.
\begin{proposition}\label{lma:gr_uni_wedge}
    The formulas
    \[
        \mathbf{R}^{i}p_* \big(\wedge^{k}{P }_n(-l)\big) = \begin{cases}
            \SO_Y, & \text{ if } i=k=l, \\
            0      & \text{ otherwise,}
        \end{cases}
    \]
    hold when $0\leq i,k,l \leq n$.
\end{proposition}
\begin{proof}
    \textbf{Step I}. We claim that in the derived category $D^b(\M(Y))$, the complex $\mathbf{R}p_* \big(\wedge^{k}{P }_n(l)\big) (l\geq 0, k >0)$ is isomorphic to the Koszul complex
    \begin{equation}\label{eq:abw}
        \ldots \to 0 \to \wedge^{k}{V } \otimes S^{l}{V } \xrightarrow{d^0} \wedge^{k-1}{V } \otimes S^{l+1}{V } \xrightarrow{d^1} \ldots \xrightarrow{d^{k-1}} \wedge^{0}{V } \otimes S^{l+k}{V } \to 0 \to \ldots
    \end{equation}
    where $\wedge^{k}{V } \otimes S^{l}{V } $ sits on degree zero and where the differential $d^i$ is defined by
    \[
        \begin{aligned}
            (a_{1} \wedge \ldots \wedge a_{k-i})\otimes (b_{1} \cdot \ldots \cdot b_{l+i}) & \mapsto \sum_{j=1}^{k-i} (-1)^{j+1} (a_{1} \wedge \ldots \wedge \widehat{a}_{j} \wedge  \ldots \wedge a_{k-i}) \otimes (a_{j} \cdot b_{1} \cdot \ldots \cdot b_{l+i}),
        \end{aligned}
    \]
    where $\widehat{a}_{j}$ indicates that such term is omitted. In particular when $l=0$, the complex \eqref{eq:abw} is acyclic and hence equal to $0$, cf. \cite[Corollary V.1.15 and p.263]{akin1982schur}. Let us prove by induction on $k$. Consider the universal exact sequence
    \begin{equation}\label{eq:u-exact-seq}
        0\to {P }_n \to p^*{V } \to \SO(1) \to 0
    \end{equation}
    We omit writing $p^*$ explicitly. Tensoring \ref{eq:u-exact-seq} yields another exact sequence
    \[
        0 \to {P }_n(l) \to {V }(l) \to \SO(l+1) \to 0
    \]
    Let us consider the case $k=1$. Taking $\mathbf{R}p_*(-)$ provides an exact triangle
    \[
        \mathbf{R}p_*\big({P }_n(l)\big) \to \mathbf{R}p_*\big({V }(l)\big) \to \mathbf{R}p_*\big(\SO(l+1)\big) \to \mathbf{R}p_*\big({P }_n(l)\big)[1]
    \]
    which is isomorphic to
    \[
        \mathbf{R}p_* \big({P }_n(l)\big) \to {V }\otimes S^{l}{V } \to S^{l+1}{V } \to \mathbf{R}p_*\big({P }_n(l)\big)[1],
    \]
    for each $l \geq 0$ by the cohomology of projective bundles, cf. \cite[Chapter III Remark 2.1.16, p.443]{ega31}. The map ${V }\otimes S^{l}{V } \to S^{l+1}{V }$ is the same as the differential $d^{0}$ in \eqref{eq:abw}. Next, suppose that the claim is true for $k-1$. Note that \eqref{eq:u-exact-seq} induces a short exact sequence
    \begin{equation}\label{eqn:q_short_exact_seq}
        0 \to (\wedge^{k}{P }_n)(l) \to (\wedge^{k}{V })(l) \to (\wedge^{k-1}{P }_n ) (l+1) \to 0,
    \end{equation}
    which induces an exact triangle
    \[
        \mathbf{R}p_* \big((\wedge^{k}{P }_n)(l)\big) \to \wedge^{k}{V }\otimes S^{l}{V }[0] \to \mathbf{R}p_* \big((\wedge^{k-1}{P }_n)(l+1)\big) \to \mathbf{R}p_* \big((\wedge^{k}{P }_n)(l)\big)[1],
    \]
    and one verifies that the middle map corresponds to the differential $d^0$ in \eqref{eq:abw}.

    \noindent   \textbf{Step II}.   By Step I, we see that $\mathbf{R}^{i}p_* \big(\wedge^{k}{P }_n\big) = 0$ whenever $k > 0,i \geq 0$. If $k= 0$, $\mathbf{R}p_* \big(\SO_{\Gr(n,{V })}) = \SO_Y $. Let us prove the proposition by induction on $k$ and $l$ and assume that $k,l >0$. The short exact sequence \eqref{eqn:q_short_exact_seq} induces a long exact sequence
    \[
        \ldots \mathbf{R}^{i-1}p_*\big((\wedge^{k}{P }_n)(-l)\big) \to \mathbf{R}^{i-1}p_*\big((\wedge^{k}{V })(-l)\big) \to \mathbf{R}^{i-1}p_*\big((\wedge^{k-1}{P }_n )(-(l-1))\big) \to \mathbf{R}^{i}p_*\big((\wedge^{k}{P }_n)(-l)\big) \ldots
    \]
    It is clear that $\mathbf{R}^{i-1}p_*\wedge^{k}{V }(-l)=0$ and we have
    \[
        \mathbf{R}^{i}p_*\big((\wedge^{k}{P }_n)(-l)\big)  \cong \mathbf{R}^{i-1}p_*\big((\wedge^{k-1}{P }_n )(-(l-1))\big).
    \]
    The proposition then follows from induction.
\end{proof}

\subsection{On the local complete intersection}    Let $V$ be a vector bundle of rank $n+1$ over a scheme $Y$. Denote by ${P }_n$ the universal bundle over $p:\Gr(n,{V }) \to Y$. Suppose that $s: {P }_n \to p^*L$ is a regular section, where $L$ is a line bundle over $Y$. We shall drop the mention of $p^*$ in the sequel. Denote by $\iota: Z \hookrightarrow \Gr(n,{V })$ the zero locus of the section $s$, and $\pi$ the composition $Z\xrightarrow{\iota} \Gr(n,{V }) \xrightarrow{p} Y$. Note that we have
\[
    \omega_\pi = L^{\otimes n} (-n)
\]
over $Z(s)$.

\begin{lemma}\label{lma:push-pull-identiy-2}
    The canonical map $\SO_Y \to \pi_*\SO_{Z}$ is an isomorphism, and $\mathbf{R}\pi_* \SO_Z \cong \SO_Y$,
\end{lemma}
\begin{proof} The complex $\mathbf{R}\iota_* \SO_Z$ can be resolved by the Koszul complex
    \[
        \begin{aligned}
            K(s)	=  \big( 0 \to \wedge^{n}{P }_n\otimes L^{\otimes -n}\to \wedge^{n-1}{P }_n\otimes L^{\otimes -(n-1)}   \to \ldots \to {P }_n \otimes L^{\otimes -1} \to {\SO_{\Gr(n,{V })}} \to 0 \big)
        \end{aligned}
    \]
    where the term $\SO_{\Gr(n,{V })}$ sits on degree zero. By Proposition \ref{lma:gr_uni_wedge} and projection formula, we get
    \[
        \mathbf{R}p_*\big(\wedge^{i}{P }_n \otimes L^m\big) = 0,
    \]
    for any $m$ and $ 1\leq  i \leq n$. The result follows from computing $\mathbf{R}p_*(K(s))$.
\end{proof}

\begin{lemma} \label{cor:push-pull-identiy-3}
    Suppose that $n$ is even. Denote by $\sqrt{\omega_\pi}:= L^{\otimes -\frac{n}{2}} \otimes \SO(-\frac{n}{2})$. Then we have $\mathbf{R}\pi_* \sqrt{\omega_\pi} = \SO_Y$.
\end{lemma}
\begin{proof} The complex $\mathbf{R}\iota_* \sqrt{\omega_\pi}$ can be resolved by the Koszul complex $K(s) \otimes\sqrt{\omega_\pi} $:
    \[
        \begin{aligned}
            0 \to \wedge^{n}{P }_n \otimes L^{\otimes \frac{n}{2}-n}(-\frac{n}{2})\to \wedge^{n-1}{P }_n & \otimes L^{\otimes \frac{n}{2}-n+1}(-\frac{n}{2}) \to \ldots \to {P }_n \otimes L^{\otimes \frac{n}{2}-1}(-\frac{n}{2}) \to { L^{\otimes \frac{n}{2}}(-\frac{n}{2})} \to 0
        \end{aligned}
    \]
    By Proposition \ref{lma:gr_uni_wedge}, we have
    \[
        \mathbf{R}p_*\big(\wedge^{i}{P }_n \otimes L^{\otimes \frac{n}{2}-i}(-\frac{n}{2})\big) = 0,
    \]
    if $i\neq \frac{n}{2}$, and when $i=\frac{n}{2}$,
    \[
        \mathbf{R}p_*\big(\wedge^{i}{P }_n \otimes L^{\otimes \frac{n}{2}-i}(-\frac{n}{2})\big) =\mathbf{R}p_*\big(\wedge^{\frac{n}{2}}{P }_n(-\frac{n}{2})\big) = \SO_{Y}[-\frac{n}{2}].
    \]
    Then the result follows.
\end{proof}
Suppose that the scheme $Y$ is Gorenstein. Then so is the scheme $Z$.
\begin{theorem}\label{cor:push-pull-identiy}
    Suppose that $n$ is even. Then there exists a unit $\langle u \rangle \in \GW_0^{[0]}(Z)$ such that the diagram
    \[
        \xymatrix{
        \GW^{[m]}(Y) \ar[d]^-{\pi^*} \ar[rr]^-{\rho} && \GW^{[m]}(Y,E\SO_Y) & \GW^{[m]}(Z,\pi^{\Delta}E\SO_Y) \ar[l]_-{\pi_*}\\
        \GW^{[m]}(Z) \ar[r]^-{-\cup \langle u \rangle} & \GW^{[m]}(Z) \ar[r]^-{\cup \sqrt{\omega_\pi}} & \GW^{[m]}(Z,\omega_\pi) \ar[r]^-{\rho} & \GW^{[m]}(Z,E\omega_\pi) \ar[u]^-{\simeq}
        }
    \]
    commutes up to homotopy.
\end{theorem}
\begin{proof}
    By Lemma \ref{cor:push-pull-identiy-3}, we have that $\mathbf{R}\pi_* \sqrt{\omega_\pi} = \SO_Y$. Then the image of $1 \in \GW_0(Y)$ under the following map
    \[
        \pi_*\rho (\pi^*(-) \cup \sqrt{\omega_\pi}): \GW_0^{[0]}(Y) \to \GW_0^{[m]}(Y,E\SO_Y)
    \]
    has underlying complex $C^\bullet$ quasi-isomorphic to $\SO_Y$. Fix a quasi-isomorphim $\phi:\SO_Y \to C^\bullet$. Such symmetric form $\pi_*\rho (\pi^*(1) \cup \sqrt{\omega_\pi})$ is determined via the isomorphism
    \[
        \mathrm{Hom}_{D^{b}_c(X)}(C^\bullet, \sharp_{E \SO_Y}(C^\bullet)) \xrightarrow[\cong]{\phi^*} \mathrm{Hom}_{D^{b}_c(X)}(\SO_Y, \sharp_{E \SO_Y}(\SO_Y)) \xrightarrow[\cong]{\rho^{-1}} \mathrm{Hom}_{D^{b}_c(X)}(\SO_Y, \SO_Y) \xrightarrow{\cong} \Gamma(Y, \SO_Y)
    \]
    by an unit $\tilde{u}^{-1} \in \mathbb{G}(Y) \subset  \Gamma(Y, \SO_Y)$.
    Then $\pi_*\rho (\pi^*(\langle\tilde{u}\rangle) \cup \sqrt{\omega_\pi})$ is isometric to $\rho(1)$. Set $u := \pi^*(\tilde{u}) \in \mathbb{G}(Z)$. Then, we get
    \[
        \pi_*\rho ((\pi^*(-)\cup \langle u \rangle ) \cup \sqrt{\omega_\pi}) =  \pi_*\rho (\pi^*(-\cup \langle \tilde{u} \rangle )  \cup \sqrt{\omega_\pi})  =  \pi_*\rho (\pi^*(\langle \tilde{u} \rangle)  \cup \sqrt{\omega_\pi}) \cup - = \rho(1) \cup - = \rho.
    \]
    where the first equality obtained by the compatibility of the cup product and pullback and where the second and the third equality follows from applying the projection formula repetitively.
\end{proof}
\begin{remark}\label{rmk:alignment-remark}
    Theorem \ref{cor:push-pull-identiy} holds true for any integer $n$ if $\omega_\pi = \SO_Z$ with the same proof.
\end{remark}

\begin{lemma}\label{lma:proper_flat_regular_section}
    Suppose that $f:Y \to S$ is a proper flat morphism, and the section $s:{P }_n \to p^* L$ remains regular on each fiber $(f\circ p)^{-1}(x)$ of $x\in S$. Then the map $\pi \circ f:Z \to S$ is also flat.
\end{lemma}
\begin{proof}
    By the valuative criterion for flatness (cf. \cite[Theorem 14.34, p.440]{gortz2020algebraic}), it suffice to prove for the case when $S =  \mathrm{Spec}(R)$ for some discrete valuation ring $R$. Note that the map $p: \Gr(n,{V }) \to Y$ is also proper and flat, then so is the composition $f\circ p:\Gr(n,{V }) \to S$. Let $x:=\mathrm{Spec}(k) \to S$ be the closed point of $S$, where $k$ is the residue filed of $R$. Then $s$ remains a regular section on $(f\circ p)^{-1}(x)$, and by \cite[\href{https://stacks.math.columbia.edu/tag/0470}{Lemma 0470}]{stacks-project}, the map $\pi \circ f$ is flat.
\end{proof}

\begin{lemma}\label{lma:affine_bundle_unit}
    Let $\vartheta:V \to Y$ be an affine bundle and suppose that $Y$ is reduced, then the following map
    \[
        \Gamma(Y, \SO_Y^{\ast}) \to \Gamma(V, \vartheta_* \SO_V^{\ast})
    \]
    is an isomorphism.
\end{lemma}
\begin{proof}
    We can first assume that $Y$ is connected, and then $V$ is of constant rank $n$ for some integer $n$. By the definition of affine bundle, there exists an affine open cover $\{U_i \}_{i\in I}$ of $Y$ such that there are isomorphisms $ V|_{U_i} \xrightarrow{\cong} \mathbb{A}_{U_i}^n$ over $U_i$. We have a map of exact sequence
    \[
        \xymatrix{
            0 \ar[r] & \Gamma(Y, \SO_Y^{\ast}) \ar[r] \ar[d] & \prod_{i \in I} \Gamma(U_i, \SO_{Y}^{\ast}) \ar[r] \ar[d] & \prod_{i \in I}\prod_{j \in I} \Gamma(U_i\cap U_j, \SO_{Y}^{\ast}) \ar[d]\\
            0 \ar[r] & \Gamma(Y, \vartheta_* \SO_V^{\ast}) \ar[r] & \prod_{i \in I} \Gamma(U_i, \vartheta_* \SO_V^{\ast}) \ar[r] & \prod_{i \in I}\prod_{j \in I} \Gamma(U_i\cap U_j, \vartheta_* \SO_V^{\ast})
        }
    \]
    Since the last two vertical maps are isomorphism (cf. \cite[Exercise 1.2]{atiyah1969introduction}), by the five lemma, we obtain the final result.
\end{proof}

\bibliographystyle{halpha}
\bibliography{bib}

\end{document}